\documentclass[reqno,11pt,final]{amsart}
\usepackage{amsthm,amsfonts,graphicx,hyperref,tikz}
\usepackage[a4paper,margin=3cm]{geometry}
\numberwithin{equation}{section}
\theoremstyle{plain}
\newtheorem{theorem}{Theorem}[section]%
\newtheorem{corollary}[theorem]{Corollary}%
\newtheorem{proposition}[theorem]{Proposition}%
\newtheorem{lemma}[theorem]{Lemma}%

\newtheorem{definition}[theorem]{Definition}%
\newtheorem{example}[theorem]{Example}%
\newtheorem{remark}[theorem]{Remark}%
\newcommand{\dC}{\mathbb{C}}
\newcommand{\dE}{\mathbb{E}}
\newcommand{\dH}{\mathbb{H}}

\newcommand{\dN}{\mathbb{N}}
\newcommand{\dP}{\mathbb{P}}
\newcommand{\dR}{\mathbb{R}}
\newcommand{\dS}{\mathbb{S}}
\newcommand{\dU}{\mathbb{U}}

\newcommand{\dZ}{\mathbb{Z}}

\newcommand{\cA}{\mathcal{A}}
\newcommand{\cC}{\mathcal{C}}\newcommand{\cD}{\mathcal{D}}
\newcommand{\cE}{\mathcal{E}}\newcommand{\cF}{\mathcal{F}}
\newcommand{\cG}{\mathcal{G}}

\newcommand{\cM}{\mathcal{M}}\newcommand{\cN}{\mathcal{N}}
\newcommand{\cO}{\mathcal{O}}\newcommand{\cP}{\mathcal{P}}
\newcommand{\cQ}{\mathcal{Q}}
\newcommand{\cS}{\mathcal{S}}\newcommand{\cT}{\mathcal{T}}

\newcommand{\cX}{\mathcal{X}}

\newcommand{\al}{\alpha}
\newcommand{\be}{\beta}
\newcommand{\De}{\Delta}
\newcommand{\de}{\delta}
\newcommand{\ga}{\gamma}
\newcommand{\Ga}{\Gamma}
\newcommand{\la}{\lambda}
\newcommand{\La}{\Lambda}

\newcommand{\Om}{\Omega}
\newcommand{\om}{\omega}

\newcommand{\Si}{\Sigma}
\newcommand{\si}{\sigma}

\newcommand{\veps}{\varepsilon}
\newcommand{\vphi}{\varphi}

\newcommand{\FLOOR}[1]{{{\lfloor#1\rfloor}}} %
\newcommand{\CEIL}[1]{{{\lceil#1\rceil}}} %
\newcommand{\ABS}[1]{{{\left| #1 \right|}}} 
\newcommand{\BRA}[1]{{{\left\{#1\right\}}}} 
\newcommand{\DOT}[1]{{{\left<#1\right>}}} 
\newcommand{\ANG}[1]{{{\langle#1\rangle}}} 
\newcommand{\NRM}[1]{{{\left\| #1\right\|}}} 
\newcommand{\PAR}[1]{{{\left(#1\right)}}} 
\newcommand{\pd}{{\partial}} 
\newcommand{\SBRA}[1]{{{\left[#1\right]}}} 
\newcommand{\RANK}{\mathrm{rank}}
\newcommand{\DIM}{\mathrm{dim}}
\newcommand{\SUPP}{\mathrm{supp}}
\newcommand{\DIST}{\mathrm{dist}}
\newcommand{\SPAN}{\mathrm{span}}
\newcommand{\COMP}{\mathrm{Comp}}
\newcommand{\INCOMP}{\mathrm{Incomp}}
\newcommand{\SPARSE}{\mathrm{Sparse}}
\newcommand{\TR}{\mathrm{Tr}}
\newcommand{\VAR}{\mathrm{Var}}
\newcommand{\IND}{\mathbf{1}}

\newcommand{\OL}[1]{\overline{#1}}
\newcommand{\Ol}{\overline}

\newcommand{\Wt}{\widetilde}
\newcommand{\CARD}{\mathrm{card}}
\newcommand{\SSK}[1]{\substack{#1}}
\newcommand{\so}{\text{\o}} 
\newcommand{\weak}{\rightsquigarrow}
\renewcommand{\Im}{\mathfrak{Im}}
\renewcommand{\Re}{\mathfrak{Re}}
\newcommand{\pwit}{\mathrm{PWIT}}
\title{Around the circular law} \date{Update of
  \href{http://dx.doi.org/10.1214/11-PS183}{Probability Surveys 9 (2012)
    1--89}. Added reference~\cite{KS}. Compiled \today} %
\author{Charles Bordenave} \email{charles.bordenave(at)math.univ-toulouse.fr}
\urladdr{http://www.math.univ-toulouse.fr/~bordenave/} \address{Universit\'e
  de Toulouse, Institut de Math\'ematiques de Toulouse, UMR CNRS 5219, France}
\author{Djalil Chafa\"\i}
\email{djalil(at)chafai.net}
\urladdr{http://djalil.chafai.net/}
\address{Universit\'e Paris-Est Marne-la-Vall\'ee, UMR CNRS 8050, France}
\begin{document}
\begin{abstract}
  These expository notes are centered around the circular law theorem, which
  states that the empirical spectral distribution of a $n\times n$ random
  matrix with i.i.d.\ entries of variance $1/n$ tends to the uniform law on
  the unit disc of the complex plane as the dimension $n$ tends to infinity.
  This phenomenon is the non-Hermitian counterpart of the semi circular limit
  for Wigner random Hermitian matrices, and the quarter circular limit for
  Marchenko-Pastur random covariance matrices. We present a proof in a
  Gaussian case, due to Silverstein, based on a formula by Ginibre, and a
  proof of the universal case by revisiting the approach of Tao and Vu, based
  on the Hermitization of Girko, the logarithmic potential, and the control of
  the small singular values. Beyond the finite variance model, we also
  consider the case where the entries have heavy tails, by using the objective
  method of Aldous and Steele borrowed from randomized combinatorial
  optimization. The limiting law is then no longer the circular law and is
  related to the Poisson weighted infinite tree. We provide a weak control of
  the smallest singular value under weak assumptions, using asymptotic
  geometric analysis tools. We also develop a quaternionic Cauchy-Stieltjes
  transform borrowed from the Physics literature.
\end{abstract}
\maketitle
{\footnotesize\tableofcontents}

\newpage

These expository notes are split in seven sections and an appendix. Section
\ref{se:spectra} introduces the notion of eigenvalues and singular values and
discusses their relationships. Section \ref{se:circular} states the circular
law theorem. Section \ref{se:gaussian} is devoted to the Gaussian model known
as the Complex Ginibre Ensemble, for which the law of the spectrum is known
and leads to the circular law. Section \ref{se:universal} provides the proof
of the circular law theorem in the universal case, using the approach of Tao
and Vu based on the Hermitization of Girko and the logarithmic potential.
Section \ref{se:related} presents some models related to the circular law and
discusses an algebraic-analytic interpretation via free probability. Section
\ref{se:heavy} is devoted to the heavy tailed counterpart of the circular law
theorem, using the objective method of Aldous and Steele and the Poisson
Weighted Infinite Tree. Finally, section \ref{se:open} lists some open
problems. The notes end with appendix \ref{se:ap:inv} devoted to a novel
general weak control of the smallest singular value of random matrices with
i.i.d.\ entries, with weak assumptions, well suited for the proof of the
circular law theorem and its heavy tailed analogue.

\begin{table}
  \footnotesize
  \begin{center}
    \begin{tabular}[c]{|r|l|} \hline
     $\log$
      & natural Neperian logarithm function (we never use the notation $\ln$)\\
      $a:=\cdots$
      & the mathematical object $a$ is defined by the formula $\cdots$ \\
      $\varliminf$ and $\varlimsup$ 
      & inferior and superior limit \\
      $\dN$
      & set of non-negative integers $1,2,\ldots$ \\
      $\dR_+$ 
      & set of non-negative real numbers $[0,\infty)$ \\
      $\dC_+$
      & set of complex numbers with positive imaginary part \\
      $\dH_+$ 
      & 
      set of $2\times 2$ matrices of the form
      {\footnotesize $\begin{pmatrix} \eta & z \\ \bar z & \eta \end{pmatrix}$} with
      $z \in \dC$ and  $\eta \in \dC_+$ \\
      $i$ 
      & complex number $(0,1)$ or some natural integer (context dependent) \\
      $\cM_n(K)$
      & set of $n\times n$ matrices with entries in $K$ \\
      $\cM_{n,m}(K)$ 
      & set of $n\times m$ matrices with entries in $K$ \\
      $A^\top$
      & transpose matrix of matrix $A$ (we never use the notation
      $A'$ or $ ^\mathrm{t}\!A$) \\
      $\bar{A}$
      & conjugate matrix of matrix $A$ or closure of set $A$ \\
      $A^{-1}$ and $A^*$ 
      & inverse and conjugate-transpose of $A$ \\
      $\TR(A)$ and $\det(A)$
      & trace and determinant of $A$ \\
      $I$ (resp.\ $I_n$)
      & identity matrix (resp.\ of dimension $n$) \\
      $A-z$ 
      & the matrix $A-zI$ (here $z\in\dC$) \\
      $s_k(A)$ 
      & $k$-th singular value of $A$ (descending order) \\
      $\la_k(A)$ 
      & $k$-th eigenvalue of $A$ (decreasing module and growing
      phase order) \\
      $P_A(z)$
      & characteristic polynomial of $A$ at point $z$, namely $\det(A-zI)$ \\
      $\mu_A$
      & empirical measure built from the eigenvalues of $A$ \\
      $\nu_A$
      & empirical measure built from the singular values of $A$ \\
      $U_\mu(z)$
      & logarithmic potential of $\mu$ at point $z$ \\
      $m_\mu(z)$
      & Cauchy-Stieltjes transform of $\mu$ at point $z$ \\
      $M_\mu(q)$
      & quaternionic transform of $\mu$ at point $q\in\dH_+$ \\
      $\Ga_A(q)$
      & quaternionic transform of $\mu_A$ at point $q\in\dH_+$ (i.e.\ $M_{\mu_A}(q)$)\\
      $\SPAN(\cdots)$
      & vector space spanned by the arguments $\cdots$ \\
      $\DOT{\cdot,\cdot}$
      & Scalar product in $\dR^n$ or in $\dC^n$ \\
      $\DIST(v,V)$
      & $2$-norm distance of vector $v$ to vector space $V$ \\
      $V^\perp$
      & orthogonal vector space of the vector space $V$ \\
      $\SUPP$
      & support (for measures, functions, and vectors) \\
      $\ABS{z}$ and $\CARD(E)$
      & module of $z$ and cardinal of $E$ \\
      $\NRM{v}_2$
      & $2$-norm of vector $v$ in $\dR^n$ or in $\dC^n$\\
      $\NRM{A}_{2\to2}$
      & operator norm of matrix $A$ for the $2$-norm (i.e.\ spectral norm) \\
      $\NRM{A}_2$
      & Hilbert-Schmidt norm of matrix $A$ (i.e.\ Schur or Frobenius norm) \\
      $o(\cdot)$ and $O(\cdot)$ 
      & classical Landau notations for asymptotic behavior \\
      $D$ 
      & most of the time, diagonal matrix \\
      $U,V,W$
      & most of the time, unitary matrices \\
      $H$ 
      & most of the time, Hermitian matrix \\
      $X$
      & most of the time, random matrix with i.i.d.\ entries \\
      $G$ 
      & most of the time, complex Ginibre Gaussian random matrix \\
      $\IND_E$ 
      & indicator of set $E$ \\
      $\pd$, $\OL{\pd}$, $\Delta$ 
      & differential operators $\frac{1}{2}(\pd_x-i\pd_y)$, 
      $\frac{1}{2}(\pd_x+i\pd_y)$, $\pd_x^2+\pd_y^2$  on $\dR^2$ \\
      $\cP(\dC)$
      & set of probability measures on $\dC$ integrating $\log\ABS{\cdot}$ at
      infinity \\
      $\cD'(\dC)$
      & set of Schwartz-Sobolev distributions on $\dC=\dR^2$ \\
      $C,c,c_0,c_1$
      & most of the time, positive constants (sometimes absolute) \\
      $X\sim \mu$
      & the random variable $X$ follows the law $\mu$ \\
      $\mu_n\weak\mu$
      & the sequence ${(\mu_n)}_{n}$ tends weakly to $\mu$ for
      continuous bounded functions \\
      $\cN(m,\Si)$
      & Gaussian law with mean vector $m$ and covariance matrix $\Si$ \\
      $\cQ_{\kappa}$ and $\cC_\kappa$
      & quarter circular law on $[0,\kappa]$ and circular law on
      $\{z\in\dC:|z|\leq\kappa\}$ \\ \hline 
    \end{tabular}
    \smallskip
    \caption{Main frequently used notations.}
    \label{tb:notations}
  \end{center}
\end{table}

All random variables are defined on a unique common probability space
$(\Om,\cA,\dP)$. A typical element of $\Om$ is denoted $\om$. Table
\ref{tb:notations} gathers most frequently used notations.

\section{Two kinds of spectra}
\label{se:spectra}

The \emph{eigenvalues} of $A\in\cM_n(\dC)$ are the roots in $\dC$ of its
characteristic polynomial $P_A(z):=\det(A-zI)$. We label them
$\lambda_1(A),\ldots,\lambda_n(A)$ so that
$|\lambda_1(A)|\geq\cdots\geq|\lambda_n(A)|$ with growing phases. The
\emph{spectral radius} is $|\lambda_1(A)|$. The eigenvalues form the algebraic
spectrum of $A$. The \emph{singular values of $A$} are defined by
\[
s_k(A):=\lambda_k(\sqrt{AA^*}) 
\]
for all $1\leq k\leq n$, where $A^*=\bar{A}^\top$ is the conjugate-transpose.
We have 
\[
s_1(A)\geq\cdots\geq s_n(A)\geq0.
\]
The matrices $A,A^\top,A^*$ have the same singular values. The
$2n\times 2n$ Hermitian matrix 
\[
H_A:=\begin{pmatrix} 0 & A \\ A^* & 0\end{pmatrix}
\]
has eigenvalues $s_1(A),-s_1(A),\ldots,s_n(A),-s_n(A)$. This turns out to be
useful because the mapping $A\mapsto H_A$ is linear in $A$, in contrast with
the mapping $A\mapsto\sqrt{AA^*}$. Geometrically, the matrix $A$ maps the unit
sphere to an ellipsoid, the half-lengths of its principal axes being exactly
the singular values of $A$. The \emph{operator norm} or \emph{spectral norm}
of $A$ is
\[
\Vert A\Vert_{2\to2}:=\max_{\Vert x\Vert_2=1}\Vert Ax\Vert_2=s_1(A)
\quad\text{while}\quad
s_n(A)=\min_{\NRM{x}_2=1}\NRM{Ax}_2.
\]
The rank of $A$ is equal to the number of non-zero singular values. If $A$ is
non-singular then $s_i(A^{-1})=s_{n-i}(A)^{-1}$ for all $1\leq i\leq n$ and
$s_n(A)=s_1(A^{-1})^{-1}=\Vert A^{-1}\Vert_{2\to 2}^{-1}$.

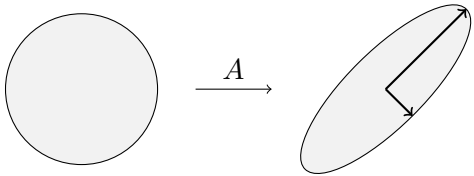
\begin{figure}[htbp]
  \begin{center}
    \begin{tikzpicture}
      \draw[fill=gray!10] (-2,0) circle (1cm);
      \draw[fill=gray!10,shift={(2,0)},rotate=-45] (0,0) ellipse (0.5cm and
      1.5cm);
      \draw[->] (-0.5,0) -- (0.5,0);
      \node[above] {$A$};
      \draw[style=thick,shift={(2,0)},rotate=-45,->] (0,0) -- (0.5,0);
      \draw[style=thick,shift={(2,0)},rotate=-45,->] (0,0) -- (0,1.5);
    \end{tikzpicture}
    \caption{\footnotesize Largest and smallest singular values of $A\in\cM_2(\dR)$. Taken
      from \cite{lamabook}.}
    \label{fi:sing}
  \end{center}
\end{figure}

Since the singular values are the eigenvalues of a Hermitian matrix, we have
variational formulas for all of them, often called the Courant-Fischer
variational formulas \cite[Th.\ 3.1.2]{MR1288752}. Namely, denoting
$\cG_{n,i}$ the Grassmannian of all $i$-dimensional subspaces, we have
\[
s_i(A)=\max_{E\in\cG_{n,i}}\min_{\SSK{x\in E\\\NRM{x}_2=1}}\NRM{Ax}_2
=\max_{E,F\in\cG_{n,i}}\min_{\SSK{(x,y)\in E\times F\\\NRM{x}_2=\NRM{y}_2=1}}\DOT{Ax,y}.
\]
Most useful properties of the singular values are consequences of their
Hermitian nature via these variational formulas, which are valid on $\dR^n$
and on $\dC^n$. In contrast, there are no such variational formulas for the
eigenvalues in great generality, beyond the case of normal matrices. If the
matrix $A$ is normal\footnote{In these notes, the word \emph{normal} is always
  used in this way, and never as a synonym for \emph{Gaussian}.} (i.e.\
$A^*A=A^*A$) then for every $1\leq i\leq n$,
\[
s_i(A)=|\lambda_i(A)|.
\]
Beyond normal matrices, the relationships between the eigenvalues and the
singular values are captured by a set of inequalities due to Weyl
\cite{MR0030693}\footnote{Horn \cite{MR0061573} showed a remarkable converse
  to Weyl's theorem: if a sequence $s_1\geq\cdots\geq s_n$ of non-negative
  real numbers and a sequence $\la_1,\ldots,\la_n$ of complex numbers of non
  increasing modulus satisfy to all Weyl's inequalities \eqref{eq:weyl0} then
  there exists $A\in\cM_{n}(\dC)$ with eigenvalues $\la_1,\ldots,\la_n$ and
  singular values $s_1,\ldots,s_n$.}, which can be obtained by using the Schur
unitary triangularization\footnote{If $A\in\cM_n(\dC)$ then there exists a
  unitary matrix $U$ such that $UAU^*$ is upper triangular.}, see for instance
\cite[Theorem 3.3.2 page 171]{MR1288752}.

\begin{theorem}[Weyl inequalities]\label{th:weyl}
  For every $n\times n$ complex matrix $A$ and $1\leq k\leq n$, 
  \begin{equation}\label{eq:weyl0}
    \prod_{i=1}^k|\lambda_i(A)|\leq \prod_{i=1}^ks_i(A).
  \end{equation}
\end{theorem}

The reversed form $\prod_{i=n-k+1}^ns_i(A) \leq
\prod_{i=n-k+1}^n|\lambda_i(A)|$ for every $1\leq k\leq n$ can be deduced
without much difficulty (exercise!). Equality is achieved for $k=n$ and we
have
\begin{equation}\label{eq:weyl1}
  \prod_{k=1}^n|\lambda_k(A)|
  =|\det(A)|
  =\sqrt{|\det(A)||\det(A^*)|}
  =|\det(\sqrt{AA^*})|
  =\prod_{k=1}^ns_k(A).
\end{equation}
By using \emph{majorization techniques} which can be found for instance in
\cite[Section 3.3]{MR1288752}, one may deduce from Weyl's inequalities that
for every real valued function $\varphi$ such that $t\mapsto\varphi(e^t)$ is
increasing and convex on $[s_n(A),s_1(A)]$, we have, for every $1\leq k\leq
n$,
\begin{equation}\label{eq:weyl2}
  \sum_{i=1}^k\varphi(|\lambda_i(A)|) \leq \sum_{i=1}^k\varphi(s_i(A)),
\end{equation}
see \cite[Theorem 3.3.13]{MR1288752}. In particular, taking $k=n$ and
$\varphi(t)=t^2$ gives
\begin{equation}\label{eq:weyl3}
  \sum_{i=1}^n|\lambda_i(A)|^2 \leq \sum_{i=1}^ns_i(A)^2
  =\TR(AA^*)=\sum_{i,j=1}^n|A_{i,j}|^2.
\end{equation}
Since $s_1(\cdot)=\NRM{\cdot}_{2\to2}$ we have for any $A,B\in\cM_{n}(\dC)$
that
\begin{equation}\label{eq:basic0}
  s_1(AB)\leq s_1(A)s_1(B)
  \quad\text{and}\quad
  s_1(A+B)\leq s_1(A)+s_1(B).
\end{equation}

We define the empirical eigenvalues and singular values measures by 
\[
\mu_A:=\frac{1}{n}\sum_{k=1}^n\delta_{\lambda_k(A)}
\quad\text{and}\quad
\nu_A:=\frac{1}{n}\sum_{k=1}^n\delta_{s_k(A)}.
\]
Note that $\mu_A$ and $\nu_A$ are supported respectively in $\dC$ and $\dR_+$.
There is a rigid determinantal relationship between $\mu_A$ and $\nu_A$,
namely from \eqref{eq:weyl1} we get
\begin{align*}
  \int\!\log|\la|\,d\mu_A(\la)
  &=\frac{1}{n}\sum_{i=1}^n\log|\la_i(A)| \\
  &=\frac{1}{n}\log|\det(A)| \\
  &=\frac{1}{n}\sum_{i=1}^n\log(s_i(A)) \\
  &=\int\!\log(s)\,d\nu_A(s).
\end{align*}
This identity is at the heart of the Hermitization technique used in sections
\ref{se:universal} and \ref{se:heavy}.

The singular values are quite regular functions of the matrix entries. For
instance, the Courant-Fischer formulas imply that the mapping $A\mapsto
(s_1(A),\ldots,s_n(A))$ is $1$-Lipschitz for the operator norm and the
$\ell^\infty$ norm in the sense that for any $A,B\in\cM_{n}(\dC)$,
\begin{equation}\label{eq:basic1}
  \max_{1\leq i\leq n}|s_i(A)-s_i(B)|\leq s_1(A-B).
\end{equation}
Recall that $\cM_n(\dC)$ or $\cM_n(\dR)$ are Hilbert spaces for the scalar
product $A\cdot B=\TR(AB^*)$. The norm $\NRM{\cdot}_2$ associated to this
scalar product, called the trace norm\footnote{Also known as the
  Hilbert-Schmidt norm, the Schur norm, or the Frobenius norm.}, satisfies to
\begin{equation}\label{eq:trace-norm}
  \NRM{A}_2^2=\TR(AA^*)=\sum_{i=1}^ns_i(A)^2=n\int\!s^2\,d\nu_A(s).
\end{equation}
The Hoffman-Wielandt inequality for the singular values states that for all
$A,B\in\cM_n(\dC)$,
\begin{equation}\label{eq:Hoffman-Wielandt}
  \sum_{i=1}^n(s_i(A)-s_i(B))^2 \leq \NRM{A-B}_2^2.
\end{equation}
In other words the mapping $A\mapsto (s_1(A),\ldots,s_n(A))$ is $1$-Lipschitz
for the trace norm and the Euclidean $2$-norm. See \cite[equation
(3.3.32)]{MR1288752} and \cite[Theorem 6.3.5]{MR1084815} for a proof.

In the sequel, we say that a sequence of (possibly signed) measures
${(\eta_n)}_{n\geq1}$ on $\dC$ (respectively on $\dR$) tends weakly to a
(possibly signed) measure $\eta$, and we denote
\[
\eta_n\weak\eta,
\]
when for all continuous and bounded function $f:\dC\to\dR$ (respectively
$f:\dR\to\dR$),
\[
\lim_{n\to\infty}\int\!f\,d\eta_n=\int\!f\,d\eta.
\]
This type of convergence does not capture the behavior of the support and of
the moments\footnote{Note that for empirical spectral distributions in random
  matrix theory, most of the time the limit is characterized by its moments,
  and this allows to deduce weak convergence from moments convergence.}.

\begin{example}[Spectra of non-normal matrices]\label{ex:nilpot}
  The eigenvalues depend continuously on the entries of the matrix. It turns
  out that for non-normal matrices, the eigenvalues are more sensitive to
  perturbations than the singular values. Among non-normal matrices, we find
  non-diagonalizable matrices, including nilpotent matrices. Let us recall a
  striking example taken from \cite{MR1929504} and \cite[Chapter
  10]{bai-silverstein-book}. Let us consider $A,B\in\cM_n(\dR)$ given by
  \[
  A=
  \begin{pmatrix}
    0 & 1 & 0 & \cdots & 0 \\
    0 & 0 & 1 & \cdots & 0 \\
    \vdots & \vdots & \vdots & & \vdots \\
    0 & 0 & 0 & \cdots & 1 \\
    0 & 0 & 0 & \cdots & 0
  \end{pmatrix}
  \quad\text{and}\quad
  B=
  \begin{pmatrix}
    0 & 1 & 0 & \cdots & 0 \\
    0 & 0 & 1 & \cdots & 0 \\
    \vdots  & \vdots & \vdots & & \vdots \\
    0 & 0 & 0 & \cdots & 1 \\
    \kappa_n & 0 & 0 & \cdots & 0
  \end{pmatrix}
  \]
  where ${(\kappa_n)}$ is a sequence of positive real numbers. The matrix $A$
  is nilpotent, and $B$ is a perturbation with small norm (and rank one!):
  \[
  \RANK(A-B)=1 \quad\text{and}\quad \NRM{A-B}_{2\to2}=\kappa_n.
  \]
  We have $\la_1(A)=\cdots=\la_{\kappa_n}(A)=0$ and thus 
  \[
  \mu_{A}=\de_0.
  \]
  In contrast, $B^n=\kappa_nI$ and thus $\la_k(B)=\kappa_n^{1/n}e^{2 k\pi
    i/n}$ for all $1\leq k\leq n$ which gives
  \[
  \mu_{B}\weak\mathrm{Uniform}\{z\in\dC:|z|=1\}
  \]
  as soon as $\kappa_n^{1/n}\to1$ (this allows $\kappa_n\to0$). On the other
  hand, from the identities
  \[
  AA^*=\mathrm{diag}(1,\ldots,1,0)
  \quad\text{and}\quad
  BB^*=\mathrm{diag}(1,\ldots,1,\kappa_n^2)
  \]
  we get $s_1(A)=\cdots=s_{n-1}(A)=1,s_n(A)=0$ and
  $s_1(B)=\cdots=s_{n-1}(B)=1,s_n(B)=\kappa_n$ for large enough $n$, and
  therefore, for any choice of $\kappa_n$, since the atom $\kappa_n$ has
  weight $1/n$,
  \[
  \nu_A\weak\de_1\quad\text{and}\quad\nu_B\weak\de_1.
  \]
  This example shows the stability of the limiting distribution of singular
  values under an additive perturbation of rank $1$ of arbitrary large norm,
  and the instability of the limiting eigenvalues distribution under an
  additive perturbation of rank $1$ of arbitrary small norm.
\end{example}

Beyond square matrices, one may define the singular values
$s_1(A),\ldots,s_{m}(A)$ of a rectangular matrix $A\in\cM_{m,n}(\dC)$ with
$m\leq n$ by $s_i(A):=\lambda_i(\sqrt{AA^*})$ for every $1\leq i\leq m$. The
famous Singular Value Decomposition (SVD, see \cite[Th.\ 3.3.1]{MR1288752})
states then that
\[
A=UDV^*
\]
where $U$ and $V$ are the unitary matrices of the eigenvectors of $AA^*$ and
$A^*A$ and where $D=\mathrm{diag}(s_1(A),\ldots,s_n(A))$ is a $m\times n$
diagonal matrix. The SVD is at the heart of many numerical techniques in
concrete applied mathematics (pseudo-inversion, regularization, low
dimensional approximation, principal component analysis, etc). Note that if
$A$ is square then the Hermitian matrix $H:=VDV^*$ and the unitary matrix
$W:=UV^*$ form the polar decomposition $A=WH$ of $A$. Note also that if $W_1$
and $W_2$ are unitary matrices then $W_1AW_2$ and $A$ have the same singular
values.

We refer to the books \cite{MR1288752} and \cite{MR1417720} for more details
on basic properties of the singular values and eigenvalues of deterministic
matrices. The sensitivity of the spectrum to perturbations of small norm is
captured by the notion of \emph{pseudo-spectrum}. Namely, for a matrix norm
$\NRM{\cdot}$ and a positive real $\veps$, the
$(\NRM{\cdot},\veps)$-pseudo-spectrum of $A$ is defined by
\[ \La_{\NRM{\cdot},\veps}(A):=\bigcup_{\NRM{A-B}\leq
  \veps}\BRA{\la_1(B),\ldots,\la_n(B)}.
\] 
If $A$ is normal then its pseudo-spectrum for the operator norm
$\NRM{\cdot}_{2\to2}$ coincides with the $\veps$-neighborhood of its spectrum.
The
pseudo-spectrum can be much larger for non-normal matrices. For instance, if
$A$ is the nilpotent matrix considered earlier, then the asymptotic (as
$n\to\infty$) pseudo-spectrum for the operator norm contains the unit disc if
$\kappa_n$ is well chosen. For more, see the book \cite{MR2155029}.

\section{Quarter circular and circular laws}
\label{se:circular}

The variance of a random variable $Z$ on $\dC$ is
$\VAR(Z)=\dE(|Z|^2)-|\dE(Z)|^2$. Let ${(X_{ij})}_{i,j\geq1}$ be an
infinite table of i.i.d.\ random variables on $\dC$ with variance $1$. We
consider the square random matrix $X:=(X_{ij})_{1\leq i,j\leq n}$ as a random
variable in $\cM_n(\dC)$. We write \emph{a.s.}, \emph{a.a.}, and \emph{a.e.}
for \emph{almost surely}, \emph{Lebesgue almost all}, and \emph{Lebesgue
  almost everywhere} respectively.

We start with a reformulation in terms of singular values of the classical
Marchenko-Pastur theorem for the ``empirical covariance matrix''
$\frac{1}{n}XX^*$. This theorem is universal in the sense that the limiting
distribution does not depend on the law of $X_{11}$.

\begin{theorem}[Marchenko-Pastur quarter circular
  law]\label{th:quarter-circular}
  a.s.\ $\nu_{n^{-1/2}X}\weak\cQ_2$ as $n\to\infty$, where $\cQ_2$ is the
  quarter circular law\footnote{Actually, it is a quarter ellipse rather than
    a quarter circle, due to the normalizing factor $1/\pi$. However, one may
    use different scales on the horizontal and vertical axes to see a true
    quarter circle, as in figure \ref{fi:qcl}.} on $[0,2]\subset\dR_+$ with
  density $x\mapsto\pi^{-1}\sqrt{4-x^2}\IND_{[0,2]}(x)$.
\end{theorem}

The $n^{-1/2}$ normalization factor is easily understood from the law of large
numbers:
\begin{equation}\label{eq:varbound}
  \int\!s^2\,d\nu_{n^{-1/2}X}(s)
  =\frac{1}{n^2}\sum_{i=1}^ns_i(X)^2
  =\frac{1}{n^2}\TR(XX^*)
  =\frac{1}{n^2}\sum_{i,j=1}^n|X_{i,j}|^2\to\dE(|X_{1,1}|^2).
\end{equation}

The central subject of these notes is the following counterpart for the
eigenvalues.

\begin{theorem}[Girko circular law]\label{th:circular}
  a.s.\ $\mu_{n^{-1/2}X}\weak\cC_1$ as $n\to\infty$, where $\cC_1$ is the
  circular law\footnote{It is not customary to call it instead the ``disc
    law''. The terminology corresponds to what we actually draw: a circle for
    the circular law, a quarter circle (actually a quarter ellipse) for the
    quarter circular law, even if it is the boundary of the support in the
    first case, and the density in the second case. See figure \ref{fi:qcl}.}
  which is the uniform law on the unit disc of $\dC$ with density
  $z\mapsto\pi^{-1}\IND_{\{z\in\dC:|z|\leq 1\}}$.
\end{theorem}

Note that if $Z$ is a complex random variable following the uniform law on the
unit disc $\{z\in\dC:|z|\leq1\}$ then the random variables $\Re(Z)$ and
$\Im(Z)$ follow the semi circular law on $[-1,1]$, but are not independent.
Additionally, the random variables $|\Re(Z)|$ and $|\Im(Z)|$ follow the
quarter circular law on $[0,1]$, and $|Z|$ follows the law with density
$\rho\mapsto\frac{1}{2}\rho\IND_{[0,1]}(\rho)$. We will see in section
\ref{se:related} that the notion of freeness developed in free probability is
the key to understand these relationships. An extension of theorem
\ref{th:quarter-circular} is the key to deduce theorem \ref{th:circular} via a
Hermitization technique, as we will see in section \ref{se:universal}.

The circular law theorem \ref{th:circular} has a long history. It was
established through a sequence of partial results during the period
1965--2009, the general case being finally obtained by Tao and Vu
\cite{tao-vu-cirlaw-bis}. Indeed Mehta \cite{MR0220494} was the first to
obtain a circular law theorem for the expected empirical spectral distribution
in the complex Gaussian case, by using the explicit formula for the spectrum
due to Ginibre \cite{MR0173726}. Edelman was able to prove the same kind of
result for the far more delicate real Gaussian case \cite{MR1437734}.
Silverstein provided an argument to pass from the expected to the almost sure
convergence in the complex Gaussian case \cite{MR841088}. Girko worked on the
universal version and came with very good ideas such as the Hermitization
technique \cite{MR773436,MR1310560,MR2046403,MR2085255,MR2130247}.
Unfortunately, his work was controversial due to a lack of clarity and
rigor\footnote{Girko's writing style is also quite original, see for instance
  the recent paper \cite{MR2747827}.}. In particular, his approach relies
implicitly on an unproved uniform integrability related to the behavior of the
smallest singular values of random matrices. Let us mention that the
Hermitization technique is also present in the work of Widom \cite{MR1300209}
on Toeplitz matrices and in the work of Goldsheid and Khoruzhenko
\cite{MR1800072}. Bai \cite{MR1428519} was the first to circumvent the problem
in the approach of Girko, at the price of bounded density assumptions and
moments assumptions\footnote{\emph{\ldots I worked for 13 years from 1984 to
    1997, which was eventually published in Annals of Probability. It was the
    hardest problem I have ever worked on.} Zhidong Bai, interview with Atanu
  Biswas in 2006 \cite{MR2407790}.}. Bai improved his approach in his book
written with Silverstein \cite{bai-silverstein-book}. His approach involves
the control of the speed of convergence of the singular values distribution.
\'Sniady considered a universal version beyond random matrices and the
circular law, using the notion of $*$-moments and Brown measure of operators
in free probability, and a regularization by adding an independent Gaussian
Ginibre noise \cite{MR1929504}. Goldsheid and Khoruzhenko \cite{MR2191234}
used successfully the logarithmic potential to derive the analogue of the
circular law theorem for random non-Hermitian tridiagonal matrices. The
smallest singular value of random matrices was the subject of an impressive
activity culminating with the works of Tao and Vu \cite{tao-vu-singular} and
of Rudelson and Vershynin \cite{MR2407948}, using tools from asymptotic
geometric analysis and additive combinatorics (Littlewood-Offord problems).
These achievements allowed G{\"o}tze and Tikhomirov
\cite{gotze-tikhomirov-new} to obtain the expected circular law theorem up to
a small loss in the moment assumption, by using the logarithmic potential.
Similar ingredients are present in the work of Pan and Zhou \cite{1687963}. At
the same time, Tao and Vu, using a refined bound on the smallest singular
value and the approach of Bai, deduced the circular law theorem up to a small
loss in the moment assumption \cite{MR2409368}. As in the works of Girko, Bai
and their followers, the loss was due to a sub-optimal usage of the
Hermitization approach. In \cite{tao-vu-cirlaw-bis}, Tao and Vu finally
obtained the full circular law theorem \ref{th:circular} by using the full
strength of the logarithmic potential, and a new control of the count of the
small singular values which replaces the speed of convergence estimates of
Bai. See also their synthetic paper \cite{MR2507275}. We will follow
essentially their approach in section \ref{se:universal} to prove theorem
\ref{th:circular}.

\begin{figure}[htbp]
  \begin{center}
    \includegraphics[width=31em,height=15em]{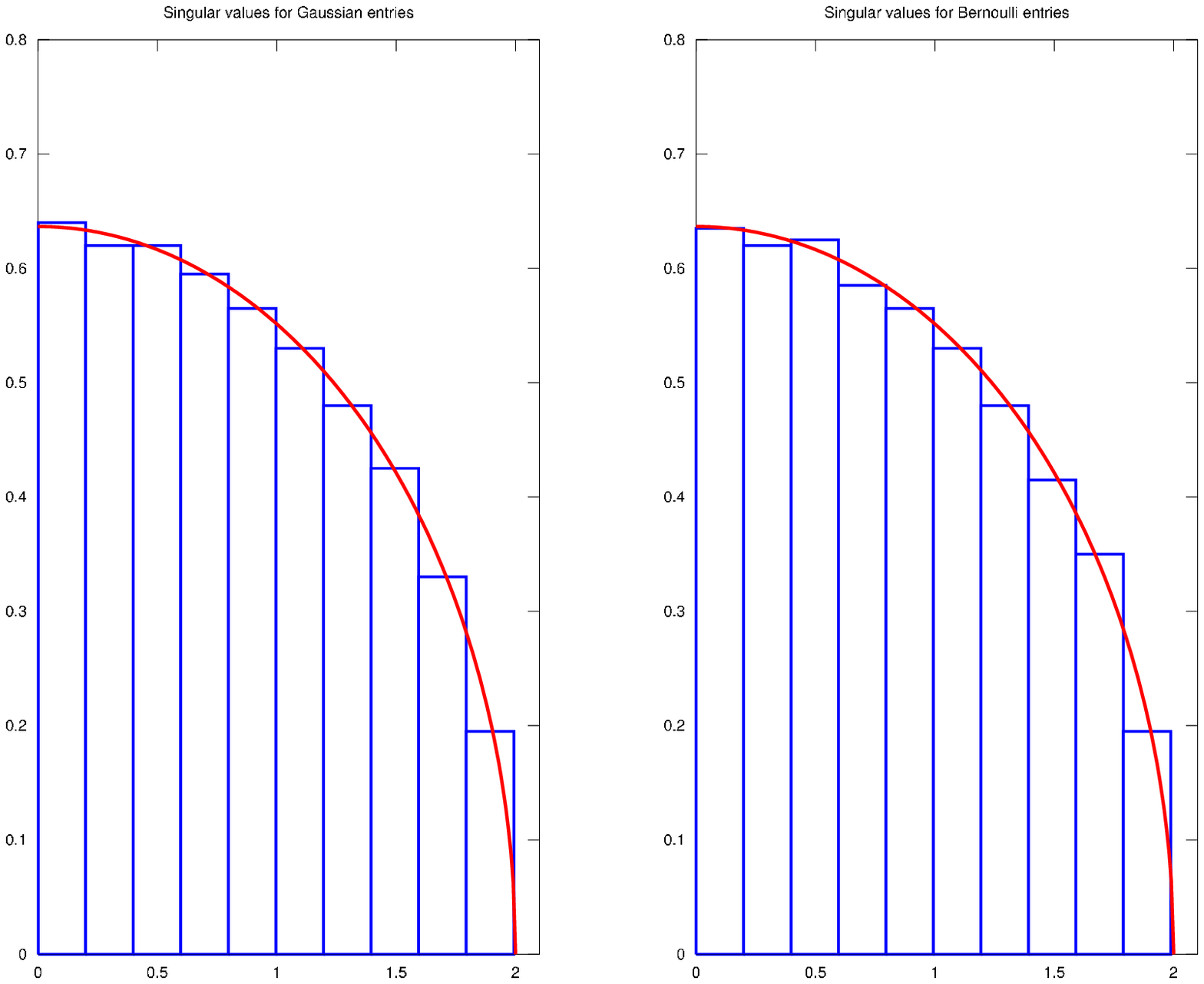}
    \includegraphics[scale=0.75]{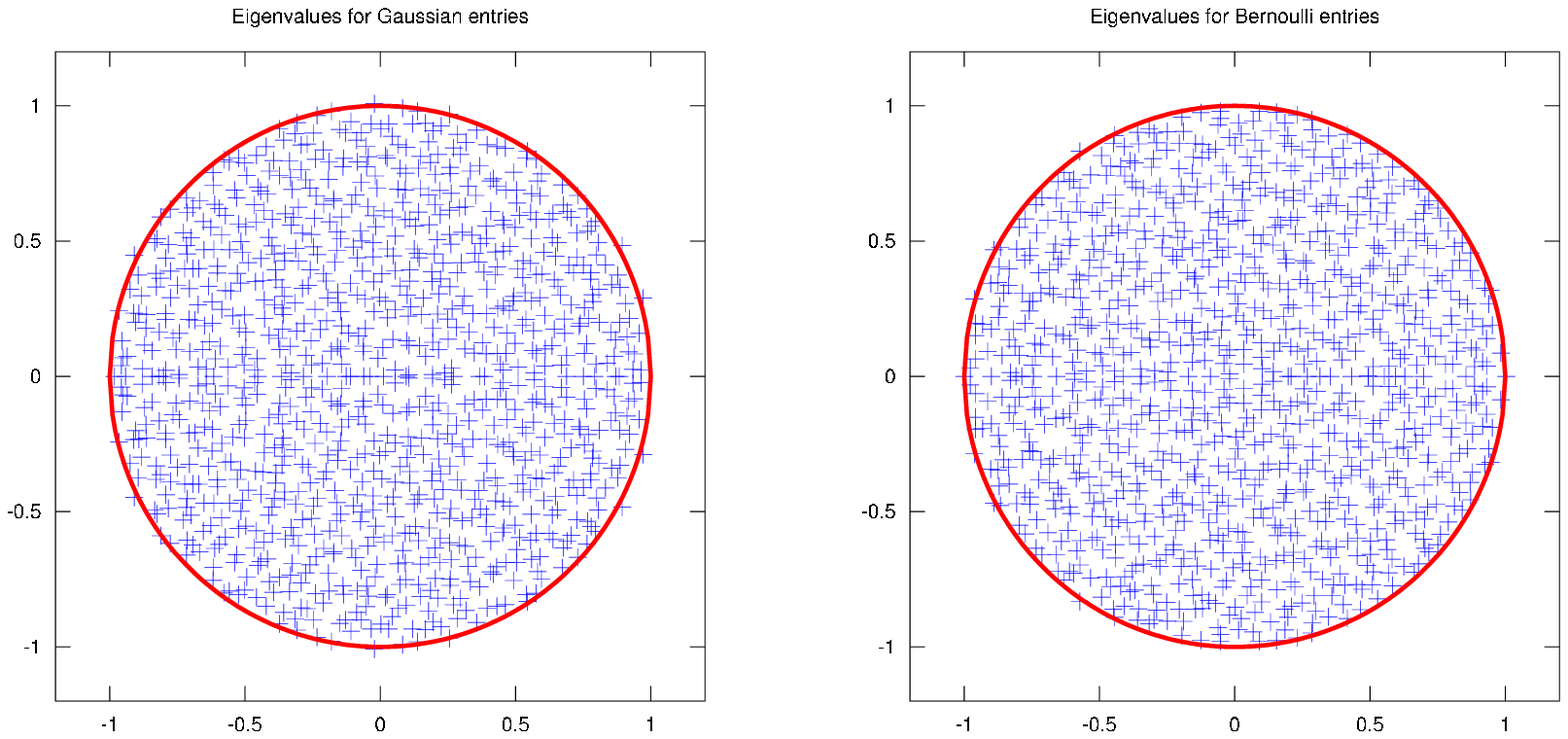}
    \caption{Illustration of universality in the quarter circular law and the
      circular law theorems \ref{th:quarter-circular} and \ref{th:circular}.
      The plots are made with the singular values (upper plots) and
      eigenvalues (lower plot) for a single random matrix $X$ of dimension
      $n=1000$. On the left hand side, $X_{11}$ follows a standard Gaussian
      law on $\dR$, while on the right hand side $X_{11}$ follows a symmetric
      Bernoulli law on $\{-1,1\}$. Since $X$ has real entries, the spectrum is
      symmetric with respect to the real axis. 
      A striking fact behind such simulations for the eigenvalues (lower
      plots) is the remarkable stability of the numerical algorithms for the
      eigenvalues despite the sensitivity of the spectrum of non-normal
      matrices. Is it the \'Sniady regularization of Brown measure theorem
      \cite{MR1929504} at work due to floating point approximate numerics?}
    \label{fi:qcl}
  \end{center}
\end{figure}

The a.s.\ tightness of $\mu_{n^{-1/2}X}$ is easily understood since by Weyl's
inequality we obtain
\[
\int\!|\la|^2\,d\mu_{n^{-1/2}X}(\la)
=\frac{1}{n^2}\sum_{i=1}^n|\la_i(X)|^2
\leq \frac{1}{n^2}\sum_{i=1}^ns_i(X)^2
= \int\!s^2\,d\nu_{n^{-1/2}X}(s).
\]

The convergence in the couple of theorems above is the weak convergence of
probability measures with respect to continuous bounded functions. We recall
that this mode of convergence does not capture the convergence of the support.
It implies only that a.s.\
\[
\varliminf_{n\to\infty}s_1(n^{-1/2}X)\geq 2
\quad\text{and}\quad 
\varliminf_{n\to\infty}|\la_1(n^{-1/2}X)|\geq 1.
\]
However, following \cite{MR1235416,MR963829,bai-silverstein-book} and
\cite{MR863545,1687963}, if $\dE(X_{1,1})=0$ and $\dE(|X_{1,1}|^4)<\infty$
then a.s.\footnote{The argument is based on Gelfand's formula: if
  $A\in\cM_n(\dC)$ then $|\la_1(A)|=\lim_{k\to\infty}\Vert A^k\Vert^{1/k}$ for
  any norm $\NRM{\cdot}$ on $\cM_n(\dC)$ (recall that all norms are equivalent in finite dimension).
  In the same spirit, the Yamamoto theorem states that
  $\lim_{k\to\infty} s_i(A^k)^{1/k}=|\la_i(A)|$ for every $1\leq i\leq n$, see
  \cite[Theorem 3.3.21]{MR1288752}.}
\[
\lim_{n\to\infty}s_1(n^{-1/2}X)= 2
\quad\text{and}\quad 
\lim_{n\to\infty}|\la_1(n^{-1/2}X)|=1.
\]
The asymptotic factor $2$ between the operator norm and the spectral radius
indicates in a sense that $X$ is a non-normal matrix asymptotically as
$n\to\infty$ (note that if $X_{11}$ is absolutely continuous then $X$ is
absolutely continuous and thus $XX^*\neq X^*X$ a.s.\ which means that $X$ is
non-normal a.s.). The law of the modulus under the circular law has density
$\rho\mapsto 2\rho\IND_{[0,1]}(\rho)$ which differs completely from the shape
of the quarter circular law $s\mapsto\pi^{-1}\sqrt{4-s^2}\,\IND_{[0,2]}(s)$,
see figure \ref{fi:nonorm}. The integral of ``$\log$'' for both laws is the
same.

\begin{figure}[htbp]
 \includegraphics[scale=0.5]{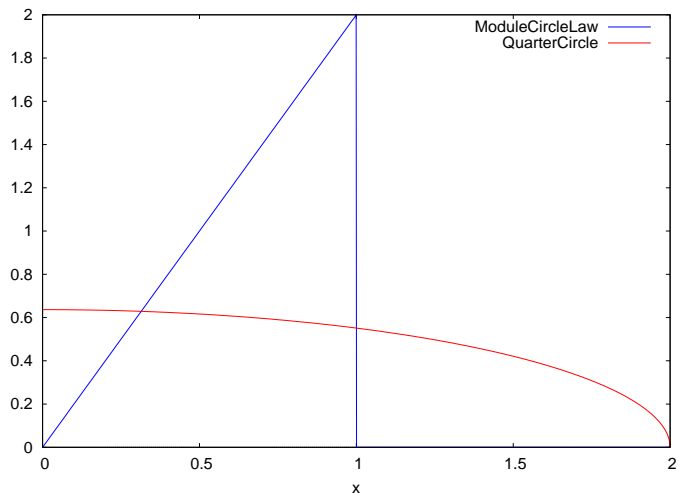}
 \caption{Comparison between the quarter circular distribution of theorem
   \ref{th:quarter-circular} for the singular values, and the modulus under
   the circular law of theorem \ref{th:circular} for the eigenvalues. The
   supports and the shapes are different. This difference indicates the
   asymptotic non-normality of these matrices. The integral of the function
   $t\mapsto\log(t)$ is the same for both distributions.}
 \label{fi:nonorm}
\end{figure}

\section{Gaussian case}
\label{se:gaussian}

This section is devoted to the case where $X_{11}\sim\cN(0,\frac{1}{2}I_2)$.
From now on, we denote $G$ instead of $X$ in order to distinguish the Gaussian
case from the general case. We say that $G$ belongs to the \emph{Complex
  Ginibre Ensemble}. The Lebesgue density of the $n\times n$ random matrix
$G=(G_{i,j})_{1\leq i,j\leq n}$ in $\mathcal{M}_n(\dC)\equiv\dC^{n\times n}$
is
\begin{equation}\label{eq:ginden}
  A\in\mathcal{M}_n(\dC)%
  \mapsto 
  \pi^{-n^2}e^{-\sum_{i,j=1}^n|A_{ij}|^2}
\end{equation}
where $A^*$ the conjugate-transpose of $A$. This law is a Boltzmann
distribution with energy 
\[
A\mapsto\sum_{i,j=1}^n|A_{ij}|^2 %
=\TR(AA^*) %
=\NRM{A}_2^2 %
=\sum_{i=1}^ns_i^2(A).
\]
This law is unitary invariant, in the sense that if $U$ and $V$ are $n\times
n$ unitary matrices then $UGV$ and $G$ are equally distributed. If $H_1$ and
$H_2$ are two independent copies of GUE\footnote{Up to scaling, a random
  $n\times n$ Hermitian matrix $H$ belongs to the Gaussian Unitary Ensemble
  (GUE) when its density is proportional to
  $H\mapsto\exp(-\frac{1}{2}\TR(H^2))=\exp(-\frac{1}{2}\sum_{i=1}^n|H_{ii}|^2-\sum_{1\leq
    i<j\leq n}|H_{ij}|^2)$. Equivalently $\{H_{ii},H_{ij}:1\leq i\leq
  n,i<j\leq n\}$ are indep.\ and $H_{ii}\sim\cN(0,1)$ and
  $H_{ij}\sim\cN(0,\frac{1}{2}I_2)$ for $i\neq j$.} then $(H_1+iH_2)/\sqrt{2}$
has the law of $G$. Conversely, the matrices $(G+G^*)/\sqrt{2}$ and
$(G-G^*)/\sqrt{2i}$ are independent and belong to the GUE.

The singular values of $G$ are the square root of the eigenvalues of the
positive semidefinite Hermitian matrix $GG^*$. The matrix $GG^*$ is a complex
Wishart matrix, and belongs to the complex Laguerre Ensemble ($\beta=2$). The
empirical distribution of the singular values of $n^{-1/2}G$ tends to the
Marchenko-Pastur quarter circular distribution (Gaussian case in theorem
\ref{th:quarter-circular}). This section is rather devoted to the study of the
eigenvalues of $G$, and in particular to the proof of the circular law theorem
\ref{th:circular} in this Gaussian settings.

\begin{lemma}[Diagonalizability]\label{le:diag}
  For every $n\geq1$, the set of elements of $\mathcal{M}_n(\dC)$ with
  multiple eigenvalues has zero Lebesgue measure in $\dC^{n\times n}$.
  In particular, the set of non-diagonalizable elements of
  $\mathcal{M}_n(\dC)$ has zero Lebesgue measure in
  $\dC^{n\times n}$.
\end{lemma}

\begin{proof}
  If $A\in\cM_n(\dC)$ has characteristic polynomial
  \[
  P_A(z)=z^n+a_{n-1}z^{n-1}+\cdots+a_0,
  \]
  then $a_0,\ldots,a_{n-1}$ are polynomials of the entries of $A$. The
  \emph{resultant} $R(P_A,P'_A)$ of $P_A,P'_A$, called the \emph{discriminant}
  of $P_A$, is the determinant of the $(2n-1)\times(2n-1)$ Sylvester matrix of
  $P_A,P'_A$. It is a polynomial in $a_0,\ldots,a_{n-1}$. We have also the
  Vandermonde formula
  \[
  |R(P_A,P'_A)|=\prod_{i{<}j}|\lambda_i(A)-\lambda_j(A)|^2.
  \]
  Consequently, $A$ has all eigenvalues distinct if and only if $A$ lies
  outside the proper polynomial hyper-surface $\{A\in\dC^{n\times
    n}:R(P_A,P'_A)=0\}$.
\end{proof}

Since $G$ is absolutely continuous, we have a.s.\ $GG^*\neq G^*G$
(non-normality). Additionally, lemma \ref{le:diag} gives that a.s.\ $G$ is
diagonalizable with distinct eigenvalues. Following Ginibre \cite{MR0173726}
-- see also \cite[Chapter 15]{MR2129906,MR2641363} and \cite{KS} -- one may
then compute the joint density of the eigenvalues $\la_1(G),\ldots,\la_n(G)$
of $G$ by integrating \eqref{eq:ginden} over the non-eigenvalues variables.
The result is stated in theorem \ref{th:complex-ginibre-density} below. It is
worthwhile to mention that in contrast with Hermitian unitary invariant
ensembles, the computation of the spectrum law is problematic if one replaces
the square potential by a more general potential, see \cite{KS}. The law of
$G$ is invariant by the multiplication of the entries with a common phase, and
thus the law of the spectrum of $G$ has also the same property. In the sequel
we set
\[
\Delta_n:=\{(z_1,\ldots,z_n)\in\dC^n:|z_1|\geq\cdots\geq|z_n|\}.
\]

\begin{theorem}[Spectrum law]\label{th:complex-ginibre-density}
  $(\lambda_1(G),\ldots,\lambda_n(G))$ has density
  $n!\varphi_n\IND_{\Delta_n}$ where
  \[
  \varphi_n(z_1,\ldots,z_n)=\frac{\pi^{-n^2}}{1!2!\cdots n!}
  \exp\PAR{-\sum_{k=1}^n|z_k|^2}\prod_{1\leq i{<}j\leq n}|z_i-z_j|^2.
  \]
  In particular, for every symmetric Borel function $F:\dC^n\to\dR$,
  \[
  \dE[F(\lambda_1(G),\ldots,\lambda_n(G))]
  =\int_{\dC^n}\!F(z_1,\ldots,z_n)\varphi_n(z_1,\ldots,z_n)\,dz_1\cdots dz_n.
  \]
\end{theorem}

We will use theorem \ref{th:complex-ginibre-density} with symmetric functions
of the form
\[
F(z_1,\ldots,z_n)
=\sum_{i_1,\ldots,i_k \text{ distinct}}f(z_{i_1})\cdots f(z_{i_k}).
\]
The Vandermonde determinant comes from the Jacobian of the diagonalization,
and can be interpreted as an electrostatic repulsion. The spectrum is a
Gaussian determinantal point process, see \cite[Chapter 4]{MR2552864}.

\begin{theorem}[$k$-points correlations]\label{th:complex-ginibre-k-points}
  Let $z\in\dC\mapsto\gamma(z)=\pi^{-1}e^{-|z|^2}$ be the density of
  the standard Gaussian $\mathcal{N}(0,\frac{1}{2}I_2)$ on $\dC$. Then
  for every $1\leq k\leq n$, the ``$k$-point correlation''
  \[
  \varphi_{n,k}(z_1,\ldots,z_k) :=
  \int_{\dC^{n-k}}\!\varphi_n(z_1,\ldots,z_k)\,dz_{k+1}\cdots dz_n
  \]
  satisfies 
  \[
  \varphi_{n,k}(z_1,\ldots,z_k) = 
  \frac{(n-k)!}{n!}\pi^{-k^2}\gamma(z_1)\cdots\gamma(z_k)
  \det\SBRA{K(z_i,z_j)}_{1\leq i,j\leq k}
  \]
  where
  \[
  K(z_i,z_j)
  :=\sum_{\ell=0}^{n-1}\frac{(z_iz_j^*)^\ell}{\ell!} %
  =\sum_{\ell=0}^{n-1}H_\ell(z_i)H_\ell(z_j)^* %
  \quad\text{with}\quad %
  H_\ell(z):=\frac{1}{\sqrt{\ell!}}z^\ell. 
  \]
  In particular, by taking $k=n$ we get
  \[
  \varphi_{n,n}(z_1,\ldots,z_n)
  =\varphi_n(z_1,\ldots,z_n)
  =\frac{1}{n!}\pi^{-n^2}\gamma(z_1)\cdots\gamma(z_n)\det\SBRA{K(z_i,z_j)}_{1\leq i,j\leq n}.
  \]
\end{theorem}

\begin{proof}  
  Calculations made by \cite[Chapter 15 page 271 equation 15.1.29]{MR2129906}
  using
  \[
  \prod_{1\leq i{<}j\leq n}|z_i-z_j|^2
  =\prod_{1\leq i{<}j\leq n}(z_i-z_j)\prod_{1\leq i{<}j\leq n}(z_i-z_j)^* 
  \]
  and 
  \[
  \det\SBRA{z_j^{i-1}}_{1\leq i,j\leq n}\det\SBRA{(z_j^*)^{i-1}}_{1\leq
    i,j\leq n}
  =\frac{1}{n!}\det\SBRA{K(z_i,z_j)}_{1\leq i,j\leq n}.
  \]
\end{proof}

Recall that if $\mu$ is a random probability measure on $\dC$ then $\dE\mu$ is
the deterministic probability measure defined
for every bounded
measurable $f$ by
\[
\int\!f\,d\dE\mu:=\dE\int\!f\,d\mu.
\] 

\begin{theorem}[Mean circular Law]
  \label{th:complex-ginibre-circle-weak}
  $\dE\mu_{n^{-1/2}G}\weak\cC_1$ as $n\to\infty$, where $\cC_1$ is the
  circular law i.e.\ the uniform law on the unit disc of $\dC$ with density
  $z\mapsto\pi^{-1}\IND_{\{z\in\dC:|z|\leq 1\}}$.
\end{theorem}

\begin{proof}
  From theorem \ref{th:complex-ginibre-k-points}, with $k=1$, we get that the
  density of $\dE\mu_{G}$ is
  \[
  \varphi_{n,1}:
  z\mapsto
  \gamma(z)\PAR{\frac{1}{n}\sum_{\ell=0}^{n-1}|H_\ell|^2(z)}
  =\frac{1}{n\pi}e^{-|z|^2}\sum_{\ell=0}^{n-1}\frac{|z|^{2\ell}}{\ell!}.
  \]
  Following Mehta \cite[Chapter 15 page 272]{MR2129906}, some calculus gives
  for every compact $C\subset\dC$
  \[
  \lim_{n\to\infty}\sup_{z\in C}
  \ABS{n\varphi_{n,1}(\sqrt{n}z)-\pi^{-1}\IND_{[0,1]}(|z|)}=0.
  \]
  The $n$ in front of $\varphi_{n,1}$ is due to the fact that we are on the
  complex plane $\dC=\dR^2$ and thus
  $d\sqrt{n}xd\sqrt{n}y=ndxdy$. Here is the start of the elementary calculus:
  for $r^2<n$,
  \[
  e^{r^2}-\sum_{\ell=0}^{n-1}\frac{r^{2\ell}}{\ell!}
  =\sum_{\ell=n}^\infty\frac{r^{2\ell}}{\ell!}
  \leq \frac{r^{2n}}{n!}\sum_{\ell=0}^\infty\frac{r^{2\ell}}{(n+1)^\ell}
  =\frac{r^{2n}}{n!}\frac{n+1}{n+1-r^2}
  \]
  while for $r^2>n$,
  \[
  \sum_{\ell=0}^{n-1}\frac{r^{2\ell}}{\ell!}
  \leq\frac{r^{2(n-1)}}{(n-1)!}\sum_{\ell=0}^{n-1}\PAR{\frac{n-1}{r^2}}^\ell
  \leq \frac{r^{2(n-1)}}{(n-1)!}\frac{r^2}{r^2-n+1}.
  \]
  By taking $r^2=|\sqrt{n}z|^2$ we obtain the convergence of the density
  uniformly on compact subsets, which implies in particular the weak
  convergence.
\end{proof}

The sequence $(H_k)_{k\in\mathbb{N}}$ forms an orthonormal basis (orthogonal
polynomials) of square integrable analytic functions on $\dC$ for the standard
Gaussian on $\dC$. The uniform law on the unit disc (known as the circular
law) is the law of $\sqrt{V}e^{2i\pi W}$ where $V$ and $W$ are i.i.d.\ uniform
random variables on the interval $[0,1]$. This point of view can be used to
interpolate between complex Ginibre and GUE via Girko's elliptic laws, see
\cite{MR2446909,MR2288065,MR2594353}.

We are ready to prove the complex Gaussian version of the circular law theorem
\ref{th:circular}.

\begin{theorem}[Circular law]\label{th:complex-ginibre-circle-strong}
  a.s.\ $\mu_{n^{-1/2}G}\weak\cC_1$ as $n\to\infty$, where $\cC_1$ is the
  circular law i.e.\ the uniform law on the unit disc of $\dC$ with density
  $z\mapsto\pi^{-1}\IND_{\{z\in\dC:|z|\leq 1\}}$.
\end{theorem}

\begin{proof}
  We reproduce Silverstein's argument, published by Hwang \cite{MR841088}. The
  argument is similar to the quick proof of the strong law of large numbers
  for independent random variables with bounded fourth moment. It suffices to
  establish the result for compactly supported continuous bounded functions.
  Let us pick such a function $f$ and set
  \[
  S_n:=\int_{\dC}\!f\,d\mu_{n^{-1/2}G}
  \quad\text{and}\quad
  S_\infty:=\pi^{-1}\int_{|z|\leq 1}\!f(z)\,dxdy.
  \]
  Suppose for now that we have
  \begin{equation}\label{eq:mom4}
    \dE[\PAR{S_n-\dE S_n}^4]=O(n^{-2}).
  \end{equation}
  By monotone convergence (or by the Fubini-Tonelli theorem),
  \[
  \dE\sum_{n=1}^\infty\PAR{S_n-\dE S_n}^4%
  =\sum_{n=1}^\infty\dE[\PAR{S_n-\dE S_n}^4]<\infty
  \]
  and consequently $\sum_{n=1}^\infty\PAR{S_n-\dE S_n}^4<\infty$
  a.s.\ which implies $\lim_{n\to\infty}S_n-\dE S_n=0$ a.s.\ Since
  $\lim_{n\to\infty}\dE S_n=S_\infty$ by theorem
  \ref{th:complex-ginibre-circle-weak}, we get that a.s.\
  \[
  \lim_{n\to\infty}S_n=S_\infty.
  \]
  Finally, one can swap the universal quantifiers on $\omega$ and $f$ thanks
  to the separability of the set of compactly supported continuous bounded
  functions $\dC\to\dR$ equipped with the supremum norm. To establish
  \eqref{eq:mom4}, we set
  \[
  S_n-\dE S_n=\frac{1}{n}\sum_{i=1}^nZ_i %
  \quad\text{with}\quad %
  Z_i:=f\PAR{\lambda_i\PAR{n^{-1/2}G}}.
  \]
  Next, we obtain, with $\sum_{i_1,\ldots}$ running over distinct indices in
  $1,\ldots,n$,
  \begin{align*}
    \dE\SBRA{\PAR{S_n-\dE S_n}^4}
    &=\frac{1}{n^4}\sum_{i_1}\dE[Z_{i_1}^4] \\
    &\quad +\frac{4}{n^4}\sum_{i_1,i_2}\dE[Z_{i_1}Z_{i_2}^3] \\
    &\quad +\frac{3}{n^4}\sum_{i_1,i_2}\dE[Z_{i_1}^2Z_{i_2}^2] \\
    &\quad +\frac{6}{n^4}\sum_{i_1,i_2,i_3}\dE[Z_{i_1}Z_{i_2}Z_{i_3}^2] \\
    &\quad +\frac{1}{n^4}\sum_{i_1,i_2,i_3,i_3,i_4}\!\!\!\!\!\!\dE[Z_{i_1}Z_{i_3}Z_{i_3}Z_{i_4}].
  \end{align*}
  The first three terms of the right hand side are $O(n^{-2})$ since
  $\max_{1\leq i\leq n}|Z_i|\leq\Vert f\Vert_\infty$. Finally, some calculus
  using the expressions of $\varphi_{n,3}$ and $\varphi_{n,4}$ provided by
  theorem \ref{th:complex-ginibre-k-points} allows to show that the remaining
  two terms are also $O(n^{-2})$. See Hwang \cite[p. 151]{MR841088}.
\end{proof}

It is worthwhile to mention that one can deduce the circular law theorem
\ref{th:complex-ginibre-circle-strong} from a large deviations principle,
bypassing the mean circular law theorem \ref{th:complex-ginibre-circle-weak},
see section \ref{se:related}.

Following Kostlan \cite{MR1148410} (see also Rider \cite{MR1986426} and
\cite{MR2552864}) the integration of the phases in the joint density of the
spectrum given by theorem \ref{th:complex-ginibre-density} leads to theorem
\ref{th:complex-ginibre-module} below.

\begin{theorem}[Layers]\label{th:complex-ginibre-module}
  If $Z_{1},\ldots,Z_{n}$ are independent with\footnote{Here $\Gamma(a,\la)$
    stands for the probability measure on $\dR_+$ with Lebesgue density $x\mapsto
    \la^a\Gamma(a)^{-1} x^{a-1}e^{-\la x}$.} $Z_k^2\sim\Gamma(k,1)$ for all
  $k$, then
  \[
  (|\lambda_1(G)|,\ldots,|\lambda_n(G)|)
  \overset{d}{=}
  (Z_{(1)},\ldots,Z_{(n)})
  \]
  where $Z_{(1)},\ldots,Z_{(n)}$ is the non-increasing reordering of
  the sequence $Z_1,\ldots,Z_n$.
\end{theorem}

Note that\footnote{Here $\chi^2(n)$ stands for the law of $\NRM{V}_2^2$ where
  $V\sim\cN(0,I_n)$.} $(\sqrt{2}Z_k)^2\sim\chi^2(2k)$ which is useful for
$\sqrt{2}G$. Since $Z_k^2\overset{d}{=}E_1+\cdots+E_k$ where $E_1,\ldots,E_k$
are i.i.d.\ exponential random variables of unit mean, we get, for every
$r>0$,
\[
\dP\PAR{\ABS{\la_1(G)}\leq \sqrt{n}r} %
=\prod_{1\leq k\leq n}\dP\PAR{\frac{E_1+\cdots+E_k}{n}\leq r^2}
\]
The law of large numbers suggests that $r=1$ is a critical value. The central
limit theorem suggests that $\ABS{\lambda_1(n^{-1/2}G)}$ behaves when $n\gg1$
as the maximum of i.i.d.\ Gaussians, for which the fluctuations follow the
Gumbel law. A quantitative central limit theorem and the Borel-Cantelli lemma
provides the follow result. The full proof is in Rider \cite{MR1986426}.

\begin{theorem}[Convergence and fluctuation of the spectral radius]
  \label{th:complex-ginibre-spectral-radius}
  \[
  \dP\PAR{\lim_{n\to\infty}|\lambda_1(n^{-1/2}G)|=1}=1.
  \]
  Moreover, if $\gamma_n:=\log(n/2\pi)-2\log(\log(n))$ then
  \[
  \sqrt{4n\gamma_n}
  \PAR{|\lambda_1(n^{-1/2}G)|-1-\sqrt{\frac{\gamma_n}{4n}}}
  \overset{d}{\underset{n\to\infty}{\longrightarrow}} \cG
  \]
  where $\mathcal{G}$ is the Gumbel law with cumulative distribution function
  $x\mapsto e^{-e^{-x}}$ on $\dR$.
\end{theorem}

The convergence of the spectral radius was obtained by Mehta \cite[chapter 15
page 271 equation 15.1.27]{MR2129906} by integrating the joint density of the
spectrum of theorem \ref{th:complex-ginibre-density} over the set
$\bigcap_{1\leq i\leq n}\{|\lambda_i|>r\}$. The same argument is reproduced by
Hwang \cite[pages 149--150]{MR841088}. Let us give now an alternative
derivation of theorem \ref{th:complex-ginibre-circle-weak}. From theorem
\ref{th:complex-ginibre-spectral-radius}, the sequence
$(\dE\mu_{n^{-1/2}G})_{n\geq1}$ is tight and every accumulation point $\mu$ is
supported in the unit disc. From theorem \ref{th:complex-ginibre-density},
such a $\mu$ is rotationally invariant, and from theorem
\ref{th:complex-ginibre-module}, the image of $\mu$ by $z\in\dC\mapsto|z|$ has
density $r\mapsto 2r\IND_{[0,1]}(r)$ (use moments!). Theorem
\ref{th:complex-ginibre-circle-weak} follows immediately.

It is remarkable that the large eigenvalues in modulus of the complex Ginibre
ensemble are asymptotically independent, which gives rise to a Gumbel
fluctuation, in contrast with the GUE and its delicate Tracy-Widom
fluctuation, see \cite{MR2288065} for an interpolation.


\begin{remark}[Real Ginibre Ensemble]\label{rk:real-ginibre}
  Ginibre considered also in his paper \cite{MR0173726} the case where $\dC$
  is replaced by $\dR$ or by the quaternions. These cases are less understood
  than the complex case due to their peculiarities. Let us focus on the Real
  Ginibre Ensemble, studied by Edelman and his collaborators. The expected
  number of real eigenvalues is equivalent to $\sqrt{2n/\pi}$ as $n\to\infty$,
  see \cite{MR1231689}, while the probability that all the eigenvalues are
  real is exactly $2^{-n(n-1)/4}$, see \cite[Corollary 7.1]{MR1437734}. The
  expected counting measure of the real eigenvalues, scaled by $\sqrt{n}$,
  tends to the uniform law on the interval $[-1,1]$, see \cite[Theorem
  4.1]{MR1231689} and figures \ref{fi:real-eig-unif}. The eigenvalues do not
  have a density in $\dC^n$, except if we condition on the real eigenvalues,
  see \cite{MR1437734}. The analogue of the weak circular law theorem
  \ref{th:complex-ginibre-circle-weak} was proved by Edelman \cite[Theorem
  6.3]{MR1231689}. More information on the structure of the Real Ginibre
  Ensemble can be found in \cite{MR2363393}, \cite{MR2530159}, and in
  \cite{KS} and \cite[Chapter 15]{MR2641363}.
\end{remark}

On overall, one can remember that the Complex Ginibre Ensemble is in a way
``simpler'' than the GUE while the Real Ginibre Ensemble is ``harder'' than
the GOE.

\begin{center}
  \fbox{Real Ginibre $\geq$ GOE $\geq$ GUE $\geq$ Complex Ginibre} 
\end{center}

\begin{figure}[htbp]
  \begin{center}
  \includegraphics[scale=.9]{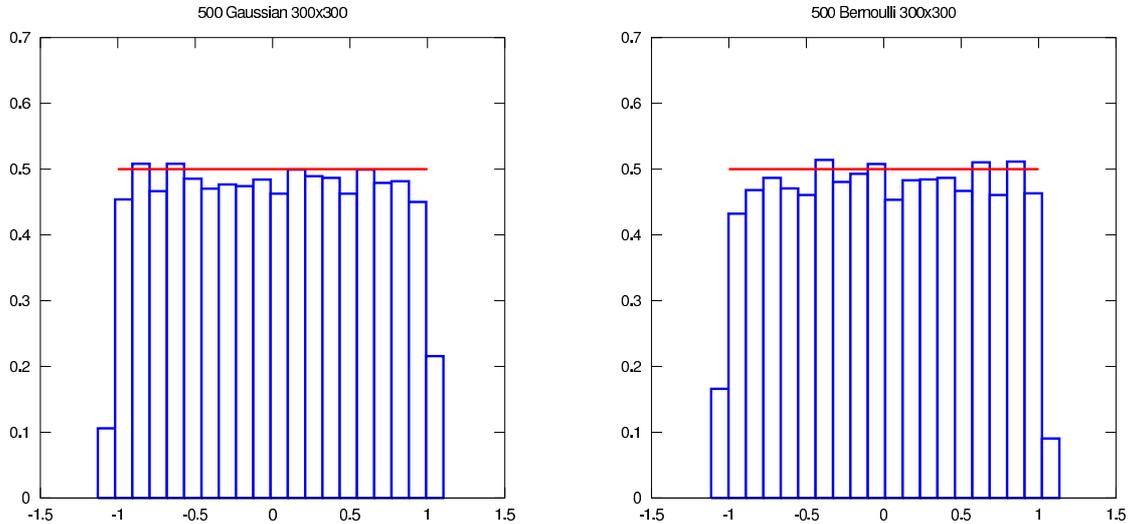}
  \caption{Histograms of real eigenvalues of $500$ i.i.d.\ copies of
    $n^{-1/2}X$ with $n=300$. The left hand side graphic corresponds to the
    standard real Gaussian case $X_{11}\sim\cN(0,1)$, while the right hand
    side graphic corresponds to the symmetric Bernoulli case
    $X_{11}\sim\frac{1}{2}(\de_{-1}+\de_1)$. See remark
    \ref{rk:real-ginibre}.}
    \label{fi:real-eig-unif}
  \end{center}
\end{figure}

\begin{remark}[Quaternionic Ginibre Ensemble]
  The quaternionic Ginibre Ensemble was considered at the origin by Ginibre
  \cite{MR0173726}. It has been recently shown \cite{benaych-georges-chapon}
  by using the logarithmic potential that there exists an analogue of the
  circular law theorem for this ensemble, in which the limiting law is
  supported in the unit ball of the quaternions field.
\end{remark}

\section{Universal case}
\label{se:universal}

This section is devoted to the proof of the circular law theorem
\ref{th:circular} following \cite{tao-vu-cirlaw-bis}. The universal
Marchenko-Pastur theorem \ref{th:quarter-circular} can be proved by using
powerful Hermitian techniques such as truncation, centralization, the method
of moments, or the Cauchy-Stieltjes trace-resolvent transform. It turns out
that all these techniques fail for the eigenvalues of non-normal random
matrices. Indeed, the key to prove the circular law theorem \ref{th:circular}
is to use a bridge pulling back the problem to the Hermitian world. This is
called \emph{Hermitization}.

Actually, and as we will see in sections \ref{se:related} and \ref{se:heavy},
there is a non-Hermitian analogue of the method of moments called the
$*$-moments, and there is an analogue of the Cauchy-Stieltjes trace-resolvent
in which the complex variable is replaced by a quaternionic type variable.

\subsection{Logarithmic potential and Hermitization}\label{se:logpot}

Let $\mathcal{P}(\dC)$ be the set of probability measures on
$\dC$ which integrate $\log\ABS{\cdot}$ in a neighborhood of infinity.
The \emph{logarithmic potential} $U_\mu$ of $\mu\in\mathcal{P}(\dC)$ is
the function $U_\mu:\dC\to(-\infty,+\infty]$ defined for all
$z\in\dC$ by
\begin{equation}\label{eq:logpot}
  U_\mu(z)=-\int_{\dC}\!\log|z-\la|\,d\mu(\la)
  =-(\log\ABS{\cdot}*\mu)(z).
\end{equation}
For instance, for the circular law $\cC_1$ of density
$\pi^{-1}\IND_{\{z\in\dC:|z|\leq1\}}$, we have, for every $z\in\dC$,
\begin{equation}\label{eq:logpotu}
  U_{\cC_1}(z)=
\begin{cases}
  -\log|z| & \text{if $|z|>1$}, \\
  \frac{1}{2}(1-|z|^2) & \text{if $|z|\leq1$},
\end{cases}
\end{equation}
see e.g.\ \cite{MR1485778}. Let $\cD'(\dC)$ be the set of
Schwartz-Sobolev distributions. We have $\cP(\dC)\subset\cD'(\dC)$. Since
$\log\ABS{\cdot}$ is Lebesgue locally integrable on $\dC$, the Fubini-Tonelli
theorem implies that $U_\mu$ is Lebesgue locally integrable on $\dC$. In
particular, $U_\mu<\infty$ a.e.\ and $U_\mu\in\cD'(\dC)$.

Let us define the first order linear differential operators in $\cD'(\dC)$
\begin{equation}\label{eq:ddbar}
  \pd := \frac{1}{2}(\pd_x-i\pd_y)
  \quad\text{and}\quad
  \Ol \pd := \frac{1}{2}(\pd_x+i\pd_y)
\end{equation}
and the Laplace operator
$\De=4\pd\Ol\pd=4\Ol\pd\pd=\pd_x^2+\pd_y^2$.
Each of these operators coincide on smooth functions with the usual
differential operator acting on smooth functions. By using Green's or Stockes'
theorems, one may show, for instance via the Cauchy-Pompeiu formula, that for
any smooth and compactly supported function $\varphi:\dC\to\dR$,
\begin{equation}\label{eq:fonsol}
-\int_{\dC}\!\Delta\vphi(z)\log\ABS{z}\,dxdy=2\pi\vphi(0)
\end{equation}
where $z=x+iy$. Now \eqref{eq:fonsol} can be written, in $\cD'(\dC)$,
\[
\Delta\log\ABS{\cdot}=2\pi\delta_0
\]
In other words, $\frac{1}{2\pi}\log\ABS{\cdot}$ is the fundamental solution of
the Laplace equation on $\dR^2$. Note that $\log\ABS{\cdot}$ is harmonic on
$\dC\setminus\{0\}$. It follows that in $\cD'(\dC)$,
\begin{equation}\label{eq:lap}
  \Delta U_\mu=-2\pi\mu.
\end{equation}
This means that for every smooth and compactly supported ``test function''
$\varphi:\dC\to\dR$,
\begin{equation}\label{eq:lapphi}
\DOT{\Delta U_\mu,\vphi}_{\cD'(\dC)}
=-\int_{\dC}\!\Delta\varphi(z)U_\mu(z)\,dxdy
=-2\pi\int_{\dC}\!\varphi(z)\,d\mu(z)
=-\DOT{2\pi\mu,\vphi}_{\cD'(\dC)}
\end{equation}
where $z=x+iy$. This means that $-\frac{1}{2\pi}U_\cdot$ is the Green operator
on $\dR^2$ (Laplacian inverse).

\begin{lemma}[Unicity]\label{le:unicity}
  For every $\mu,\nu\in\mathcal{P}(\dC)$, if $U_\mu=U_\nu$ a.e.\ then
  $\mu=\nu$.
\end{lemma}

\begin{proof}
  Since $U_\mu=U_\nu$ in $\cD'(\dC)$, we get $\Delta U_\mu=\Delta
  U_\nu$ in $\cD'(\dC)$. Now \eqref{eq:lap} gives $\mu=\nu$ in
  $\cD'(\dC)$, and thus $\mu=\nu$ as measures since $\mu$ and $\nu$
  are Radon measures. Note more generally that the lemma remains valid if
  $U_\mu=U_\nu+h$ for some harmonic $h\in\cD'(\dC)$.
\end{proof}

If $A$ is a $n\times n$ complex matrix and $P_A(z):=\det(A-zI)$ is its
characteristic polynomial,
\[
U_{\mu_A}(z)
=-\int_{\dC}\!\log\ABS{\la-z}\,d\mu_A(\la)
=-\frac{1}{n}\log\ABS{\det(A-zI)}
=-\frac{1}{n}\log\ABS{P_A(z)}
\]
for every $z\in\dC\setminus\{\lambda_1(A),\ldots,\lambda_n(A)\}$. We have also
the alternative expression\footnote{Girko uses the name ``$V$-transform of
  $\mu_A$'', where $V$ stands for ``Victory''.}
\begin{equation}\label{eq:UESD}
  U_{\mu_A}(z)
  =-\frac{1}{n}\log\det(\sqrt{(A-zI)(A-zI)^*})
  =-\int_0^\infty\!\log(t)\,d\nu_{A-zI}(t).
\end{equation}
One may retain from this determinantal Hermitization that for any
$A\in\cM_n(\dC)$,
\[
\framebox{knowledge of $\nu_{A-zI}$ for
  a.a.\ $z\in\dC$ \quad$\Rightarrow$\quad knowledge of $\mu_A$}
\]
Note that from \eqref{eq:lap}, for every smooth compactly supported function
$\vphi:\dC\to\dR$,
\[
  2\pi\int\!\varphi\,d\mu_A
  =\int_{\dC}\!(\Delta \varphi)\log|P_A|\,dxdy.
\]
The identity \eqref{eq:UESD} bridges the eigenvalues with the singular values,
and is at the heart of the next lemma, which allows to deduce the
convergence of $\mu_A$ from the one of $\nu_{A-zI}$. The strength of this
Hermitization lies in the fact that contrary to the eigenvalues, one can
control the singular values with the entries of the matrix using powerful
methods such as the method of moments or the trace-resolvent Cauchy-Stieltjes
transform. The price paid here is the introduction of the auxiliary variable
$z$. Moreover, we cannot simply deduce the convergence of the integral from
the weak convergence of $\nu_{A-zI}$ since the logarithm is unbounded on
$\dR_+$. We circumvent this problem by requiring uniform integrability. We
recall that on a Borel measurable space $(E,\mathcal{E})$, a Borel function
$f:E\to\dR$ is \emph{uniformly integrable} for a sequence of probability
measures $(\eta_n)_{n\geq1}$ on $E$ when
\begin{equation*}
  \lim_{t\to\infty}\sup_{n\geq 1}\int_{\{|f|>t\}}\!|f|\,d\eta_n=0.
\end{equation*}
We will use this property as follows: if $\eta_n\weak\eta$ as $n\to\infty$ for
some probability measure $\eta$ and if $f$ is continuous and uniformly
integrable for $(\eta_n)_{n\geq1}$ then $f$ is $\eta$-integrable and
\[
\lim_{n\to\infty}\int\!f\,d\eta_n=\int\!f\,d\eta.
\] 

\begin{remark}[Weak convergence and uniform integrability in probability]
  \label{rk:inprob}
  Let $T$ be a topological space such as $\dR$ or $\dC$, and its Borel
  $\si$-field $\cT$. Let ${(\eta_n)}_{n\geq1}$ be a sequence of random
  probability measures on $(T,\cT)$ and $\eta$ be a probability measure on
  $(T,\cT)$. We say that $\eta_n\weak\eta$ \emph{in probability} if for all
  bounded continuous $f:T\to\dR$ and any $\veps >0$,
  \[
  \lim_{n\to\infty}\dP\PAR{\ABS{\int\!f\,d\eta_n-\int\!f\,d\eta}>\veps}=0.
  \]
  This is implied by the a.s.\ weak convergence. We say that a measurable
  function $f:T\to\dR$ is \emph{uniformly integrable in probability} for
  ${(\eta_n)}_{n\geq1}$ when for any $\veps>0$,
  \[
  \lim_{t\to\infty}\sup_{n\geq 1}\dP\PAR{\int_{|f|>t}\!|f|\,d\eta_n>\veps}=0.
  \]
  We will use this property as follows: if $\eta_n\weak\eta$ in probability
  and if $f$ is uniformly integrable for ${(\eta_n)}_{n\geq1}$ in probability
  then $f$ is $\eta$-integrable and $\int\!f\,d\eta_n$ converges in
  probability to $\int\!f\,d\eta$. This will be helpful in section
  \ref{se:heavy} together with lemma \ref{le:girko} in order to circumvent the
  lack of almost sure bounds on small singular values for heavy tailed random
  matrices.
\end{remark}

The idea of using Hermitization goes back at least to Girko \cite{MR773436}.
However, the proofs of lemmas \ref{le:girko} and \ref{le:unimaj} below are
inspired from the approach of Tao and Vu \cite{tao-vu-cirlaw-bis}.

\begin{lemma}[Hermitization]\label{le:girko}
  Let $(A_n)_{n\geq1}$ be a sequence of complex random matrices where $A_n$ is
  $n\times n$ for every $n\geq1$. Suppose that there exists a family of (non-random) 
  probability measures $(\nu_z)_{z\in\dC}$ on $\dR_+$ such that, for a.a.\ $z\in\dC$, a.s.
  \begin{itemize}
  \item[(i)]  $\nu_{A_n-zI}\weak\nu_z$ as $n\to\infty$ 
  \item[(ii)]  $\log$ is uniformly integrable for
    $\PAR{\nu_{A_n-zI}}_{n\geq1}$.
    \end{itemize}
    Then there exists a probability measure $\mu\in\mathcal{P}(\dC)$ such that
  \begin{itemize}
  \item[(j)] a.s.\ $\mu_{A_n}\weak\mu$ as $n\to\infty$ 
  \item[(jj)] for a.a.\ $z\in\dC$,
    \[
    U_\mu(z)=-\int_0^\infty\!\log(s)\,d\nu_z(s).
    \]
  \end{itemize}
  Moreover, if the convergence (i) and the uniform integrability (ii) both
  hold \emph{in probability} for a.a.\ $z\in\dC$ (instead of for a.a.\ $z\in\dC$, a.s.), then (j-jj) hold with the a.s.\ weak convergence in (j) replaced
  by the \emph{weak convergence in probability}.
\end{lemma}

\begin{proof}[Proof of lemma \ref{le:girko}]
  Let us give the proof of the a.s.\ part. We first observe that one can swap
  the quantifiers ``a.a.'' on $z$ and ``a.s.'' on $\om$ in front of (i-ii).
  Namely, let us call $P(z,\omega)$ the property ``the function $\log$ is
  uniformly integrable for $\PAR{\nu_{A_n(\omega)-zI}}_{n\geq1}$ and
  $\nu_{A_n(\omega)-zI}\weak\nu_z$". The assumptions of the lemma provide a
  measurable Lebesgue negligible set $C$ in $\dC$ such that for all $z \not\in
  C$ there exists a probability one event $E_z$ such that for all $\omega \in
  E_z$, the property $P(z,\omega)$ is true. From the Fubini-Tonelli theorem,
  this is \emph{equivalent} to the existence of a probability one event $E$
  such that for all $\omega \in E$, there exists a Lebesgue negligible
  measurable set $C_\omega$ in $\dC$ such that for all $z \not\in C_\omega$,
  the property $P(z,\omega)$ is true.

  From now on, we fix an arbitrary $\om\in E$. For every $z \not\in C_\omega$,
  let us define the probability measure $\nu:=\nu_z$ and the triangular arrays
  ${(a_{n,k})}_{1\leq k\leq n}$ and ${(b_{n,k})}_{1\leq k\leq n}$ by
  \[
  a_{n,k}:=|\lambda_k(A_n(\om)-zI)| \quad\text{and}\quad
  b_{n,k}:=s_k(A_n(\om)-zI).
  \]
  Note that $\mu_{A_n(\om)-zI}=\mu_{A_n(\om)}*\delta_{-z}$. Thanks to the Weyl
  inequalities \eqref{eq:weyl0} and to the assumptions (i-ii), one can use
  lemma \ref{le:unimaj} below, which gives that $(\mu_{A_n(\om)})_{n\geq1}$ is
  tight, that for a.a.\ $z \in \dC$, $\log\ABS{z-\cdot}$ is uniformly
  integrable for $(\mu_{A_n(\om)})_{n\geq1}$, and that
  \begin{equation*}
    \lim_{n\to\infty}U_{\mu_{A_n(\om)}}(z)=-\int_0^\infty\!\log(s)\,d\nu_z(s)=U(z).
  \end{equation*}
  Consequently, if the sequence ${(\mu_{A_n(\om)})}_{n\geq1}$ admits two
  probability measures $\mu_\om$ and $\mu_\om'$ as accumulation points for the
  weak convergence, then both $\mu_\om$ and $\mu_\om'$ belong to $\cP(\dC)$
  and $U_{\mu_\om}=U=U_{\mu_\om'}$ a.e., which gives $\mu_\om=\mu_\om'$ thanks
  to lemma \ref{le:unicity}. Therefore, the sequence
  ${(\mu_{A_n(\om)})}_{n\geq1}$ admits at most one accumulation point for the
  weak convergence. Since the sequence ${(\mu_{A_n(\om)})}_{n\geq1}$ is tight,
  the Prohorov theorem implies that ${(\mu_{A_n(\om)})}_{n\geq1}$ converges
  weakly to some probability measure $\mu_\om\in\cP(\dC)$ such that
  $U_{\mu_\om}=U$ a.e. Since $U$ is deterministic, it follows that
  $\om\mapsto\mu_\om$ is deterministic by lemma \ref{le:unicity} again. This
  achieves the proof of the a.s.\ part of the lemma. The proof of the ``in
  probability'' part of the lemma follows the same lines, using this time the
  ``in probability'' part of lemma \ref{le:unimaj}.
\end{proof}

\begin{remark}[Weakening uniform integrability in lemma \ref{le:girko}]
  \label{rk:weakui}
  The set of $z$ in $\dC$ such that $z$ is an atom of $\dE\mu_{A_n}$ for some
  $n\geq1$ is at most countable, and has thus zero Lebesgue measure. Hence,
  for a.a.\ $z \in \dC$, a.s.\ for all $n \geq 1$, $z$ is not an eigenvalue of
  $A_n$. Thus for a.a.\ $z \in \dC$, a.s.\ for all $n \geq 1$, 
  \[
  \int\!\log(s)\,d\nu_{A_n-zI}(s)<\infty.
  \]
  Hence, assumption (ii) in the a.s.\ part of lemma \ref{le:girko} holds if
  for a.a.\ $z \in \dC$, a.s.
  \[
  \lim_{t\to\infty}\varlimsup_{n \to \infty}\int_{\{|f|>t\}}\!|f| \,d\nu_{A_n-zI} (s)=0
  \]
  where $f=\log$. Similarly, regarding ``in probability'' part of lemma
  \ref{le:girko}, one can replace the $\sup$ by $\varlimsup$ in the definition
  of uniform integrability in probability.
\end{remark}

The following lemma is in a way the skeleton of proof of lemma \ref{le:girko}
(no matrices). It states essentially a propagation of a uniform
logarithmic integrability for a couple of triangular arrays, provided that a
logarithmic majorization holds between the arrays. 

\begin{lemma}[Logarithmic majorization and uniform integrability]
  \label{le:unimaj}
  Let $(a_{n,k})_{1\leq k\leq n}$ and $(b_{n,k})_{1\leq k\leq n}$ be two
  triangular arrays in $\dR_+$. Define the discrete probability measures
  \[
  \mu_n:=\frac{1}{n}\sum_{k=1}^n\delta_{a_{n,k}}
  \quad\text{and}\quad
  \nu_n:=\frac{1}{n}\sum_{k=1}^n\delta_{b_{n,k}}.
  \]
  If the following properties hold 
  \begin{itemize}
  \item[(i)] $a_{n,1}\geq\cdots\geq a_{n,n}$ and $b_{n,1}\geq\cdots\geq
    b_{n,n}$ for $n\gg1$,
  \item[(ii)] $\prod_{i=1}^k a_{n,i} \leq \prod_{i=1}^k b_{n,i}$ for every
    $1\leq k\leq n$ for $n\gg1$,
  \item[(iii)] $\prod_{i=k}^n b_{n,i} \leq \prod_{i=k}^n a_{n,i}$ for every
    $1\leq k\leq n$ for $n\gg1$,
  \item[(iv)] $\nu_n\weak\nu$ as $n\to\infty$ for some probability measure
    $\nu$,
  \item[(v)] $\log$ is uniformly integrable for $(\nu_n)_{n\geq1}$,
  \end{itemize}
  then 
  \begin{itemize}
  \item[(j)] $\log$ is uniformly integrable for $(\mu_n)_{n\geq1}$ (in
    particular, ${(\mu_n)}_{n\geq1}$ is tight),
  \item[(jj)] we have, as $n\to\infty$,
    \[
    \int_0^\infty\!\log(t)\,d\mu_n(t)
    =\int_0^\infty\!\log(t)\,d\nu_n(t)\to\int_0^\infty\!\log(t)\,d\nu(t),
    \]
  \end{itemize}
  and in particular, for every accumulation point $\mu$ of $(\mu_n)_{n\geq1}$,
  \[
  \int_0^\infty\!\log(t)\,d\mu(t)=\int_0^\infty\!\log(t)\,d\nu(t).
  \]
  Moreover, assume that $(a_{n,k})_{1\leq k\leq n}$ and $(b_{n,k})_{1\leq
    k\leq n}$ are random triangular arrays in $\dR_+$ defined on a common
  probability space such that (i-ii-iii) hold a.s.\ and (iv-v) hold in
  probability. Then (j-jj) hold in probability.
\end{lemma}

\begin{proof}
  An elementary proof can be found in \cite[Lemma C2]{cirmar}. Let us give an
  alternative argument. Let us start with the deterministic part. From the de
  la Vall\'ee Poussin criterion (see e.g.\ \cite[Theorem
  22]{dellacheriemeyer}), assumption (v) is equivalent to the existence of a
  non-decreasing convex function $J:\dR_+\to\dR_+$ such that
  $\lim_{t\to\infty}J(t)/t=\infty$, and
  \[
  \sup_n \int\!J(|\log(t)|)\,d\nu_{n}(t)< \infty. 
  \]
  On the other hand, assumption (i-ii-iii) implies that for every real valued
  function $\varphi$ such that $t\mapsto\varphi(e^t)$ is non-decreasing and
  convex, we have, for every $1\leq k\leq n$,
  \[
  \sum_{i=1}^k\varphi(a_{n,k} ) \leq \sum_{i=1}^k\varphi(b_{n,k} ),
  \]
  see \cite[Theorem 3.3.13]{MR1288752}. Hence, applying this for $k = n$ and $\varphi = J$, 
  \[
  \sup_n\int\!J(|\log(t)|)\,d\mu_{n}(t)< \infty. 
  \]
  We obtain by this way (j). Statement (jj) follows trivially.
  
  We now turn to the proof of the ``in probability'' part of the lemma.
  Arguing as in \cite[Theorem 22]{dellacheriemeyer}, the statement (v) of
  uniform convergence in probability is \emph{equivalent} to the existence for
  all $\delta >0$ of a non-decreasing convex function $J_\delta:\dR_+\to\dR_+$
  such that $\lim_{t \to \infty}J_\delta(t)/t=\infty $, and
  \[
  \sup_n \dP\PAR{\int\!J_\delta(|\log(t)|)\,d\nu_{n}(t)\leq 1} < \delta. 
  \]
  Since $J_\delta $ is non-decreasing  and convex we deduce as above 
  \[
  \int\!J_\delta(|\log(t)|)\,d\mu_{n}(t)\leq \int\!J_\delta(|\log(t)|)\,d\nu_{n}(t).
  \] 
  This proves (j). Statement (jj) is then a consequence of remark \ref{rk:inprob}. 
\end{proof}

\begin{remark}[Logarithmic potential and Cauchy-Stieltjes transform]
  \label{rk:causti}
  We may define the Cauchy-Stieltjes transform $m_\mu:\dC\to\dC\cup\{\infty\}$
  of a probability measure $\mu$ on $\dC$ by
  \[
  m_\mu(z) :=\int_{\dC}\!\frac{1}{\la-z}\,d\mu(\la).
  \]
  Since $1/\ABS{\cdot}$ is Lebesgue locally integrable on $\dC$, the
  Fubini-Tonelli theorem implies that $m_\mu(z)$ is finite for a.a.\
  $z\in\dC$, and moreover $m_\mu$ is locally Lebesgue integrable on $\dC$ and
  thus belongs to $\cD'(\dC)$. Suppose now that $\mu\in\cP(\dC)$. The
  logarithmic potential is related to the Cauchy-Stieltjes transform via the
  identity
  \[
  m_\mu =2\pd U_\mu
  \]
  in $\cD'(\dC)$. In particular, since
  $4\pd\Ol\pd=4\Ol\pd\pd=\Delta$ as operators
  on $\cD'(\dC)$, we obtain, in $\cD'(\dC)$,
  \[
  2\Ol\pd m_\mu=-\Delta U_\mu=-2\pi\mu.
  \]
  Thus we can recover $\mu$ from $m_\mu$. Note that for any $\veps>0$, $m_\mu$
  is bounded on 
  \[
  D_\veps=\{z\in\dC:\DIST(z,\SUPP(\mu))>\veps\}.
  \]
  If $\SUPP(\mu)$ is one-dimensional then one may completely recover $\mu$
  from the knowledge of $m_\mu$ on $D_\veps$ as $\veps\to0$. Note also that
  $m_\mu$ is analytic outside $\SUPP(\mu)$, and is thus characterized by its
  real part or its imaginary part on arbitrary small balls in the connected
  components of $\SUPP(\mu)^c$. If $\SUPP(\mu)$ is not one-dimensional then
  one needs the knowledge of $m_\mu$ inside the support to recover $\mu$. If
  $A\in\cM_n(\dC)$ then $m_{\mu_A}$ is the trace of the resolvent
  \[
  m_{\mu_A}(z)=\TR((A-zI)^{-1})
  \]
  for every $z\in\dC\setminus\{\la_1(A),\ldots,\la_n(A)\}$. For non-Hermitian
  matrices, the lack of a Hermitization identity expressing $m_{\mu_A}$ in
  terms of singular values explains the advantage of the logarithmic potential
  $U_{\mu_A}$ over the Cauchy-Stieltjes transform $m_{\mu_A}$ for spectral
  analysis.
\end{remark}

\begin{remark}[Logarithmic potential and logarithmic energy]\label{rk:logene}
  The term ``logarithmic potential'' comes from the fact that $U_\mu$ is the
  electrostatic potential of $\mu$ viewed as a distribution of charged
  particles in the plane $\dC=\dR^2$ \cite{MR1485778}. The so called
  logarithmic energy of this distribution of charged particles is
  \begin{equation}\label{eq:logene}
    \mathcal{E}(\mu)
    :=\int_{\dC}\!U_\mu(z)\,d\mu(z)
    =-\int_{\dC}\int_{\dC}\!\log\ABS{z-\la}\,d\mu(z)d\mu(\la).
  \end{equation}
  The circular law minimizes $\mathcal{E}(\cdot)$ under a second moment
  constraint \cite{MR1485778}. If $\SUPP(\mu)\subset\dR$ then $\cE(\mu)$
  matches up to a sign and an additive constant the Voiculescu free entropy
  for one variable in free probability theory \cite[Proposition
  4.5]{MR1296352} (see also the formula \ref{eq:ldp-rate}).
\end{remark}  

\begin{remark}[From converging potentials to weak convergence]\label{re:convpot}
  As for the Fourier transform, the pointwise convergence of logarithmic
  potentials along a sequence of probability measures implies the weak
  convergence of the sequence to a probability measure. We need however some
  strong tightness. More precisely, if ${(\mu_n)}_{n\geq1}$ is a sequence in
  $\cP(\dC)$ and if $U:\dC\to(-\infty,+\infty]$ is such that
  \begin{itemize}
  \item [(i)] for a.a.\ $z\in\dC$, $ \lim_{n\to\infty}U_{\mu_n}(z)=U(z)$,
  \item[(ii)] $\log(1+\ABS{\cdot})$ is uniformly integrable for $(\mu_n)_{n
      \geq 1}$,
  \end{itemize} then there exists $\mu \in \cP(\dC)$ such that $U_\mu=U$ a.e.\
  and $\mu = -\frac{1}{2\pi}\Delta U$ in $\cD'(\dC)$ and
  \[
  \mu_n\weak\mu.
  \]
  Let us give a proof inspired from \cite[Proposition 1.3 and Appendix
  A]{MR2191234}. From the de la Vall\'ee Poussin criterion (see e.g.\
  \cite[Theorem 22]{dellacheriemeyer}), assumption (ii) implies that for every
  real number $r\geq1$, there exists a non-decreasing convex function $J
  :\dR_+\to\dR_+$, which may depend on $r$, such that $\lim_{t \to \infty}
  J(t)/t=\infty$, and $J(t)\leq 1+t^2$, and
  \[
  \sup_n \int\!J(\log(r+|\la|))\,d\mu_{n}(\la)<\infty. 
  \]
  Let $K\subset \dC$ be an arbitrary compact set. We choose $r = r(K) \geq 1$
  large enough so that the ball of radius $r-1 $ contains $K$, and therefore
  for every $z\in K$ and $\la\in\dC$,
  \[
  J(|\log\ABS{z-\la}|) %
  \leq (1 + |\log\ABS{z-\la}|^2)\IND_{\{|\la|\leq r\}} %
  +J(\log(r+|\la|))\IND_{\{|\la|>r\}}.
  \]
  The couple of inequalities above, together with the fact that the function
  $(\log\ABS{\cdot})^2$ is locally Lebesgue integrable on $\dC$, imply, by
  using Jensen and Fubini-Tonelli theorems,
  \[
  \sup_n\int_K\!J(|U_n(z)|)\,dxdy
  \leq \sup_n\iint\!\IND_K(z)J(|\log\ABS{z-\la}|)\,d\mu_n(\la)\, dxdy<\infty,
  \]
  where $z=x+iy$ as usual. Since the de la Vall\'ee Poussin criterion is
  necessary and sufficient for uniform integrability, this means that the
  sequence ${(U_{\mu_n})}_{n\geq1}$ is locally uniformly Lebesgue integrable.
  Consequently, from (i) it follows that $U$ is locally Lebesgue integrable
  and that $U_{\mu_n}\to U$ in $\cD'(\dC)$. Since the differential operator
  $\De$ is continuous in $\cD'(\dC)$, we find that $\De U_{\mu_n}\to\De U$ in
  $\cD'(\dC)$. Since $\De U\leq0$, it follows that $\mu:=-\frac{1}{2\pi}\De U$
  is a measure (see e.g.\ \cite{MR717035}). Since for a sequence of measures,
  convergence in $\cD'(\dC)$ implies weak convergence, we get
  $\mu_n=-\frac{1}{2\pi}\De U_{\mu_n}\weak\mu=-\frac{1}{2\pi}\De U$. Moreover,
  by assumptions (ii) we get additionally that $\mu\in\cP(\dC)$. It remains to
  show that $U_\mu=U$ a.e.\ Indeed, for any smooth and compactly supported
  $\vphi:\dC\to\dR$, since the function $\log\ABS{\cdot}$ is locally Lebesgue
  integrable, the Fubini-Tonelli theorem gives
  \[
  \int\!\vphi(z)U_{\mu_n}(z)\,dz%
  =-\int\!\PAR{\int\!\vphi(z)\log|z-w|\,dz}\,d\mu_n(w).
  \]
  Now the function
  $\vphi*\log\ABS{\cdot}:w\in\dC\mapsto\int\!\vphi(z)\log|z-w|\,dz$ is
  continuous and is $\cO(\log\ABS{1+\cdot})$. Therefore, by (i-ii) we get
  $U_{\mu_n}\to U_\mu$ in $\cD'(\dC)$ and thus $U_\mu=U$ in $\cD'(\dC)$ and
  then a.e.
\end{remark}

\subsection{Proof of the circular law}

The proof of theorem \ref{th:circular} is based on the Hermitization lemma
\ref{le:girko}. The part (i) of lemma \ref{le:girko} is obtained from
corollary \ref{cor:DS} below.

\begin{theorem}[Convergence of singular values with additive perturbation]
  \label{th:DS}
  Let $(M_n)_{n \geq 1}$ be a deterministic sequence such that $M_n \in \cM_n
  (\dC)$ for every $n$. If $\nu_{M_n}\weak\rho$ as $n\to\infty$ for some
  probability measure $\rho$ on $\dR_+$ then there exists a probability
  measure $\nu_\rho$ on $\dR_+$ which depends only on $\rho$ and such that
  a.s.\ $\nu_{n^{-1/2}X + M_n}\weak\nu_\rho$ as $n\to\infty$.
\end{theorem}

Theorem \ref{th:DS} appears as a special case of the work of Dozier and
Silverstein for information plus noise random matrices \cite{MR2322123}. Their
proof relies on powerful Hermitian techniques such as truncation,
centralization, trace-resolvent recursion via Schur block inversion, leading
to a fixed point equation for the Cauchy-Stieltjes transform of $\nu_\rho$. It
is important to stress that $\nu_\rho$ does not depend on the law of $X_{11}$
(recall that $X_{11}$ has unit variance). One may possibly produce an
alternative proof of theorem \ref{th:DS} using free probability theory.

\begin{corollary}[Convergence of singular values]
 \label{cor:DS}
 For all $z \in \dC$, there exists a probability measure $\nu_z$ depending
 only on $z$ such that a.s.\ $\nu_{n^{-1/2}X - zI }\weak\nu_z$ as
 $n\to\infty$.
\end{corollary}

For completeness, we will give in sub-section \ref{subsec:proofcongSV} a proof of
corollary \ref{cor:DS}. Note also that for $z=0$, we recover the quarter
circular Marchenko-Pastur theorem \ref{th:quarter-circular}.

It remains to check the uniform integrability assumption (ii) of lemma
\ref{le:girko}. From Markov's inequality, it suffices to show that for all
$z\in\dC$, there exists $p>0$ such that a.s.\
\begin{equation}\label{eq:uibeta}
  \varlimsup_{n\to\infty}\int\!s^{-p}\,d\nu_{n^{-1/2}X-zI}(s) < \infty
  \quad\text{and}\quad
  \varlimsup_{n\to\infty} \int\!s^{p}\,d\nu_{n^{-1/2}X-zI}(s)<\infty.
\end{equation}

The second statement in \eqref{eq:uibeta} with $p\leq 2$ follows from the
strong law of large numbers \eqref{eq:varbound} together with
\eqref{eq:basic1}, which gives $s_i(n^{-1/2}X-zI)\leq s_i(n^{-1/2}X)+|z|$ for
all $1\leq i\leq n$.

The first statement in \eqref{eq:uibeta} concentrates most of the difficulty
behind theorem \ref{th:circular}. In the next two sub-sections, we will prove
and comment the following couple of key lemmas taken from
\cite{tao-vu-cirlaw-bis} and \cite{MR2409368} respectively.

\begin{lemma}[Count of small singular values]\label{le:scount}
  There exist $c_0>0$ and $0<\ga<1$ such a.s.\ for $n\gg1$ and $n^{1-\ga}
  \leq i \leq n-1$ and all $M\in\cM_n(\dC)$,
  \[
  s_{n-i}(n^{-1/2}X+M) \geq c_0 \frac{i}{n}.  
  \]
\end{lemma}

Lemma \ref{le:scount} is more meaningful when $i$ is close to $n^{1-\ga}$. For
$i=n-1$, it gives only a lower bound on $s_1$. The linearity in $i$
corresponds to what we can expect on spacing.

\begin{lemma}[Polynomial lower bound on least singular value]\label{le:sn}
  For every $a,d>0$, there exists $b>0$ such that if $M$ is a deterministic
  complex $n\times n$ matrix with $s_1(M)\leq n^d$ then
  \[
  \dP(s_n(X+M)\leq n^{-b})\leq n^{-a}.
  \]
  In particular there exists $b>0$ which may depend on $d$ such that a.s.\ for
  $n\gg1$,
  \[
  s_n(X+M)\geq n^{-b}.
  \]
\end{lemma}

For ease of notation, we write $s_i$ in place of $s_i(n^{-1/2}X-zI)$. Applying
lemmas \ref{le:scount}-\ref{le:sn} with $M = - z I$ and $M = - z \sqrt n I$
respectively, we get, for any $c>0$, $z\in\dC$, a.s.\ for $n\gg1$,
\begin{align*}
  \frac{1}{n}\sum_{i=1}^ns_i^{-p}
  &\leq 
  \frac{1}{n}\sum_{i=1}^{n-\FLOOR{n^{1-\ga}}}s_i^{-p}
  +\frac{1}{n}\sum_{i=n-\FLOOR{n^{1-\ga}}+1}^ns_i^{-p} \\
  &\leq c_0^{-p}\frac{1}{n}\sum_{i=1}^n\PAR{\frac{n}{i}}^p+2n^{-\ga}n^{bp}.
\end{align*}
The first term of the right hand side is a Riemann sum for
$\int_0^1\!s^{-p}\,ds$ which converges as soon as $0<p<1$. We finally obtain
the first statement in \eqref{eq:uibeta} as soon as $0<p<\min(\ga/b,1)$. Now
the Hermitization lemma \ref{le:girko} ensures that there exists a probability
measure $\mu\in\cP(\dC)$ such that a.s.\ $\mu_{Y}\weak\mu$ as $n\to\infty$ and
for all $z\in\dC$,
\[
U_\mu(z)=-\int_0^\infty\!\log(s)\,d\nu_z(s).
\]
Since $\nu_z$ does not depend on the law of $X_{11}$ (we say that it is then
universal), it follows that $\mu$ also does not depend on the law of $X_{11}$,
and therefore, by using the circular law theorem
\ref{th:complex-ginibre-circle-strong} for the Complex Ginibre Ensemble we
obtain that $\mu$ is the uniform law on the unit disc. Alternatively,
following Pan and Zhou \cite[Lemma 3]{1687963}, one can avoid the knowledge of
the Gaussian case by computing the integral of
$\int_0^\infty\!\log(s)\,d\nu_z(s)$ which should match the formula
\eqref{eq:logpotu} for the logarithmic potential of the uniform law on the
unit disc.

\subsection{Count of small singular values}

This sub-section is devoted to lemma \ref{le:scount} used in the proof of
theorem \ref{th:circular} to check the uniform integrability assumption in
lemma \ref{le:girko}.

\begin{proof}[Proof of lemma \ref{le:scount}]
  We follow the original proof of Tao and Vu \cite{tao-vu-cirlaw-bis}. Up to
  increasing $\gamma$, it is enough to prove the statement for all $2
  n^{1-\ga} \leq i \leq n-1$ for some $\gamma \in (0,1)$ to be chosen later.
  To lighten the notations, we denote by $s_1\geq\cdots\geq s_n$ the singular
  values of $Y:=n^{-1/2}X+M$. We fix $2 n^{1-\ga} \leq i \leq n-1$ and we
  consider the matrix $Y'$ formed by the first $m:=n- \CEIL{i/2}$ rows of
  $\sqrt{n}Y$. Let $s_1'\geq\cdots\geq s_m'$ be the singular values of $Y'$.
  By the Cauchy-Poincar\'e interlacing\footnote{If $A\in\cM_{n}(\dC)$ and
    $1\leq m\leq n$ and if $B\in\cM_{m,n}(\dC)$ is obtained from $A$ by
    deleting $r:=n-m$ rows, then $s_i(A)\geq s_i(B)\geq s_{i+r}(A)$ for every
    $1\leq i\leq m$. In particular, $[s_{m}(B),s_1(B)]\subset[s_n(A),s_1(A)]$,
    i.e.\ the smallest singular value increases while the largest singular
    value is diminished. See \cite[Corollary 3.1.3]{MR1288752}}, we get
  \[
  n ^{-1/2} s'_{n-i} \leq  s_{n-i}
  \]
  Next, by lemma \ref{le:tvneg} we obtain
  \[
  s'^{-2}_1  + \cdots + s'^{-2}_{n -  \CEIL{i/2}} %
  = \DIST_1^{-2}+\cdots+\DIST_{n -\CEIL{i/2}}^{-2}, 
  \]
  where $\DIST_j:=\DIST(R_j,H_j)$ is the distance from the
  $j^\text{th}$ row $R_j$ of $Y'$ to $H_j$, the subspace spanned by the other
  rows of $Y'$. In particular, we have
  \begin{equation}\label{eq:stodist}
    \frac{i}{2n} s^{-2}_{n-i} 
    \leq is'^{-2}_{n-i}
    \leq \sum_{j=n-\CEIL{i}}^{n-\CEIL{i/2}}s_{j}'^{-2}
    \leq \sum_{j=1}^{n-\CEIL{i/2}}\DIST_{j}^{-2}. 
  \end{equation}
  Now $H_j$ is independent of $R_j$ and $\DIM(H_j)\leq n-\frac{i}{2}\leq
  n-n^{1-\ga}$, and thus, for the choice of $\gamma$ given in the forthcoming
  lemma \ref{le:concdist},
  \[
  \sum_{n\gg1}\dP\PAR{\bigcup_{i=2n^{1-\ga}}^{n-1}\bigcup_{j=1}^{n-\CEIL{i/2}}
    \BRA{\DIST_j %
      \leq\frac{\sqrt{i}}{2\sqrt{2}}} }<\infty
  \]
  (note that the exponential bound in lemma \ref{le:concdist} kills the
  polynomial factor due to the union bound over $i,j$). Consequently, by the
  first Borel-Cantelli lemma, we obtain that a.s.\ for $n\gg1$, all $2 n^{1-\ga}
  \leq i \leq n-1$, and all $1 \leq j \leq n - \CEIL{i/2}$,
  \[
  \DIST_j %
  \geq \frac{\sqrt{i}}{2\sqrt{2}} %
  \geq \frac{\sqrt{i}}{4}
  \]
  Finally, \eqref{eq:stodist} gives $s^{2}_{n-i}\geq (i^2)/(32n^2)$, i.e.\ the
  desired result with $c_0 := 1/(4\sqrt{2})$.
\end{proof}

\begin{lemma}[Distance of a random vector to a subspace]\label{le:concdist}
  There exist $\ga>0$ and $\de>0$ such that for all $n\gg1$, $1\leq i\leq
  n$, any deterministic vector $v\in\dC^n$ and any subspace $H$ of $\dC^n$
  with $1 \leq \DIM(H) \leq n - n ^{1-\ga}$, we have, denoting
  $R:=(X_{i1},\ldots,X_{in}) + v$,
  \[
  \dP\PAR{\DIST(R,H) \leq \frac{1}{2}\sqrt {n-\DIM(H)}} %
  \leq \exp(-n^\de).
  \]
\end{lemma}

The exponential bound above is obviously not optimal, but is more than enough
for our purposes: in the proof of lemma \ref{le:scount}, a large enough
polynomial bound on the probability suffices.

\begin{proof}
  The argument is due to Tao and Vu \cite[Proposition 5.1]{tao-vu-cirlaw-bis}.
  We first note that if $H'$ is the vector space spanned by $H$, $v$ and $\dE
  R$, then $\DIM(H') \leq \DIM(H) + 2$ and
  \[
  \DIST(R,H) \geq \DIST(R,H') = \DIST(R',H'), 
  \]
  where $R' := R-\dE(R)$. We may thus directly suppose without loss of
  generality that $v=0$ and that $\dE(X_{ik})=0$. Then, it is easy to check
  that
  \[
  \dE(\DIST(R,H)^2)=n-\DIM(H)
  \]
  (see computation below). The lemma is thus a statement on the deviation
  probability of $\DIST(R,H)$. We first perform a truncation. Let $0 < \veps <
  1/3$. Markov's inequality gives
  \[
  \dP(|X_{ik}|\geq n^\veps) \leq n^{-2\veps}.
  \]
  Hence, from Hoeffding's deviation inequality\footnote{If $X_1,\ldots,X_n$
    are independent and bounded real r.v.\ with $d_i:=\max(X_i)-\min(X_i)$,
    and if $S_n:=X_1+\cdots+X_n$, then $\dP(S_n-\dE S_n\leq tn)\leq
    \exp(-2n^2t^2/(d_1^2+\cdots+d_n^2))$ for any $t\geq0$. See \cite[Th.\
    5.7]{MR1036755}.}, for $n \gg1$,
  \[
  \dP\PAR{\sum_{k=1}^n \IND_{\{|X_{ik}|\leq n^{\veps}\}} < n - n^{1-\veps}} %
  \leq \exp(-2n^{1-2\veps} (1-n^{-\veps})^2) %
  \leq \exp(-n^{1-2\veps}).
  \]
  It is thus sufficient to prove that the result holds by conditioning on 
  \[
  E_m := \{ |X_{i1} |\leq n^{\veps}, \ldots ,|X_{i m }|\leq n^{\veps} \} %
  \quad\text{with}\quad %
  m := \CEIL{n - n^{1-\veps}}.
  \]
  Let $\dE_m[\,\cdot\,]:=\dE[\,\cdot\,|E_m;\cF_m]$ denote the conditional
  expectation given $E_m$ and the filtration $\cF_m$ generated by
  $X_{i,m+1},\ldots,X_{i,n}$. Let $W$ be the subspace spanned by
  \[
  H,\quad %
  u = (0, \ldots, 0 , X_{i,m+1},\ldots,X_{i,n}), \quad %
  w = ( \dE_m[ X_{i1} ] , \ldots ,\dE_m [X_{im } ], 0 , \ldots, 0).
  \]
  Then, by construction $\DIM(W) \leq \DIM(H) + 2$ and $W$ is
  $\cF_m$-measurable. We also have
  \[
  \DIST(R,H) \geq \DIST(R,W) = \DIST(Y, W), 
  \]
  where $Y = (X_{i1} - \lambda, \ldots, X_{im} - \lambda, 0, \ldots,0) = R - u
  - w$ and $\lambda = \dE_m[ X_{i1} ]$. Next we have
  \[
  \si^2:= \dE_m\SBRA{Y_{1}^2} %
  = \dE\SBRA{\PAR{X_{i 1 }-\dE\SBRA{X_{i 1 }\bigm| |X_{i 1 }|\leq n^{\veps}}}^2\Bigm||X_{i 1 }|\leq n^{\veps}} %
  = 1 - o(1).
  \]
  Now, let us consider the disc $D := \{z \in \dC : |z |\leq n^{\veps}\}$ and
  define the $D^m \to \dR_+$ convex function $f : x \mapsto \DIST( (x, 0,
  \ldots, 0) , W)$. From the triangle inequality, $f$ is $1$-Lipschitz:
  \[
    |f(x)-f(x')| \leq \DIST(x,x').
  \]
  We deduce from Talagrand's concentration inequality\footnote{If
    $X_1,\ldots,X_n$ are i.i.d.\ r.v.\ on $D:=\{z\in\dC:|z|\leq r\}$ and if
    $f:D^n\to\dR$ is convex, $1$-Lipschitz, with median $M$, then $\dP
    (|f(X_1,\ldots,X_n)-M| \geq t ) \leq 4 \exp(-\frac{t^2}{16r^2})$ for any
    $t\geq0$. See \cite{MR1361756} and \cite[Cor. 4.9]{MR1849347}.} that
  \begin{equation}\label{eq:distAH}
    \dP_m\PAR{\ABS{\DIST(Y, W) - M_m} \geq t} %
    \leq  4  \exp\PAR{-\frac{t^2}{16n^{2 \veps}}}, 
  \end{equation}
  where $M_m$ is the median of $\DIST(Y, W)$ under $\dE_m$. In
  particular, 
  \[
  M_m \geq \sqrt{ \dE_m \DIST^2 (Y, W) }- c n^{\veps}.
  \]
  Also, if $P$ denotes the orthogonal projection on the orthogonal of $W$, we
  find
  \begin{align*}
    \dE_m\DIST^2(Y, W) 
    & = \sum_{k=1}^m \dE_m \SBRA{ Y_{k}^2 } P_{kk} \\
    & = \sigma^2 \PAR{\sum_{k=1}^n P_{kk} - \sum_{k=m+1}^n P_{kk}} \\
    & \geq \sigma^2 \PAR{n -\DIM(W) -(n-m)} \\
    & \geq \sigma^2 \PAR{n -\DIM(H) -n^{1 - \veps}  - 2} 
  \end{align*}
  We choose some $0 < \gamma < \veps$. Then, from the above expression for any
  $1/2< c < 1$ and $n\gg 1$, $M_m\geq c \sqrt{ n - \DIM(H)} $. We choose
  finally $t = (c - 1/2) \sqrt{ n - \DIM(H)}$ in \eqref{eq:distAH}.
\end{proof}

The following lemma, taken from \cite[Lemma A4]{tao-vu-cirlaw-bis}, is used in
the proof of lemma \ref{le:scount}.

\begin{lemma}[Rows and trace norm of the inverse]\label{le:tvneg}
  Let $1\leq m\leq n$. If $A\in\cM_{m,n}(\dC)$ has full rank, with
  rows $R_1,\ldots,R_{m}$ and $R_{-i}:=\SPAN\{R_j:j\neq i\}$, then
   \[ 
   \sum_{i=1}^{m}s_i(A)^{-2}=\sum_{i=1}^{m}\DIST(R_i,R_{-i})^{-2}.
   \]  
\end{lemma}

\begin{proof}
  The orthogonal projection of $R_i^*$ on the subspace $R_{-i}$ is
  $B^*(BB^*)^{-1}BR_i^*$ where $B$ is the $(m-1)\times n$ matrix obtained from
  $A$ by removing the row $R_i$. In particular, we have
  \[
  \ABS{ R_i}_2^2-\DIST_2(R_i,R_{-i})^2
  =\ABS{ B^*(BB^*)^{-1}BR_i^*}_2^2
  = (BR_i^*)^*(BB^*)^{-1}BR_i^*
  \]
  by the Pythagoras theorem. On the other hand, the Schur block inversion
  formula states that if $M$ is a $m\times m$ matrix then for every partition
  $\{1,\ldots,m\}=I\cup I^c$,
  \begin{equation}\label{eq:schur}
    (M^{-1})_{I,I}=(M_{I,I}-M_{I,I^c}(M_{I^c,I^c})^{-1}M_{I^c,I})^{-1}.
  \end{equation}
  Now we take $M=AA^*$ and $I=\{i\}$, and we note that
  $(AA^*)_{i,j}=R_iR_j^*$, which gives
  \[
  ((AA^*)^{-1})_{i,i}=(R_iR_i^*-(BR_i^*)^*(BB^*)^{-1}BR_i^*)^{-1}
  =\DIST_2(R_i,R_{-i})^{-2}.
  \]
  The desired formula follows by taking the sum over $i\in\{1,\ldots,m\}$.
\end{proof}

\begin{remark}[Local Wegner estimates]
  Lemma \ref{le:scount} says that $\nu_{n^{-1/2}X-zI}([0,\eta])\leq\eta/C$ for
  every $\eta \geq 2 C n^{-\ga}$. In this form, we see that lemma
  \ref{le:scount} is an upper bound on the counting measure
  $n\nu_{n^{-1/2}X-zI}$ on a small interval $[0,\eta]$. This type of estimate
  has already been studied. Notably, an alternative proof of lemma
  \ref{le:scount} can be obtained following the work of \cite{ESY10} on the
  resolvent of Wigner matrices.
\end{remark}

\subsection{Smallest singular value}

This sub-section is devoted to lemma \ref{le:sn} which was used in the proof of
theorem \ref{th:circular} to get the uniform integrability in lemma
\ref{le:girko}. 

The full proof of lemma \ref{le:sn} by Tao and Vu in \cite{MR2409368} is based
on Littlewood-Offord type problems. The main difficulty is the possible
presence of atoms in the law of the entries (in this case $X$ is
non-invertible with positive probability). Regarding the assumptions, the
finite second moment hypothesis on $X_{11}$ is not crucial and can be
considerably weakened. For the sake of simplicity, we give here a simplified
proof when the law of $X_{11}$ has a bounded density on $\dC$ or on $\dR$
(which implies that $X+M$ is invertible with probability one). In lemma
\ref{le:snRV} in Appendix \ref{se:ap:inv}, we prove a general statement of
this type at the price of a weaker probabilistic estimate which is still good
enough to obtain the uniform integrability ``in probability'' required by
lemma \ref{le:girko}.

\begin{proof}[Proof of lemma \ref{le:sn} with bounded density assumption]
  It suffices to show the first statement since the last statement follows
  from the first Borel-Cantelli lemma used with $a>1$. 
  
  For every $x,y\in\dC^n$ and $S\subset\dC^n$, we set $x\cdot
  y:=x_1\Ol{y_1}+\cdots+x_n\Ol{y_n}$ and $\NRM{x}_2:=\sqrt{x\cdot
    x}$ and $\DIST(x,S):=\min_{y\in S}\NRM{x-y}_2$. Let
  $R_1,\ldots,R_n$ be the rows of $X+M$ and set
  \[
  R_{-i}:=\SPAN\{R_j;j\neq i\}
  \]
  for every $1\leq i\leq n$. The lower bound in lemma \ref{le:rvdist} gives
  \[
  \min_{1\leq i\leq n} \DIST(R_i,R_{-i}) \leq \sqrt{n}\,s_n(X+M)
  \]
  and consequently, by the union bound, for any $u\geq0$,
  \[
  \dP(\sqrt{n}\,s_n(X+M)\leq u) %
  \leq n\max_{1\leq i\leq n}\dP( \DIST(R_i,R_{-i})%
  \leq u).
  \]
  Let us fix $1\leq i\leq n$. Let $Y_i$ be a unit vector orthogonal to
  $R_{-i}$. Such a vector is not unique, but we may just pick one which is
  independent of $R_i$. This defines a random variable on the unit sphere
  $\mathbb{S}^{n-1}=\{x\in\dC^n:\NRM{x}_2=1\}$. By the Cauchy-Schwarz
  inequality,
  \[
  |R_i\cdot Y_i|\leq \NRM{\pi_i(R_i)}_2\NRM{Y_i}_2= \DIST(R_i,R_{-i})
  \]
  where $\pi_i$ is the orthogonal projection on the orthogonal complement of
  $R_{-i}$. Let $\nu_i$ be the distribution of $Y_i$ on $\mathbb{S}^{n-1}$.
  Since $Y_i$ and $R_i$ are independent, for any $u\geq0$,
  \[
  \dP( \DIST(R_i,R_{-i})\leq u) %
  \leq \dP(|R_i\cdot Y_i|\leq u) %
  = \int_{\mathbb{S}^{n-1}}\!\!\!\dP(|R_i\cdot y|\leq u)\,d\nu_i(y).
  \]
  Let us assume that $X_{11}$ has a bounded density $\varphi$ on $\dC$. Since
  $\NRM{y}_2=1$ there exists an index $j_0\in\{1,\ldots,n\}$ such that
  $y_{j_0}\neq 0$ with $\ABS{y_{j_0}}^{-1}\leq \sqrt{n}$. The complex random
  variable $R_i\cdot y$ is a sum of independent complex random variables and
  one of them is $X_{ij_0}\,\Ol{y_{j_0}}$, which is absolutely
  continuous with a density bounded above by $\sqrt{n}\,\NRM{\varphi}_\infty$.
  Consequently, by a basic property of convolutions of probability measures,
  the complex random variable $R_i\cdot y$ is also absolutely continuous with
  a density $\varphi_i$ bounded above by $\sqrt{n}\,\NRM{\varphi}_\infty$, and
  thus
  \[
  \dP(|R_i\cdot y|\leq u) %
  = \int_{\dC}\,\IND_{\{|s| \leq u\}}\varphi_i(s)\,ds %
  \leq \pi u^2\,\sqrt{n}\,\NRM{\varphi}_\infty.
  \]  
  Therefore, for every $b>0$, we obtain the desired result 
  (the $O$ does not depend on $M$)
  \[
  \dP(s_n(X+M)\leq n^{-b-1/2}) = O(n^{3/2-2b}).
  \]
  This scheme remains indeed valid in the case where $X_{11}$ has a bounded
  density on $\dR$.
\end{proof}

\begin{lemma}[Rows and operator norm of the inverse]\label{le:rvdist}
  Let $A$ be a complex $n\times n$ matrix with rows $R_1,\ldots,R_n$. Define
  the vector space $R_{-i}:=\SPAN\{R_j:j\neq i\}$. We have then
  \[
  n^{-1/2}\min_{1\leq i\leq n}\DIST(R_i,R_{-i}) \leq
  s_n(A) \leq \min_{1\leq i\leq n}\DIST(R_i,R_{-i}).
  \]
\end{lemma}

\begin{proof}[Proof of lemma \ref{le:rvdist}]
  The argument, due to Rudelson and Vershynin, is buried in \cite{MR2407948}.
  Since $A$ and $A^\top$ have same singular values, one can consider the
  columns $C_1,\ldots,C_n$ of $A$ instead of the rows. For every column vector
  $x\in\dC^n$ and $1\leq i\leq n$, the triangle inequality and the identity
  $Ax=x_1C_1+\cdots+x_n C_n$ give
  \[
  \Vert Ax\Vert_2
  \geq \DIST(Ax,C_{-i}) 
  =\min_{y\in C_{-i}} \NRM{ Ax-y}_2 
  =\min_{y\in C_{-i}} \NRM{ x_iC_i-y}_2 
  =\vert x_i\vert\DIST(C_i,C_{-i}).
  \]
  If $\NRM{ x}_2 =1$ then necessarily $\vert x_i\vert \geq
  n^{-1/2}$ for some $1\leq i\leq n$ and therefore
  \[
  s_n(A) %
  =\min_{\NRM{ x}_2 =1}\NRM{ Ax}_2 %
  \geq n^{-1/2}\min_{1\leq i\leq n}\DIST(C_i,C_{-i}).
  \]
  Conversely, for every $1\leq i\leq n$, there exists a vector $y$ with
  $y_i=1$ such that
  \[
  \DIST(C_i,C_{-i}) %
  =\NRM{ y_1C_1+\cdots+y_nC_n}_2 %
  =\NRM{ Ay}_2 %
  \geq \NRM{ y}_2 %
  \min_{\NRM{ x}_2=1}\NRM{ Ax}_2 %
  \geq s_n(A)
  \]
  where we used the fact that $\Vert y\Vert^2_2 = |y_1|^2+\cdots+|y_n|^2\geq
  |y_i|^2=1$.
\end{proof}

\begin{remark}[Assumptions for the control of the smallest singular value]\label{re:lesn}
  In the proof of lemma \ref{le:sn} with the bounded density assumption, we
  have not used the assumption on the second moment of $X_{11}$ nor the
  assumption on the norm of $M$.
\end{remark}

\subsection{Convergence of singular values measure}
\label{subsec:proofcongSV}

This sub-section is devoted to corollary \ref{cor:DS}. The proof is divided
into five steps.

\subsubsection*{Step One: Concentration of singular values measure}

First, it turns out that it is sufficient to prove the convergence to $\nu_z$
of $\dE \nu_{n^{-1/2}X -z}$. Indeed, for matrices with independent rows, there
is a remarkable concentration of measure phenomenon. More precisely, recall
that the total variation norm of $f:\dR \to\dR$ is defined as
\[
\NRM{f}_\textsc{TV}:=\sup \sum_{k \in \dZ} | f(x_{k+1})-f(x_k) |, \, 
\]
where the supremum runs over all sequences $(x_k)_{k \in \dZ}$ such that
$x_{k+1} \geq x_k$ for any $k \in \dZ$. If $f = \IND_{(-\infty,s]}$ for some
$s\in\dR$ then $\NRM{f}_\textsc{TV}=1$, while if $f$ has a derivative in
$L^1(\dR)$, $\NRM{f}_\textsc{TV}=\int\!|f'(t)|\,dt.$ The following lemma is
extracted from \cite{bordenave-caputo-chafai-heavygirko}, see also
\cite{MR2535081}.

\begin{lemma}[Concentration for the singular values empirical
  measure]\label{le:concspec}
  If $M$ is a $n\times n$ complex random matrix with independent rows (or
  with independent columns) then for any $f:\dR \to\dR$ going to $0$ at $ \pm \infty$ with
  $\NRM{f}_\textsc{TV}\leq1$ and every $t\geq0$,
  \[
  \dP\PAR{\ABS{\int\!f\,d\nu_M -\dE\int\!f\,d\nu_M} \geq t} %
  \leq 2 \exp\PAR{- 2 n t^2}.
  \]
\end{lemma}

It is worth to mention that if $M$ has independent entries which satisfy a
uniform sub-Gaussian behavior, then for all Lipschitz function, the
concentration of $\int\!f\,d\nu_M$ has a rate $n^2$ and not $n$, see e.g.\ the
work of Guionnet and Zeitouni \cite{MR1781846}.

\begin{proof}
  If $A,B\in\cM_n(\dC)$ and if $F_A(\cdot):=\nu_A((-\infty,\cdot))$ and
  $F_B(\cdot):=\nu_B((-\infty,\cdot))$ are the cumulative distribution
  functions of the probability measures $\nu_A$ and $\nu_B$ then it is easily
  seen from the Lidskii inequality for singular
  values\footnote{\label{fn:finite-rank-interlacing} If $A,B\in\cM_{n}(\dC)$
    with $\RANK(A-B)\leq k$, then $s_{i-k}(A) \geq s_i(B) \geq
    s_{i+k}(A)$ for any $1\leq i\leq n$ with the convention $s_i\equiv\infty$
    if $i<1$ and $s_i\equiv0$ if $i>n$. This allows the extremes to blow. See
    \cite[Th.\ 3.3.16]{MR1288752}.} that
  \[
  \NRM{ F_A- F_B}_\infty\leq\frac{\RANK(A-B)}{n}.
  \]
  Now for a smooth $f:\dR\to\dR$, we get, by integrating by parts,
  \[
  \ABS{\int\!f\,d\nu_A-\int\!f\,d\nu_B} 
  =\ABS{\int_\dR \!f'(t)(F_A(t)-F_B(t))\,dt} 
  \leq \frac{\RANK(A-B)}{n}\int_\dR\!|f'(t) |\,dt.
  \]
  Since the left hand side depends on at most $2n$ points, we get, by
  approximation, for every measurable function $f:\dR\to\dR$
  with $\NRM{ f}_{\mathrm{TV}} \leq 1$,
  \begin{equation}\label{eq:IPPinterlacing}
  \ABS{\int\!f\,d\nu_A-\int\!f\,d\nu_{B}} %
  \leq \frac{\RANK(A-B)}{n}.
  \end{equation}
  From now on, $f:\dR\to\dR$ is a fixed measurable function with
  $\NRM{ f}_{\mathrm{TV}}\leq 1$. For every row vectors
  $x_1,\ldots,x_n$ in $\dC^n$, we denote by $A(x_1,\ldots,x_n)$ the
  $n\times n$ matrix with rows $x_1,\ldots,x_n$ and we define
  $F:(\dC^n)^n\to\dR$ by
  \[
  F(x_1,\ldots,x_n):=\int\!f\,d\mu_{A(x_1,\ldots,x_n)}.
  \]
  For any $i\in\{1,\ldots,n\}$ and any row vectors $x_1,\ldots,x_n,x_i'$ of
  $\dC^n$, we have
  \[
  \RANK(A(x_1,\ldots,x_{i-1},x_i,x_{i+1},\ldots,x_n)
  -A(x_1,\ldots,x_{i-1},x_i',x_{i+1},\ldots,x_n))
  \leq 1
  \]
  and thus
  \[
  |F(x_1,\ldots,x_{i-1},x_i,x_{i+1},\ldots,x_n)
    -F(x_1,\ldots,x_{i-1},x_i',x_{i+1},\ldots,x_n)|
    \leq \frac{1}{n}.
  \]
  Finally, the desired result follows from the McDiarmid-\-Azuma-\-Hoeffding
  concentration inequality for bounded differences\footnote{If
    $X_1,\ldots,X_n$ are independent r.v. in $\cX_1,\ldots,\cX_n$ and if $f :
    \cX_1\times\cdots\times\cX_n \to \dR$ is a measurable function then
    $\dP\PAR{|f(X_1,\ldots,X_n)-\dE f(X_1,\ldots,X_n)| \geq t} \leq
    2\exp(-2t^2/(c_1^2+\cdots+c_n^2))$ for any $t \geq 0$, where
    $c_k:=\sup_{x,x'\in\cD_k}|f(x)-f(x')|$ and $\cD_k:=\{(x,x'):x_i=x'_i\text{
      for all $i\neq k$}\}$. We refer to McDiarmid \cite{MR1036755}.} applied
  to the function $F$ and to the random variables $R_1,\ldots,R_n$ (the rows
  of $M$).
\end{proof}

\subsubsection*{Step Two: Truncation and centralization}

In the second step, we prove that it is sufficient to prove the convergence
for entries with bounded support. More precisely, we define
\[
Y_{ij} = X_{ij} \IND_{\{| X_{ij} |\leq \kappa\}},
\]
where $\kappa = \kappa_n$ is a sequence growing to infinity. Then if $Y =
(Y_{ij} )_{ 1 \leq i , j \leq n}$, we have from Hoffman-Wielandt inequality
\eqref{eq:Hoffman-Wielandt},
\[
\frac{1}{n}\sum_{k=1}^n\ABS{s_k(n^{-1/2}Y-zI)-s_k(n^{-1/2}X-zI)}^2%
\leq \frac{1}{n^2}\sum_{1 \leq i,j \leq n}
|X_{ij}|^2\IND_{\{|X_{ij}|>\kappa\}}.
\]
By assumption $\dE | X_{ij} |^2 \IND_{\{|X_{ij}|>\kappa\}}$ goes to $0$ as
$\kappa$ goes to infinity. Hence, by the law of large numbers, the right hand
side of the above inequality converges a.s.\ to $0$. On the left hand side we
recognize the square of the Wasserstein $W_2$ coupling distance\footnote{The
  $W_2$ distance between two probability measures $\eta_1,\eta_2$ on $\dR$ is
  $W_2(\eta_1,\eta_2):=\inf\dE(|X_1-X_2|^2)^{1/2}$ where the inf runs over the
  set of r.v.\ $(X_1,X_2)$ on $\dR\times\dR$ with $X_1\sim\eta_1$ and
  $X_2\sim\eta_2$. In the case where $\eta_1=\frac{1}{n}\sum_{i=1}^n\de_{a_i}$
  with $0\leq a_i\nearrow$ and $\eta_2=\frac{1}{n}\sum_{i=1}^n\de_{b_i}$ with
  $0\leq b_i\nearrow$ then
  $W_2(\eta_1,\eta_2)^2=\frac{1}{n}\sum_{i=1}^n(a_i-b_i)^2$.} between
$\nu_{n^{-1/2} Y -zI}$ and $\nu_{n^{-1/2} X -zI }$. Since the convergence in
$W_2$ distance implies weak convergence, we deduce that it is sufficient to
prove the convergence of $\dE \nu_{n^{-1/2} Y - z I }$ to $\nu_z$.

Next, we turn to the centralization by setting
\[
Z_{ij} = Y_{ij} - \dE Y_{ij} %
= Y_{ij} - \dE X_{11}\IND_{\{|X_{11}|\leq\kappa\}}.
\]
Then if $Z = (Z_{ij} )_{ 1 \leq i , j \leq n}$, we have from the Lidskii
inequality for singular values,
\[
\max_{t > 0} \ABS{\nu_{n^{-1/2} Y - z I }([0,t])-\nu_{n^{-1/2} Z - z I }([0,t])}%
\leq \frac{\RANK(Y - Z)}{n} %
\leq \frac{1}{n}.
\]
In particular, it is sufficient to prove the convergence of
$\dE\nu_{n^{-1/2}Z-zI}$ to $\nu_z$.

In summary, in the remainder of this sub-section, we will allow the law of
$X_{11}$ to depend on $n$ but we will assume that
\begin{equation}\label{eq:propX11}
  \dE X_{11} = 0  \, , %
  \quad \dP ( |X_{11} |\geq \kappa_n )  = 0 %
  \quad \text{and} \quad \dE |X_{11} |^2   = \sigma_n^2,
\end{equation}
where $\kappa = \kappa_n = o ( \sqrt n )$ is a sequence growing to infinity and
$\sigma=\sigma_n$ goes to $1$ as $n$ goes to infinity.

\subsubsection*{Step Three: Linearization}

We use a popular linearization technique: we remark the identity of the
Cauchy-Stieltjes transform, for $\eta \in \dC_+$,
\begin{equation}\label{eq:resstieljes}
  m_{\check \nu_{n^{-1/2}  X  - zI}} (\eta) =  \frac 1 {2n} \TR ( H (z)  - \eta I)^{-1}, 
\end{equation}
where $\check \nu ( \cdot) = (\nu ( \cdot) + \nu ( - \cdot) ) /2$ is the
symmetrized version of a measure $\nu$, and
\[
H(z) :=
\begin{pmatrix} 
  0 & n^{-1/2}  X - z \\ 
  ( n^{-1/2}  X- z)^* &  0  
\end{pmatrix}.
\]
Through a permutation of the entries, this matrix $H(z)$ is equivalent to the
matrix 
\[
B(z) = B - q (z,0) \otimes I_n
\]
where
\[
q(z,\eta) := 
\begin{pmatrix} \eta & z \\ \bar z  & \eta \end{pmatrix}
\] 
and for every $1\leq i,j\leq n$,
\[
B_{ij} :=  \frac{1}{\sqrt n}   
\begin{pmatrix} 0 & X_{ij} \\ \bar X_{ji}  & 0 \end{pmatrix}.
\]
Note that $B (z) \in \cM_n ( \cM_{2} (\dC)) \simeq \cM_{2n} ( \dC)$ is
Hermitian and its resolvent is denoted by
\[
R(q) = (B(z) -  \eta I_{2n} )^{-1}  = (B - q (z,\eta)  \otimes I_n )^{-1} .
\]
Then $R(q)\in\cM_n(\cM_{2}(\dC))$ and, by \eqref{eq:resstieljes}, we deduce
that
\[
m_{\check \nu_{n^{-1/2}  X-zI}} (\eta)  =   \frac 1 {2n} \TR R (q).
\]
We set 
\[
R(q)_{kk} = 
\begin{pmatrix} 
  a_{k} (q) & b _{k} (q) \\ 
  c _{k} (q)& d _{k} (q)
\end{pmatrix}
\in \cM_{2} (\dC).
\]
It is easy to check that
\begin{equation}\label{eq:a=d}
  a(q) := \frac 1 n \sum_{k= 1}^n  a _{k} (q) %
  = \frac{1}{n}\sum_{k= 1}^nd_{k} (q)%
  \quad \text{and  } \quad %
  b (q) := \frac{1}{n}\sum_{k= 1}^nb _{k} (q) %
  = \frac 1 n \sum_{k= 1}^n  \bar c_{k}(q),
\end{equation}
(see the forthcoming lemma \ref{le:propRes}). So finally,
\begin{equation} \label{eq:resstieljes2}
  m_{\check \nu_{n^{-1/2}  X  - zI}} (\eta)  =  a (q).
\end{equation}
Hence, in order to prove that $\dE \nu_{n^{-1/2} X - zI}$ converges, it is
sufficient to prove that $\dE a(q)$ converges to, say, $\alpha(q)$ which, by
tightness, will necessarily be the Cauchy-Stieltjes transform of a symmetric
measure.

\subsubsection*{Step Four: Approximate fixed point equation}

We use a resolvent method to deduce an approximate fixed point equation
satisfied by $a(q)$. Schur's block inversion \eqref{eq:schur} gives
\[
R_{nn} = \PAR{ \frac { 1}{\sqrt n} 
  \begin{pmatrix} 
    0 & X_{nn} \\ 
    \bar X_{nn} & 0 
  \end{pmatrix} - q - Q^* \Wt R Q } ^{-1},
\]
where $Q \in  \cM_{n-1, 1} ( \cM_{2} (\dC))$, 
\[
Q_ i = \frac{1}{\sqrt n}
\begin{pmatrix}  0 & X_{n i } \\ \bar X_{i n} & 0 \end{pmatrix}
\]
and, with $\Wt B = (B_{ij})_{1 \leq i,j \leq n-1}$, $\Wt B (z) =
\Wt B - q (z, 0) \otimes I_{n-1}$,
\[
\Wt R %
= ( \Wt B - q \otimes I_{n-1}) ^{-1} %
= ( \Wt B (z) - \eta I_{2(n-1)}) ^{-1}
\]
is the resolvent of a minor. We denote by $\cF_{n-1}$ the smallest
$\sigma$-algebra spanned by the variables $(X_{ij})_{1 \leq i ,j \leq n-1}$.
We notice that $\Wt R$ is $\cF_{n-1}$-measurable and is independent of $Q$. If
$\dE_n [\,\cdot\,] := \dE [\,\cdot\,|\cF_{n-1}]$, we get, using
\eqref{eq:propX11} and \eqref{eq:a=d}
\begin{align*}
  \dE_n \SBRA{Q^* \Wt R  Q }  
  & = \sum_{1 \leq k,\ell\leq n-1}\dE_n\SBRA{Q_k^*\Wt R_{k \ell}Q_\ell}\\
  & = \frac {\sigma^2}{ n} \sum_{k = 1}^{n-1} %
  \begin{pmatrix} \Wt a_{k} & 0\\ 0 & \Wt d_{k} \end{pmatrix}\\
  & = \frac {\sigma^2}{ n} \sum_{k = 1}^{n-1} %
  \begin{pmatrix} \Wt a_{k} & 0\\ 0 & \Wt a_{k} \end{pmatrix},
\end{align*}
where 
\[
\Wt R_{kk} = 
\begin{pmatrix} 
  \Wt a_{k} & \Wt b_{k} \\ 
  \Wt c_{k} & \Wt d_{k} %
\end{pmatrix}.
\]
Recall that $\Wt B(z)$ is a minor of $B(z)$. We may thus use
interlacing as in \eqref{eq:IPPinterlacing} for the function $f = (\cdot -
\eta)^{-1}$, and we find
\[
\ABS{\sum_{k = 1}^{n-1} \Wt a_{k} - \sum_{k = 1}^{n} a_{k}} %
\leq 2 \int_{\dR} \frac{1}{|x - \eta |^2} dx %
= O \PAR{\frac{1}{\Im(\eta)}}.
\]
Hence, we have checked that
\[
\dE_n \SBRA{Q^* \Wt R  Q }  %
=   \begin{pmatrix}   a  &  0    \\  0  &   a  \end{pmatrix} + \veps_1, 
\]
with $\NRM{\veps_1}_2=o(1)$ (note here that $q(z , \eta) $ is fixed). Moreover,
we define
\[
\veps_2:=
\dE_n \SBRA{\PAR{Q^* \Wt R  Q-  \dE_n \SBRA{Q^* \Wt R  Q}}^*  %
  \PAR{Q^* \Wt R Q- \dE_n \SBRA{Q^* \Wt R Q}}}. %
\]
Since $\NRM{\Wt R}_2 \leq \Im (\eta) ^{-1}$, 
\[
\NRM{\Wt R^*_{ii} \Wt R_{ii} }_2 \leq  \Im ( \eta)^{-2} \quad \text{and } \quad \TR \left( \sum_{i,j} \Wt R^*_{ij} \Wt R_{ji} \right)  = \TR  \left( \Wt R^* \Wt R  \right)   \leq  2 n \Im ( \eta)^{-2}.
\] 
Also, by \eqref{eq:propX11} 
\[
\dE  | X^2_{ij} - {\sigma}^2 |^2 \leq 2 \kappa^2 {\sigma}^2.   
\]
Then, an elementary
computation gives 
\[
\NRM{\veps_2}_2= O \PAR{ \frac {\kappa^2}{n \Im(\eta)^2 }} = o(1).
\]
Also, we note by lemma \ref{le:concspec} that $a(q)$ is close to its
expectation: 
\[
\dE|a(q)-\dE a(q)|^2 = O \PAR{ \frac {1}{n \Im(\eta)^2 }}= o(1).
\]
Thus, the matrix
\[
D =   \frac{1}{\sqrt n} 
\begin{pmatrix} 0 & X_{nn} \\ \bar X_{nn} & 0 \end{pmatrix} %
- Q^* \Wt R Q %
+ \dE \begin{pmatrix} a & 0 \\ 0 & a \end{pmatrix}
\]
has a norm which converges to $0$ in expectation as $n\to\infty$. Now, we use
the identity
\[
R_{nn} + \PAR{q+\dE\begin{pmatrix} a & 0\\ 0 & a\end{pmatrix}}^{-1} %
= R_{nn} \, D \, \PAR{q+\dE\begin{pmatrix} a & 0 \\ 0 & a \end{pmatrix}}^{-1}.
\]
Hence, since the norms of $\PAR{ q + \dE \begin{pmatrix} a & 0 \\ 0 &
    a \end{pmatrix} } ^{-1}$ and $R_{nn}$ are at most $\Im (\eta)^{-1}$,
we get
\[
\dE R_{nn} %
= -\PAR{q+\dE\begin{pmatrix} a & 0 \\ 0 & a \end{pmatrix}}^{-1}+\veps
\]
with $\NRM{\veps}_2= o(1)$. In other words, using exchangeability,
\[
\dE \begin{pmatrix} a & b \\ \bar b & a \end{pmatrix}   %
= -\PAR{q + \dE \begin{pmatrix} a  &  0 \\ 0 & a \end{pmatrix}}^{-1} + \veps. 
\]

\subsubsection*{Step Five: Unicity of the fixed point equation}

From what precedes, any accumulation point of $\dE \begin{pmatrix} a & b \\
  \bar b & a \end{pmatrix}$ is solution of the fixed point equation
\begin{equation}\label{eq:FPCircular}
  \begin{pmatrix} \alpha & \beta\\ \bar \beta & \alpha \end{pmatrix} %
  = -   \PAR{q+\begin{pmatrix} \alpha & 0\\ 0 & \alpha\end{pmatrix}}^{-1},
\end{equation}
with $\alpha = \alpha (q) \in \dC_+$. We find
\[
\alpha =\frac{ \alpha+ \eta}{| z|^2 - ( \alpha + \eta)^2 }.
\]
Hence, $\alpha$ is a root of a polynomial of degree $3$. Hence, to conclude the proof of corollary \ref{cor:DS}, it is sufficient to prove that there is unique symmetric measure whose Cauchy-Stieltjes transform is solution of this fixed point equation. For any $\eta \in
\dC_+$, it is simple to check that this equation has a unique solution in
$\dC_+$ which can be explicitly computed. Alternatively, we know from
\eqref{eq:resstieljes2} and Montel's theorem that $\eta\in\dC_+\mapsto \alpha
(q(z, \eta))\in\dC_+$ is analytic. In particular, it is sufficient to check
that there is a unique solution in $\dC_+$ for $\eta = it$, with $t >0$. To
this end, we also notice from \eqref{eq:resstieljes2} that $\alpha (q) \in i
\dR_+$ for $q = q(z,it)$. Hence, if $h(z,t) = \Im (\alpha (q ))$, we find
\[
h = \frac{ h + t }{| z|^2 + ( h  + t)^2 }.
\]
Thus, $h  \ne 0$ and 
\[
1 = \frac{ 1 + t  h ^{-1} }{| z|^2 + ( h  + t)^2 }.
\]
The right hand side in a decreasing function in $h$ on $(0, \infty)$ with
limits equal to $+ \infty$ and $0$ at $h \to 0$ and $h \to \infty$. Thus,
there is a unique solution of the above equation.

We have thus proved that $\dE \begin{pmatrix} a & b \\ \bar b &
  a \end{pmatrix}$ converges. The proof of corollary \ref{cor:DS} is over.

\subsection{The quaternionic resolvent: an alternative look at the circular law}
\label{subsec:convres}
\subsubsection*{Motivation}

The aim of this sub-section is to develop an efficient machinery to analyze the
spectral measures of a non-Hermitian matrix which avoids a direct use of the
logarithmic potential and the singular values. This approach is built upon
methods in the physics literature, e.g.\
\cite{FZ97,Gudowska-Nowak,rogerscastillo,rogers2010}. As we will see, it is a
refinement of the linearization procedure used in the proof of corollary
\ref{cor:DS}. Recall that the Cauchy-Stieltjes transform of a measure $\nu$ on
$\dR$ is defined, for $\eta \in \dC_+$, as
\[
m_{\nu} ( \eta) = \int_{\dR}\!\frac{1}{x-\eta}\,d\nu(x).
\]
The Cauchy-Stieltjes transform characterizes every probability measure on
$\dR$, and actually, following Remark \ref{rk:causti}, every probability
measure on $\dC$. However, if the support of the measure is not
one-dimensional, then one needs the knowledge of the Cauchy-Stieltjes
transform inside the support, which is not convenient. For a probability
measure on $\dC$, it is tempting to define a quaternionic Cauchy-Stieltjes
transform. For $q\in\dH_+$, where
\[
\dH_+ := \BRA{ 
  \begin{pmatrix} 
    \eta & z \\ 
    \bar z & \eta
  \end{pmatrix}, 
  z \in \dC, \eta \in \dC_+},
\]
we would define
\[
M_\mu(a) = \int_{\dC}\!%
\PAR{ %
  \begin{pmatrix} 
    0 & \la \\ 
    \bar \la & 0 
  \end{pmatrix} 
  - q}^{-1}\,d\mu(\lambda) \; \in \; \dH_+.
\]
This transform characterizes the measure: in $\cD'(\dC)$, 
\[
\lim_{t \downarrow 0} (\pd M_\mu ( q ( z, it) )_{12} = - \pi \mu,
\]
where $\pd$ is as in \eqref{eq:ddbar} and
\[
q(z,\eta)  :=  %
\begin{pmatrix} 
  \eta & z \\ 
  \bar z & \eta
\end{pmatrix}.
\]
If $A \in \cM_n (\dC)$ is normal then $M_{\mu_A}$ can be recovered from the
trace of a properly defined quaternionic resolvent. If $A$ is not normal, the
situation is however more delicate and needs a more careful treatment.

\subsubsection*{Definition of quaternionic resolvent} 

For further needs, we will define this quaternionic resolvent in any Hilbert
space. Let $H$ be an Hilbert space with inner product $\ANG{\cdot,\cdot}$. We
define the Hilbert space $H_2 = H \times \dZ/2\dZ$. For $x = (y, \veps) \in
H_2$, we set $\hat x = ( y, \veps +1 )$. In particular, this transform is an
involution $\hat {\hat x} = x$. There is the direct sum decomposition $H_2 =
H_0 \oplus H_1$ with $H_ \veps = \{ x = (y, \veps) : y \in H \}$.

Let $A$ be a linear operator defined on a dense domain $D(A) \subset H$. This
operator can be extended to an operator on $D(A) \otimes \dZ/2\dZ$ by setting
$A x= (a y, \veps)$, for all $x = (y,\veps) \in D(A) \otimes \dZ/2\dZ$ (in
other word we extend $A$ by $A \otimes I_2$). We define the operator $B$ in
$D(B) = D(A) \otimes \dZ/2\dZ$ by
\[
B x = 
\begin{cases}
  A^* \hat x & \text{if $x \in H_0$} \\
  A \hat{ x} & \text{if $x \in H_1$}.
\end{cases}
\]
For $x \in H$, if $\Pi_x :  H_2  \to \dC^2$ denotes the orthogonal
projection on $((x,0), ( x,1))$, for $x, y\in D(A)$, we find
\[ 
B_{xy} := \Pi_x B \Pi^*_y %
=   
\begin{pmatrix} 0 & \ANG{x,Ay} \\ \ANG{x,A^*y} & 0 \end{pmatrix} \in\cM_{2}(\dC).
\]
The operator $B$ will be called the \emph{bipartized operator} of $A$, it is
an Hermitian operator (i.e.\ for all $x, y \in D(B)$, $\ANG{Bx,y}=
\ANG{x,By}$). If $B$ is essentially self-adjoint (i.e.\ it has a unique
self-adjoint extension), we may define the quaternionic resolvent of $A$ for
all $q \in \dH_+$ as
\[
R_A (q) = (B -  I_{H} \otimes q )^{-1} 
\]
Indeed, if $q= q(z, \eta)$, we note that $R_A$ is the usual resolvent at
$\eta$ of the essentially self-adjoint operator $B (z) = B - I_{H} \otimes
q(z,0)$. Hence $R_A$ inherits the usual properties of resolvent operators
(analyticity in $\eta$, bounded norm). We define
\[
R_A (q)_{x y} := \Pi_x R_A(q)
\Pi^*_y. 
\]
If $H$ is separable and $(e_i)_{i \geq 1}$ is a canonical orthonormal basis of
$H$, we simply write $R_{i j}$ instead of ${R_A(q)}_{e_i e_j}$, $i, j \in V$.
Finally, if $A \in \cM_{n}( \dC)$, we set
\[
\Gamma_A (q)  = \frac 1 n  \sum_{k=1}^n  R_A(q)_{kk}.  
\]
If $A$ is normal then it can be checked that $R(q)_{kk} \in \dH_+$ and 
$
\Gamma_A (q) = M_{\mu_A}(q). 
$
However, if $A$ is not normal, this formula fails to hold. However, the next
lemma explains how to recover anyway $\mu_A$ from the resolvent.

\begin{lemma}[From quaternionic transform to spectral measures]
  \label{le:propRes}
  Let $A \in \cM_n(\dC)$ and $q = q(z,\eta) \in \dH_+$. Then,
  \[\Gamma_A (q)  = \begin{pmatrix}
   a (q) & b (q)   \\  \bar b (q)   &  a (q) 
  \end{pmatrix} \in \dH_+. \] 
  Moreover, \[
 m_{\check \nu_{ A - z } } (\eta)  =  a(q) 
  \]
  and, in $\cD' (\dC)$,  
  \[
  \mu_A = - \frac 1 \pi  \lim_{q ( z, it ) : t \downarrow 0} \pd  b  (q). 
  \]
\end{lemma}

\begin{proof}
  For ease of notation, assume that $z = 0$ and set $\tau (\cdot) = \frac 1 n
  \TR (\cdot)$. If $P$ is the permutation matrix associated to the permutation
  $\sigma (2 k - 1) = k$, $\sigma (2 k) = n +k$, we get
  \[
  (B - I_{H} \otimes q )^{-1} %
  = P^* \begin{pmatrix}
    -\eta  & A  \\
    A^* & - \eta
  \end{pmatrix}^{-1}  P 
  =  
  - 
 P   \begin{pmatrix} 
    \eta ( \eta ^2 - AA ^*)^{-1}   
    & A ( \eta ^2 - A^* A)^{-1}  \\ 
    A^*  ( \eta ^2 - AA ^*)^{-1} 
    &   \eta  ( \eta ^2 - A^* A)^{-1} 
  \end{pmatrix}P. 
  \]
  Hence, 
  \[
    \Gamma_A ( q ) =   - 
    \begin{pmatrix} 
  \eta \tau ( \eta ^2 - AA ^*)^{-1}     
      &  \tau \PAR{ A ( \eta ^2 - A^* A)^{-1}}  \\ 
      \tau  \PAR{A^*  ( \eta ^2 -AA ^*)^{-1} }  
      &     \eta \tau ( \eta ^2 - A^* A)^{-1} 
    \end{pmatrix}.
  \]
  Notice that 
  \begin{align*}
    m_{\check \nu_{A} } (\eta) %
    &= \frac{1}{2}\int\!\frac{1}{x-\eta}-\frac{1}{x+\eta}\,d\nu_{A}(x)\\
    &= \int\!\frac{\eta}{x^2-\eta^2}\,d\nu_{A}(x) \\
    &= \int\!\frac{\eta}{x-\eta^2}\,d\mu_{AA^*}(x) \\
    &= \eta \tau (AA^* - \eta ^2 )^{-1}.
  \end{align*}
  Note also that $\mu_{A^* A } = \mu_{AA^*}$ implies that 
  \[
  \tau ( \eta ^2 - AA^*)^{-1} = \tau ( \eta ^2 -A^* A)^{-1}.
  \]
  Finally, since $\tau$ is a trace,
  \begin{align*}
    \tau \PAR{A ( \eta ^2 - A^* A)^{-1}} 
    &= \tau \PAR{ ( \eta ^2 - A^* A)^{-1} A } \\
    &= \Ol {\tau \PAR{ ( \eta ^2 - A^* A)^{-1}A }^* } \\
    &= \Ol{\tau  \PAR{A^*  ( \eta ^2 - AA ^*)^{-1} }}.
  \end{align*}
  Applying the above to $A-z$, we deduce the first two statements.
  
  For the last statement,  we write
  \[  
  \int\!\log|s + it|\,d\nu_{A - z }(s) %
  = \frac{1}{2}\int\!\log(s^2+t^2)\,d\nu_{A-z}(s) %
  = \frac{1}{2} \tau\log((A-z)(A-z)^*+t^2).
  \]
  Hence, from Jacobi formula (see remark \ref{rk:causti} for the definition of
  $\pd$ and $\Ol\pd$)
  \begin{align*}
    \Ol \pd\int\!\log|s+it|\,d\nu_{A-z}(s) 
    & = \frac{1}{2} %
    \tau\PAR{((A-z)(A-z)^*+t^2)^{-1}\Ol\pd((A -z )(A-z)^*+t^2)} \\
    & = -\frac{1}{2}\tau\PAR{((A-z)(A-z)^*+t^2)^{-1}(A-z)} \\
    & = -\frac{1}{2} b(q(z,it)).
  \end{align*}
  The function $\int\!\log|s+it|\,d\nu_{A-z}(s)$ decreases monotonically to
  \[
  \int\!\log (s)\,d\nu_{A-z}(s) = - U_{\mu_A}(z)
  \]
  as $t \downarrow 0$. Hence, in distribution,
  \[
  \mu_A = %
  \lim_{t \downarrow 0} \frac{2}{\pi} \pd \Ol \pd %
  \int\!\log|s+it|\,d\nu_{A-z}(s). 
  \]
  The conclusion follows.
\end{proof}

\subsubsection*{Girko's Hermitization lemma revisited}

There is a straightforward extension of Girko's lemma \ref{le:girko} that uses
the quaternionic resolvent.

\begin{lemma}[Girko Hermitization]\label{le:girkoQUAT}
  Let $(A_n)_{n\geq1}$ be a sequence of complex random matrices defined on a
  common probability space where $A_n$ takes its values in $\cM_n(\dC)$. Assume that for all $q \in \dH_+$, there exists
  \[
  \Gamma (q) =    \begin{pmatrix} a (q) & b(q) \\ \bar b(q) &
    a(q) \end{pmatrix}   \; \in \; \dH_+
  \]  
  such that for a.a.\ $z\in\dC$, $\eta  \in \dC_+$, with $q = q (z,\eta)$,
  \begin{itemize}
  \item[(i')] a.s.\ (respectively in probability) $\Gamma_{A_n} (q)$ converges to $\Gamma(q)$ as
    $n\to\infty$
  \item[(ii)] a.s.\ (respectively in probability) $\log$ is uniformly integrable for
    $\PAR{\nu_{A_n-zI}}_{n\geq1}$ 
  \end{itemize}
  Then there exists a probability measure $\mu\in\mathcal{P}(\dC)$ such
  that
  \begin{itemize}
  \item[(j)] a.s.\ (respectively in probability) $\mu_{A_n}\weak\mu$ as $n\to\infty$ 
  \item[(jj')] in $\cD'(\dC)$, 
    \[
    \mu  = - \frac 1 \pi \lim_{ q ( z, it) : t \downarrow 0} \pd b (q).
    \]
  \end{itemize}
\end{lemma}

Note that, by lemma \ref{le:propRes}, assumption (i') implies assumption
(i) of lemma \ref{le:girko} : the limit probability measure $\nu_z$ is
characterized by 
\[ 
m_{\check \nu_{z}} (\eta) = a(q).
\]
The potential interest of lemma \ref{le:girkoQUAT} lies in the formula for
$\mu$. It avoids any use the logarithmic potential.

\subsubsection*{Concentration}

The quaternionic resolvent enjoys a simple concentration inequality, exactly
as for the empirical singular values measure.

\begin{lemma}[Concentration for the quaternionic resolvent]\label{le:concQUAT}
  If $A$ is a random matrix in $\cM_n(\dC)$ with independent rows (or columns)
  then for any $q = q(z,\eta) \in \dH_+$ and $t\geq0$,
  \[
  \dP\PAR{\NRM{\Gamma_A (q) -\dE \Gamma_A (q)}_2 \geq t} %
  \leq 2 \exp\PAR{- \frac{n \Im (\eta)^2 t^2}{8}}.
  \]
\end{lemma}

\begin{proof}
  Let $M, N \in \cM_n ( \dC)$ with bipartized matrices $B, C \in \cM_{2n} (
  \dC)$. We have
  \begin{equation}\label{eq:diffQUAT}
    \NRM{\Gamma_M (q)-\Gamma_N (q)}_2 %
    \leq \frac{4\,\RANK(M-N)}{n\Im(\eta)}. 
  \end{equation}
  Indeed, from the resolvent identity, for any $q \in \dH_+$, 
  \[
  D = R_{M} ( q) - R_{N} (q) = R_M ( q) ( C - B ) R_N (q).  
  \]
  It follows that $D$ has rank $r \leq \RANK (B -C) = 2 \, \RANK(M-N) $. Also,
  recall that the operator norm of $D$ is at most $2 \Im (\eta) ^{-1}$. Hence,
  in the singular values decomposition
  \[
  D = \sum_{i =1} ^r s_i u_i v_i^* 
  \]
  we have $s_i \leq 2 \Im (\eta) ^{-1}$. If $\Pi_k : \dC^{2n} \to \dC^2$ is
  the orthogonal projection on $\SPAN\{e_{2k-1},e_{2k}\}$, then
  \[
  \Gamma_M (q) - \Gamma_N (q) %
  = \frac 1 n \sum_{k=1}^n \Pi_k D \Pi^*_k %
  = \frac 1 n \sum_{i=1}^r s_i \sum_{k=1}^n ( \Pi_k u_i ) ( \Pi_k v_i ) ^*.
  \] 
  Using Cauchy-Schwartz inequality, 
  \[
  \NRM{\Gamma_M (q) - \Gamma_N (q)}_2 %
  \leq \frac{1}{n} \sum_{i=1}^r s_i %
  \sqrt{\PAR{\sum_{k=1}^n %
      \NRM{\Pi_k u_i}_2^2}\PAR{\sum_{k=1}^n\NRM{\Pi_k v_i}_2^2}} %
  = \frac 1 n \sum_{i=1}^r s_i .
  \]
  We obtain precisely \eqref{eq:diffQUAT}. The remainder of the proof is now
  identical to the proof of lemma \ref{le:concspec}: we express $\Gamma_A (q)
  -\dE \Gamma_A (q)$ has a sum of bounded martingales difference.
 \end{proof}

\subsubsection*{Computation for the circular law}

As pointed out in \cite{rogerscastillo}, the circular law is easily found from
the quaternionic resolvent. Indeed, using lemma \ref{le:concQUAT} and the
proof of corollary \ref{cor:DS}, we get, for all $q \in \dH_+$, a.s.\
\[
\lim_{n \to \infty} \Gamma_{n ^{-1/2} X } (q ) %
= \Gamma (q) %
= %
\begin{pmatrix} 
  \alpha (q) & \beta (q) \\ \bar \beta (q) & \alpha(q) 
\end{pmatrix},
\]
where, from \eqref{eq:FPCircular},
\[
\Gamma = - \PAR{ q + \mathrm{diag}( \Gamma ) } ^{-1} \;\footnote{This equation
  is the analog of the fixed point equation satisfied by the Cauchy-Stieltjes
  transform $m$ of the semi circular law: $m (\eta) = - ( \eta + m (\eta)
  )^{-1}$.}.
\]
In the proof of corollary \ref{cor:DS}, we have checked that for $\eta = it$,
$\alpha (q) = i h (z,t) \in i \dR_+$ where
\[
1 = \frac{ 1 + t  h ^{-1} }{| z|^2 + ( h  + t)^2 }.
\]
We deduce easily that 
\[
\lim_{t \downarrow 0} h(z,t) %
= \begin{cases}
  \sqrt {  1- | z|^2 } & \text{if $|z|\leq 1$} \\
  0 & \text{otherwise.}
  \end{cases}
\]
Then, from
\[
\beta ( q  )  =    \frac{- z}{| z|^2 - ( a(q) + \eta )^2 }, 
\]
we find 
\[
\lim_{ q ( z, it) : t \downarrow 0}  \beta ( q )  %
= %
\begin{cases}
  - z & \text{if $|z|\leq 1$} \\
  0  &  \text{otherwise.} 
\end{cases} 
\]
As lemma \ref{le:girkoQUAT} dictates, if we compose by $-\pi^{-1} \pd$ we
retrieve the circular law.

\section{Related results and models}
\label{se:related}

\subsection{Replacement principle and universality for shifted matrices}

It is worthwhile to state the following lemma, which can be seen as a variant
of the Hermitization lemma \ref{le:girko}. The next statement is slightly
stronger than its original version in \cite[Theorem 2.1]{tao-vu-cirlaw-bis}.

\begin{lemma}[Replacement principle]\label{le:replacement}
  Let $(A_n)_{n\geq1}$ and $(B_n)_{n\geq1}$ be two sequences where $A_n$ and
  $B_n$ are random variables in $\cM_n(\dC)$. If for a.a.\ $z \in \dC$, a.s.
  \begin{itemize}  
  \item[(k)] $\lim_{n\to\infty}  U_{\mu_{A_n}} (z)   -U_{\mu_{B_n}} (z) =0$
  \item[(kk)] $\log ( 1 + \cdot) $ is uniformly integrable for $( \nu_{A_n })_{n \geq 1}$ and
    $(\nu_{B_n })_{n \geq 1}$
    \end{itemize}
  then a.s.\ $\mu_{A_n}-\mu_{B_n}\weak 0$ as $n\to\infty$.
\end{lemma}

A proof of the lemma follows by using the argument in the proof of lemma
\ref{le:girko}. Using their replacement principle, Tao and Vu have proved in
\cite{tao-vu-cirlaw-bis} that the universality of the limit spectral measures
of random matrices goes far beyond the circular law. We state it here in a
slightly stronger form than the original version, see \cite{ECP2011-10}.

\begin{theorem}[Universality principle for shifted matrices]
  \label{th:universality}
  Let $X$ and $G$ be the random matrices considered in sections
  \ref{se:gaussian} and \ref{se:universal} obtained from infinite tables with
  i.i.d.\ entries. Consider a deterministic sequence $(M_n)_{n \geq 1}$ such
  that $M_n \in \cM_n (\dC)$ and for some $p > 0$,
  \[
  \varlimsup_{n \to \infty} \int\!s^p\,d\nu_{M_n}(s) < \infty.
  \]
  Then a.s.\ $\mu_{n^{-1/2}X+M_n}-\mu_{n^{-1/2}G+M_n}\weak0$ as $n\to\infty$.
  \end{theorem}

\subsection{Related models}

We give a list of models related to the circular law theorem
\ref{th:circular}.

\subsubsection*{Sparsity} 

The circular law theorem \ref{th:circular} may remain valid if one allows the
entries law to depend on $n$. This extension contains for instance sparse
models in which the law has an atom at $0$ with mass $p_n\to1$ at a certain
speed, see \cite{gotze-tikhomirov-new,MR2409368,wood}.

\subsubsection*{Outliers}

The circular law theorem \ref{th:circular} allows the blow up of an arbitrary
(asymptotically negligible) fraction of the extremal eigenvalues. Indeed, it
was shown by Silverstein \cite{MR1284550} that if $\dE(|X_{11}|^4)<\infty$ and
$\dE(X_{11})\neq0$ then the spectral radius $|\la_1(n^{-1/2}X)|$ tends to
infinity at speed $\sqrt{n}$ and has a Gaussian fluctuation. This observation
of Silverstein is the base of \cite{djalil-nccl}, see also the ideas of Andrew
\cite{MR1062321}. More recently, Tao studied in \cite{tao-outliers} the
outliers produced by various types of perturbations including general additive
perturbations.

\subsubsection*{Sum and products}

The scheme of proof of theorem \ref{th:circular} (based on Hermitization,
logarithmic potential, and uniform integrability) turns out to be quite
robust. It allows for instance to study the limit of the empirical
distribution of the eigenvalues of sums and products of random matrices, see
\cite{ECP2011-10}, and also \cite{gotze-tikhomirov-product} in relation with
Fuss-Catalan laws. We may also mention \cite{orourke-soshnikov}. The crucial
step lies in the control of the small singular values.

\subsubsection*{Cauchy and the sphere}

It is well known that the ratio of two independent standard real Gaussian
variables is a Cauchy random variable, which has heavy tails. The complex
analogue of this phenomenon leads to a complex Cauchy random variable, which
is also the image law by the stereographical projection of the uniform law on
the sphere. The matrix analogue consists in starting from two independent
copies $G_1$ and $G_2$ of the Complex Ginibre Ensemble, and to consider the
random matrix $Y=G_1^{-1}G_2$. The limit of $\mu_Y$ was analyzed by Forrester
and Krishnapur \cite{forrester-krishnapur}. Note that $Y$ does not have
i.i.d.\ entries.

\subsubsection*{Random circulant matrices}

The eigenvalues of a non-Hermitian circulant matrix are linear functionals of
the matrix entries. Meckes \cite{meckes} used this fact together with the
central limit theorem in order to show that if the entries are i.i.d.\ with
finite positive variance then the scaled empirical spectral distribution of
the eigenvalues tends to a Gaussian law. We can imagine a heavy tailed version
of this phenomenon with $\alpha$-stable limiting laws.

\subsubsection*{Single ring theorem}

Let $D\in\cM_n(\dR_+)$ be a random diagonal matrix and $U,V\in\cM_n(\dC)$ be
two independent Haar unitary matrices, independent of $D$. The law of
$X:=UDV^*$ is unitary invariant by construction, and $\nu_X=\mu_D$ (it is a
random SVD). Assume that $\mu_D$ tends to some limiting law $\nu$ as
$n\to\infty$. It was conjectured by Feinberg and Zee \cite{MR1477381} that
$\mu_X$ tends to a limiting law which is supported in a centered ring of the
complex plane, i.e.\ a set of the form $\{z\in\dC:r\leq|z|\leq R\}$. Under
some additional assumptions, this was proved by Guionnet, Krishnapur, and
Zeitouni \cite{guionnet-krishnapur-zeitouni} by using the Hermitization
technique and specific aspects such as the Schwinger-Dyson non-commutative
integration by parts. Guionnet and Zeitouni have also obtained the convergence
of the support in a more recent work \cite{guionnet-zeitouni}. The Complex
Ginibre Ensemble is a special case of this unitary invariant model. Very
recently, Khoruzhenko discovered a new and relatively simple model
(quadratized rectangular Ginibre matrix) which gives rise to a single ring.

\subsubsection*{Large deviations and logarithmic potential with external field}

The circular law theorem \ref{th:complex-ginibre-circle-strong} for the
Complex Ginibre Ensemble can be seen as a special case of the circular law
theorem for unitary invariant random matrices with eigenvalues density
proportional to 
\[
(\la_1,\ldots,\la_n)\mapsto
\exp\PAR{-\frac{1}{2n}\sum_{i=1}^nV(\la_i)}\prod_{i<j}|\la_i-\la_j|^2
\]
where $V:\dC\mapsto\dR$ is a smooth potential growing enough at infinity.
Since
\[
\exp\PAR{-\frac{1}{2n}\sum_{i=1}^nV(\la_i)} \prod_{i<j}|\la_i-\la_j|^2 %
=\exp\PAR{-\frac{1}{2n}\sum_{i=1}^nV(\la_i)+\frac{1}{2}\sum_{i<j}\log|\la_i-\la_j|}
\]
we discover an empirical version of the logarithmic energy functional
$\cE(\cdot)$ defined in \eqref{eq:logene} penalized by the ``external''
potential $V$. Indeed, it has been shown by Hiai and Petz \cite{MR1606719}
(see also Ben Arous and Zeitouni \cite{MR1660943}) that the Complex Ginibre
Ensemble satisfies a large deviations principle at speed $n^2$ for the weak
topology on the set of symmetric probability measures (with respect to
conjugacy), with good rate function given by
\begin{equation}\label{eq:ldp-rate}
  \mu\mapsto\frac{1}{2}\PAR{\cE(\mu)+\int\!V\,d\mu}-\frac{3}{8}
  =\frac{1}{4}\iint\!(V(z)+V(\la)-2\log|\la-z|)\,d\mu(z)d\mu(\la)-\frac{3}{8}.
\end{equation}
This rate function achieves its minimum $0$ at point $\mu=\cC_1$. This is
coherent with the fact that the circular law $\cC_1$ is the minimum of the
logarithmic energy among the probability measures on $\dC$ with fixed
variance, see the book of Saff and Totik \cite{MR1485778}. Note that this
large deviations principle gives an alternative proof of the circular law for
the Ginibre Ensemble thanks to the first Borel-Cantelli lemma.


\subsubsection*{Dependent entries}

According to Girko, in relation to his ``canonical equation K20'', the
circular law theorem \ref{th:circular} remains valid for random matrices with
independent rows provided some natural hypotheses \cite{MR1887675}. A circular
law theorem is available for random Markov matrices including the Dirichlet
Markov Ensemble \cite{cirmar}, and random matrices with i.i.d.\ log-concave
isotropic rows\footnote{An absolutely continuous probability measure on
  $\dR^n$ is log-concave if its density is $e^{-V}$ with $V$
  convex.\label{fn:logconc}} \cite{EJP2011-37}. Another Markovian model
consists in a non-Hermitian random Markov generator with i.i.d.\ off-diagonal
entries, which gives rise to a new limiting spectral distribution, possibly
not rotationally invariant, which can be interpreted using free probability
theory, see \cite{bordenave-caputo-chafai-cirgen}. Yet another model related
to projections in which each row has a zero sum is studied in
\cite{tao-outliers}. To end up this tour, let us mention another kind of
dependence which comes from truncation of random matrices with depend entries
such as Haar unitary matrices. Namely, let $U$ be distributed according to the
uniform law on the unitary group $\dU_n$ (we say that $U$ is Haar unitary).
Dong, Jiang, and Li have shown in \cite{dong-jiang-li} that the empirical
spectral distribution of the diagonal sub-matrix $(U_{ij})_{1\leq i,j\leq m}$
tends to the circular law if $m/n\to0$, while it tends to the arc law (uniform
law on the unit circle $\{z\in\dC:|z|=1\}$) if $m/n\to1$. Other results of the
same flavor can be found in \cite{MR2480790}.

\subsubsection*{Tridiagonal matrices}

The limiting spectral distributions of random tridiagonal Hermitian matrices
with i.i.d.\ entries are not universal and depend on the law of the entries,
see \cite{MR2480789} for an approach based on the method of moments. The
non-Hermitian version of this model was studied by Goldsheid and Khoruzhenko
\cite{MR2191234} by using the logarithmic potential. Indeed, the tridiagonal
structure produces a three terms recursion on characteristic polynomials which
can be written as a product of random $2\times 2$ matrices, leading to the
usage of a multiplicative ergodic theorem to show the convergence of the
logarithmic potential (which appears as a Lyapunov exponent). In particular,
neither the Hermitization nor the control the smallest and small singular
values are needed here. Indeed the approach relies directly on remark
\ref{re:convpot}. Despite this apparent simplicity, the structure of the
limiting distributions may be incredibly complicated and mathematically
mysterious, as shown on the Bernoulli case by the physicists Holz, Orland, and
Zee \cite{MR1986425}.

\subsection{Free probability interpretation}


As we shall see, the circular law and its extensions have an interpretation in
free probability theory, a sub-domain of operator algebra theory. Before going
further, we should recall briefly certain classical notions of operator
algebra. We refer to Voiculescu, Dykema and Nica \cite{MR1217253} for a
complete treatment of free non-commutative variables, see also the book by
Anderson, Guionnet, and Zeitouni for the link with random matrices
\cite{MR2760897}. In the sequel, $H$ is an Hilbert space and we consider a
pair $(\cM, \tau)$ where $\cM$ is an algebra of bounded operators on $H$,
stable by the adjoint operation $*$, and where $\tau : \cM \to \dC$ is a
linear map such that $\tau (1) = 1$, $\tau ( a a^* ) = \tau (a^* a) \geq 0$.

\subsubsection*{Definition of Brown measure} For $a \in \cM$, define $|a |=
\sqrt { a a^*}$. For $b$ self-adjoint element in $\cM$, we denote by $\mu_b$
the spectral measure of $b$: it is the unique probability measure on the real
line satisfying, for any integer $k \in \dN$,
\[
\tau(b^k) = \int\!t^kd\mu_b(t). 
\]
Also, if $a \in \cM$, we define
\[
\nu_a = \mu_{|a|}. 
\]
Then, in the spirit of \eqref{eq:UESD}, the Brown measure \cite{MR866489} of
$a \in \cM$ is the unique probability measure $\mu_a$ on $\dC$, which
satisfies for almost all $z \in \dC$,
\[
\int\!\log|z-\la|\,d\mu_a(\la) = \int\!\log(s)\,d\nu_{a-z}(s).
\]
In distribution, it is given by the formula\footnote{The quantity
  $\displaystyle{\exp \int\!\log(t)\,d\mu_{|a|}(t)}$ is known as the
  Fuglede-Kadison determinant of $a \in \cM$, see \cite{MR0052696}.}
\begin{equation} \label{eq:defbrown} %
  \mu_a = \frac{1}{2\pi} \Delta \int\!\log(s)\,d\nu_{a-z}(s).
\end{equation}
The fact that the above definition is indeed a probability measure requires a
proof, which can be found in \cite{MR2339369}. Our notation is consistent:
first, if $a$ is self-adjoint, then the Brown (spectral) measure coincides
with the spectral measure. Secondly, if $\cM = \cM_n(\dC)$ and $\tau :=
\frac{1}{n} \TR$ is the normalized trace on $\cM_n(\dC)$, then we retrieve our
usual definition for $\nu_A$ and $\mu_A$. It is interesting to point out that
the identity \eqref{eq:defbrown} which is a consequence of the definition of
the eigenvalues when $\cM = \cM_n(\dC)$ serves as a definition in the general
setting of von Neumann algebras.

Beyond bounded operators, and as explained in Brown \cite{MR866489} and in
Haagerup and Schultz \cite{MR2339369}, it is possible to define, for a class
$\bar \cM \supset \cM$ of closed densely defined operators affiliated with
$\cM$, a probability measure on $\dC$ called the Brown spectral measure of $a
\in \bar\cM$.

\subsubsection*{Failure of the method of moments}

For non-Hermitian matrices, the spectrum does not necessarily belong to the
real line, and in general, the limiting spectral distribution is not supported
in the real line. The problem here is that the moments are not enough to
characterize laws on $\dC$. For instance, if $Z$ is a complex random
variable following the uniform law $\cC_\kappa$ on the centered disc
$\{z\in\dC;|z|\leq\kappa\}$ of radius $\kappa$ then for every $r\geq0$,
$\mathbb{E}(Z^r)=0$ and thus $\cC_\kappa$ is not characterized by its
moments. Any rotational invariant law on $\dC$ with light tails shares
with $\cC_\kappa$ the same sequence of null moments. One can try to
circumvent the problem by using ``mixed moments'' which uniquely determine
$\mu$ by the Weierstrass theorem. Namely, for every $A\in\cM_n(\dC)$, if
$A=UTU^*$ is the Schur unitary triangularization of $A$ then for every
integers $r,r'\geq0$ and with $z=x+iy$ and $\tau=\frac{1}{n}\TR$,
\[
\int_{\dC}\!z^r\Ol{z}^{r'}\,d\mu_A(z) %
=\frac 1 n \sum_{i=1}^n\lambda_i^r(A)\Ol{\lambda_i(A)}^{r'} %
=\tau(T^r\Ol{T}^{r'}) %
\neq \tau(T^rT^{*r'}) %
=\tau(A^rA^{*r'}).
\]
Indeed equality holds true when $\Ol{T}=T^*$, i.e.\ when $T$ is diagonal,
i.e.\ when $A$ is normal. This explains why the method of moments looses its
strength for non-normal operators. To circumvent the problem, one may think
about using the notion of $\star$-moments. Note that if $A$ is normal then for
every word $A^{\veps_1}\cdots A^{\veps_k}$ where
$\veps_1,\ldots,\veps_n\in\{1,*\}$, we have $\tau(A^{\veps_1}\cdots
A^{\veps_k})=\tau(A^{k_1}A^{*k_2})$ where $k_1,k_2$ are the number of
occurrence of $A$ and $A^*$.

\subsubsection*{$\star$-distribution} 
The $\star$-distribution of $a \in \cM$ is the collection of all its
$\star$-moments: 
\[
\tau (a^{\veps_1} a^{\veps_2} \cdots a^{\veps_n}),
\]
where $n\geq1$ and $\veps_1,\ldots,\veps_n\in\{1,*\}$. The element $c \in \cM$
is {\em circular} when it has the $\star$-distribution of $( s_1 + i s_2)
/\sqrt{2} $ where $s_1$ and $s_2$ are free semi circular variables with
spectral measure of Lebesgue density
$x\mapsto\frac{1}{\pi}\sqrt{4-x^2}\IND_{[-2,2]}(x)$.

The $\star$-distribution of $a\in\cM$ allows to recover the moments of
$|a-z|^2=(a-z)(a-z)^*$ for all $z\in\dC$, and thus $\nu_{a-z}$ for all
$z\in\dC$, and thus the Brown measure $\mu_a$ of $a$. Actually, for a random
matrix, the $\star$-distribution contains, in addition to the spectral
measure, an information on the eigenvectors of the matrix.

We say that a sequence of matrices $(A_n)_{n\geq1}$ where $A$ takes it values
in $\cM_n( \dC)$ converges in $\star$-moments to $a \in \cM$, if all
$\star$-moments converge to the $\star$-moments of $a \in \cM$. For example,
if $G \in\cM_n(\dC)$ is our complex Ginibre matrix, then a.s.\ as $n \to
\infty$, $n^{-1/2}G$ converges in $\star$-moments to a circular element.

\subsubsection*{Discontinuity of the Brown measure}
Due to the unboundedness of the logarithm, the Brown measure $\mu_a$ depends
discontinuously on the $\star$-moments of $a$ \cite{MR1876844,MR1929504}.
The limiting measures are perturbations by ``balayage''.
A simple counter example is given by the matrices of example \ref{ex:nilpot}.
For random matrices, this discontinuity is circumvented in the Girko
Hermitization by requiring a uniform integrability, which turns out to be
a.s.\ satisfied the random matrices $n^{-1/2}X$ in the circular law theorem
\ref{th:circular}.

However, \'Sniady \cite[Theorem 4.1]{MR1929504} has shown that it is always
possible to regularize the Brown measure by adding an additive noise. More
precisely, if $G$ is as above and $(A_n)_{n\geq1}$ is a sequence of matrices
where $A_n$ takes its values in $\cM_n(\dC)$, and if the $\star$-moments of
$A_n$ converge to the $\star$-moments of $a \in \cM$ as $n\to\infty$, then
a.s.\ $n \to \infty$ $\mu_{A_n+ t n^{-1/2} G}$ converges to $\mu_{a + t c}$,
$c$ is circular element free of $a$. In particular, by choosing a sequence
$t_n$ going to $0$ sufficiently slowly, it is possible to regularize the Brown
measure: a.s.\ $\mu_{A_n + t_n n^{-1/2} G}$ converges to $\mu_{a}$. Note that
the universality theorem \ref{th:universality} shows that the same result
holds if we replace $G$ by our matrix $X$. We refer to Ryan \cite{MR1624843}
and references therein for the analysis of the convergence in $\star$-moments.
See also the forthcoming book of Tao \cite{tao-topics-in-random-matrix}.
The \'Sniady theorem was revisited recently by Guionnet, Wood, and Zeitouni
\cite{guionnet-wood-zeitouni}.

\section{Heavy tailed entries and new limiting spectral distributions}\label{se:heavy}

This section is devoted to the study of the analogues of the quarter circular
and circular law theorems \ref{th:quarter-circular}-\ref{th:circular} when
$X_{11}$ has an infinite variance (and thus heavy tails). The approach taken
from \cite{bordenave-caputo-chafai-heavygirko} involves many ingredients
including the Hermitization of section \ref{se:universal}. To lighten the
notations, we often abridge $A-zI$ into $A-z$ for an operator or matrix $A$
and a complex number $z$.

\begin{figure}[htbp]
  \begin{center}
    \includegraphics[scale=0.63]{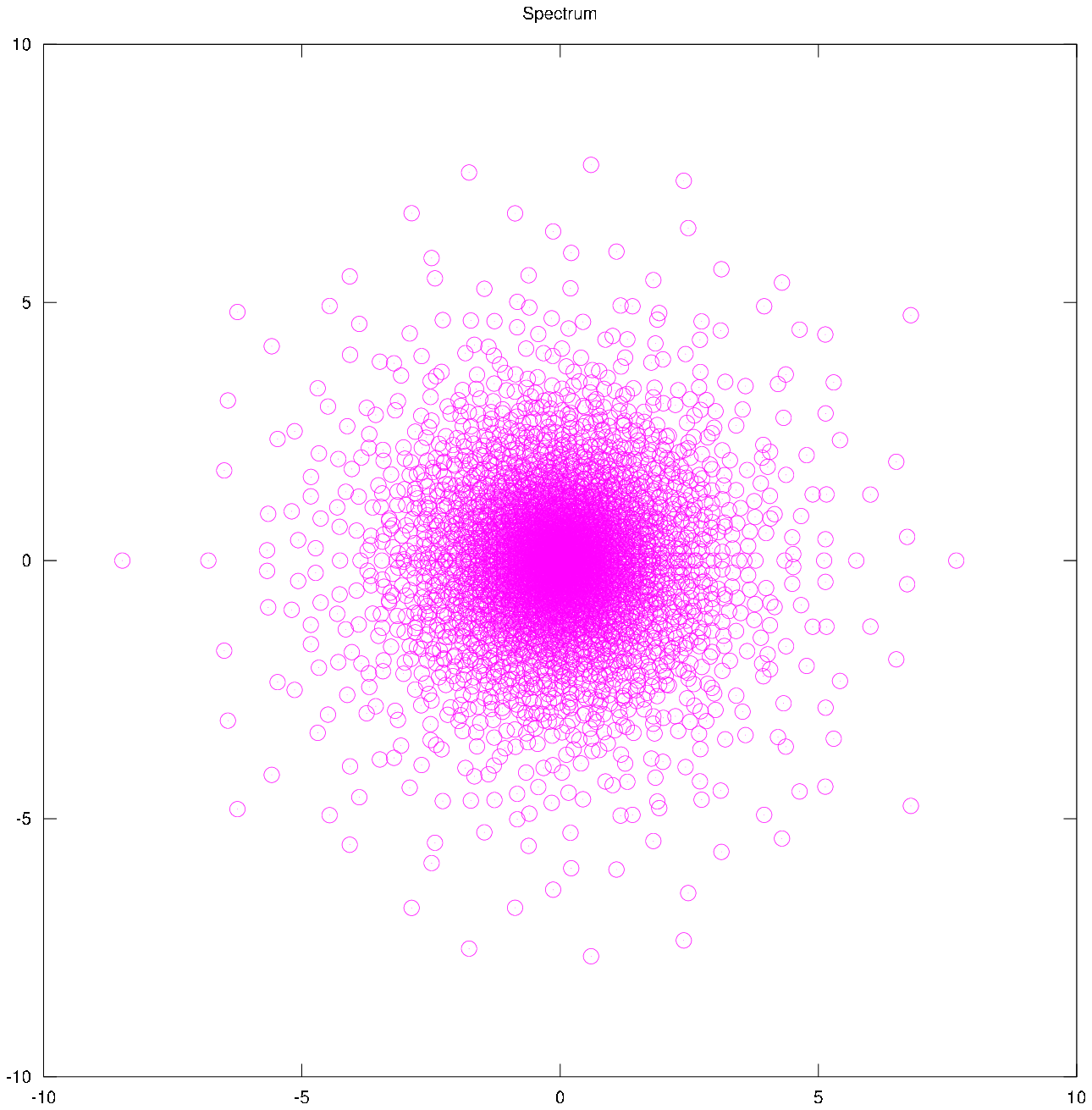}
    \includegraphics[scale=0.5]{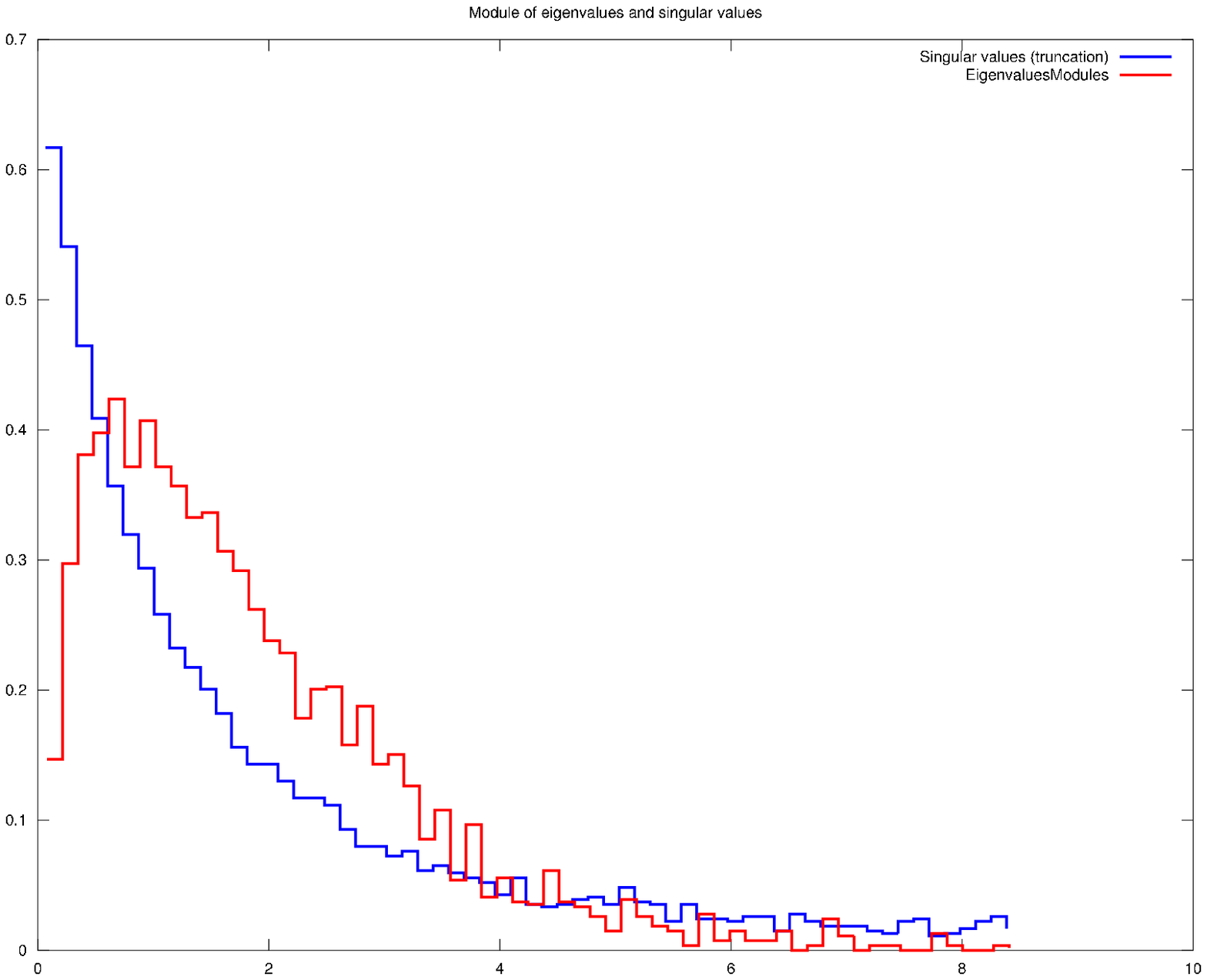} 
  \caption{The upper plot shows the spectrum of a single $n\times n$ matrix
    $n^{-1/\alpha}X$ with $n=4000$ and i.i.d.\ heavy tailed entries with
    $X_{11}\overset{d}{=}\veps U^{-1/\alpha}$ with $\alpha=1$ and $U$ uniform
    on $[0,1]$ and $\veps$ uniform on $\{-1,1\}$ independent of $U$. The lower
    plot shows the histogram of the singular values (blue) and the histogram
    of the module of the eigenvalues (red) of this random matrix. The singular
    values vector is trimmed to avoid extreme values.}
  \label{fi:heavy}
  \end{center}
\end{figure}

\subsection{Heavy tailed analogs of quarter circular and circular laws}

We now come back to an array $X := (X_{ij})_{1 \leq i , j \leq n}$ of i.i.d.\
random variables on $\dC$. We lift the hypothesis that the entries have a
finite second moment: we will assume that,
\begin{itemize}
  \item for some $0 < \alpha < 2$,
    \begin{equation}\label{eq:tail}
      \lim_{ t \to \infty} t^{\alpha} \dP ( | X_{11} | \geq t ) = 1,
    \end{equation}
  \item as $t \to \infty$, the conditional probability
    \[
    \dP\PAR{\frac{X_{11}}{|X_{11}|}\in \cdot\,\bigm|\,|X_{11}|\geq t}
    \]
    converges to a probability measure on the unit circle
    $\dS^{1}:=\{z\in\dC:|z|=1\}$.
\end{itemize}
The law of the entries belongs then to the domain of attraction of an
$\alpha$-stable law. An example is obtained when $|X_{11}|$ and
$X_{11}/|X_{11}|$ are independent with $|X_{11}|=|S|$ where $S$ is real
symmetric $\alpha$-stable. Another example is given by $X_{11}=\varepsilon
W^{-1/\alpha}$ with $\varepsilon$ and $W$ independent such that $\varepsilon$
is supported in $\dS^{1}$ while $W$ is uniform on $[0,1]$.

The interest on this type of random matrices has started with the work of the
physicists Bouchaud and Cizeau \cite{BouchaudCizeau}. One might think that the
analog of the Ginibre ensemble is a matrix with i.i.d.\ $\alpha$-stable
entries. It turns out that this random matrix ensemble is not unitary
invariant and there is no explicit expression for the distribution of its
eigenvalues. This lack of comparison with a canonical ensemble makes the
analysis of the limit spectral measures more delicate. We may first wonder
what is the analog of the quarter circular law theorem
\ref{th:quarter-circular}. This question has been settled by Belinschi, Dembo
and Guionnet \cite{BDG09} (built upon the earlier work of Ben Arous and
Guionnet \cite{BG08}).

\begin{theorem}[Singular values of heavy tailed random
  matrices]\label{th:HTquartercircle}
  There exists a probability measure $\nu_\alpha$ on $\dR_+$ such that a.s.\ 
  $\nu_{n^{-1/\alpha}X}\weak\nu_\alpha$ as $n\to\infty$.
\end{theorem}

This probability measure $\nu_\alpha$ depends only on $\alpha$. It does not
have a known explicit closed form but has been studied in
\cite{BG08,bordenave-caputo-chafai-ii,BDG09}. We know that $\nu_\alpha$ has a
bounded continuous density $f_\alpha$ on $\dR_+$, which is analytic on some
neighborhood of $\infty$. The explicit value of $f_\alpha (x)$ is only known
for $x = 0$. But, more importantly, we have
\[
\lim_{ t \to \infty}  t^{\alpha +1 } f_\alpha ( t ) = \alpha. 
\]
In particular, $\nu_\alpha$ inherits the tail behavior of the entries:
\[ 
\lim_{ t \to \infty} t^{\alpha} \nu_\alpha ( [ t , \infty) ) = 1 .
\]
The measure $\nu_\alpha$ is a perturbation of the quarter circular law: it can
be proved that $\nu_\alpha$ converges weakly to the quarter circular law as
$\alpha$ converges to $2$. Contrary to the finite variance case, the
$n^{-1/\alpha}$ normalization cannot be understood from the computation of
\[
\int\!s^2\,d\nu_{n^{-1/\alpha}X}(s) %
=\frac{1}{n^{1 + 1 / \alpha}}\sum_{i,j=1}^n|X_{ij}|^2
\]
since the later diverges. A proof of the tightness of $\nu_{n^{-1/\alpha}X}$
requires some extra care that we will explain later on. However, at a
heuristic level, we may remark that if $R_1, \ldots , R_n$ denotes the rows of
$n^{-1/\alpha}X$ then for each $k$,
\[
\NRM{R_k}^2_2 = \frac{1}{n^{2 /\alpha}} \sum_{i=1}^n |X_{k i }|^2 
\]
converges weakly to a non-negative $\frac{\alpha}{2}$-stable random variable.
Hence the $n^{-1/\alpha}$ normalization stabilizes the norm of each row of
$X$.

Following \cite{bordenave-caputo-chafai-heavygirko}, we may also investigate
the behavior of the eigenvalues of $X$. Here is the analogue of the circular
law theorem \ref{th:circular} for our heavy tailed entries matrix model.

\begin{theorem}[Eigenvalues of heavy tailed random matrices]\label{th:HTcircle}
  There exists a probability measure $\mu_\alpha$ on $\dC$ such that in
  probability $\mu_{n^{-1/\alpha}X}\weak\mu_\alpha$ as $n\to\infty$. Moreover,
  if $X_{11}$ has a bounded density, then the convergence is almost sure.
\end{theorem}

We believe that theorem \ref{th:HTcircle} can be upgraded to an a.s.\ weak
convergence, but our method does not catch this due to slow ``in probability''
controls on small singular values.

Again, the measure $\mu_\alpha$ depends only on $\alpha$ and is not known
explicitly. However, it is isotropic and has a bounded continuous density with
respect to Lebesgue measure $dxdy$ on $\dC$: $d\mu_\alpha (z) = g_\alpha ( |z
|) dxdy$. The value of $g_\alpha (r)$ is explicit for $r = 0$. As
$r\to\infty$, the tail behavior of $g_\alpha$ is up to multiplicative constant
equivalent to
\[
  r^{2 ( \alpha - 1) } e^{- \frac{\alpha}{2} r^\alpha}.
\]
This exponential decay is quite surprising and contrasts with the power tail
behavior of $f_\alpha$. It indicates that $X$ is typically far from being a
normal matrix. Also, we see that the eigenvalues limit spectrum is more
concentrated than the singular values limit spectrum. In fact, in the finite
variance case, the phenomenon is already present: the quarter circular law
has support $[0,2]$ while the circular law has support the unit disc. Again,
the measure is $\mu_\alpha$ is perturbation of the circular law: $\mu_\alpha$
converges weakly to the circular law as $\alpha$ converges to $2$.

The proof of theorem \ref{th:HTcircle} will follow the general strategy of
Girko's Hermitization. Lemma \ref{le:girko} gives a characterization of the
limit measure in terms of its logarithmic potential. Here, it turns out to be
not so convenient in order to analyze the measure $\mu_\alpha$. We will rather
use the quaternionic version of Girko's Hermitization, i.e.\ lemma
\ref{le:girkoQUAT}. For statement $(i')$ in lemma \ref{le:girkoQUAT}, we will
prove a generalized version of theorem \ref{th:HTquartercircle}.

\begin{theorem}[Singular values of heavy tailed random matrices]\label{th:HTMPz}
  For all $z \in \dC$ there exists a probability measure $\nu_{\alpha,z}$ on
  $\dR_+$ such that a.s.\ $\nu_{n^{-1/\al}X-z}\weak\nu_{\alpha,z}$ as
  $n\to\infty$. Moreover, with the notations used in lemma \ref{le:girkoQUAT},
  for all $q = q(z,\eta)\in \dH_+$, there exists $\Gamma (q) \in \dH_+$, such
  that a.s.\ $\Gamma_{n^{-1/\alpha}X}( q)$ converges to $\Gamma(q)$ and $\Gamma
  (q)_{11} = m_{\check \nu_{\alpha,z} } (\eta)$.
\end{theorem}


\subsubsection*{Objective method - sparse random graphs and trees}
The strategy for proving theorem \ref{th:HTMPz}, borrowed from
\cite{bordenave-caputo-chafai-heavygirko}, will differ significantly from the
proof of theorem \ref{th:quarter-circular}. We will prove that $n^{-1/\alpha}
X$ converges in some sense, as $n\to\infty$, to a limit random operator $A$
defined in the Hilbert space $\ell^2(\dN)$. This will be done by using the
``objective method'' initially developed by Aldous and Steele in the context
of randomized combinatorial optimization, see \cite{aldoussteele}. We build an
explicit operator on Aldous' Poisson Weighted Infinite Tree (PWIT) and prove
that it is the local limit of the matrices $n^{-1/\alpha} X$ in an appropriate
sense. While Poisson statistics arises naturally as in all heavy tailed
phenomena, the fact that a tree structure appears in the limit is roughly
explained by the observation that non-vanishing entries of the rescaled matrix
$n^{-1/\alpha}X$ can be viewed as the adjacency matrix of a sparse random
graph which locally looks like a tree. In particular, the convergence to PWIT
is a weighted-graph version of familiar results on the local tree structure of
Erd\H{o}s-R\'enyi random graphs.

\subsubsection*{Free probability}
We note finally that it is possible to associate to the PWIT a natural
operator algebra $\cM$ with a tracial state $\tau$. Then for some operator $a$
affiliated to $\cM$, the probability measure $\mu_\alpha$ is equal to the
Brown measure $\mu_a$ of $a$, and $\nu_\alpha = \mu_{|a|} = \nu_a$ is the
singular value measure of $a$. See the work of Aldous and Lyons \cite[Example
9.7 and Sub-Section 5]{aldouslyons,lyons}, and the recent work of Male
\cite{male}.

\subsection{Tightness and uniform integrability}
\label{subsec:UIHT}

\subsubsection*{Large singular values} We first prove the a.s.\ tightness of
the sequences $(\mu_{n^{-1/\alpha}X})_{n\geq1}$ and
$(\nu_{n^{-1/\alpha}X-z})_{n\geq1}$ with $z \in \dC$. It is sufficient to
prove that for some $p>0$, for all $z \in \dC$,
\begin{equation}
  \varlimsup_{n\to\infty} \int\!s^p\,d\nu_{n^{-1/\alpha}X-z}(s) < \infty.
  \label{eq:tightness}
\end{equation}
From \eqref{eq:basic1}, for any $A\in\cM_n(\dC)$, with have $s_i(A-z)\leq s_i
(A)+|z|$ and thus
\[
\int\!s^p\,d\nu_{n^{-1/\alpha}X-z}(s) %
\leq \int\,(s+|z|)^p\,d\nu_{n^{-1/\alpha}X}(s).
\]
Moreover, from \ref{eq:weyl2} we get, for any $p>0$,
\[
\int\!|\la|^p\,d\mu_{n^{-1/\alpha}X} (\la) %
\leq \int\!s^p\,d\nu_{n^{-1/\alpha}X}(s). 
\]
In summary, it it is sufficient to prove that for some $p>0$, a.s.\ 
\begin{equation}
  \varlimsup_{n\to\infty} \int\!s^p\,d\nu_{n^{-1/\alpha}X}(s) < \infty. 
  \label{eq:tightness0}
\end{equation}
and \eqref{eq:tightness} will follow. We shall use the following Schatten
bound: for all $0 < p \leq 2$,
\[
 \int\!s^p\,d\nu_{A}(s) \leq \frac{1}{n}\sum_{k=1}^n\NRM{R_k}_2^p. 
\]
for every $A\in\cM_n(\dC)$, where $R_1, \ldots, R_n$ are the rows of $A$ (for
a proof, see Zhan \cite[proof of Theorem 3.32]{zhan}). The above inequality is
an equality if $p =2$ (for $p >2$, the inequality is reversed). For our
matrix, $A =n^{-1/\alpha}X$, we find
\[
\int\!|s|^p\,d\nu_{n^{-1/\alpha}X}(s) %
\leq \frac{1}{n} %
\sum_{k=1}^n\PAR{\frac{1}{n^{2/\alpha}}\sum_{i=1}^n|X_{ki}|^2}^{\frac p 2}.
\]
The strategy of proof of \eqref{eq:tightness0} is now clear: the right hand
side is a sum of i.i.d.\ variables, and from \eqref{eq:tail}, $Y_{k,n} =
n^{-2/\alpha} \sum_{i=1}^n |X_{ki} |^2$ is the domain attraction of a
non-negative $\alpha/2$-stable law. We may thus expect, and it is possible to
prove, that for $q$ small enough, 
\[
\varlimsup_{n\to\infty} \dE Y_{k,n} ^{4q} <\infty.
\]
Then, the classical proof of the strong law of large numbers for independent random variables bounded in $L^4$ implies \eqref{eq:tightness0}.

\subsubsection*{Uniform integrability}

We will prove statement $(ii)$ of lemma \ref{le:girkoQUAT} in probability. Fix
$z \in \dC$. Using \eqref{eq:tightness}, we shall prove the uniform
integrability in probability of $\min(0,\log)$ for ${(\nu_{n^{-1/\alpha}X -
    z})}_{n\geq1}$. From Markov's inequality, it is sufficient to prove that
for some $c > 0$,
\begin{equation}\label{eq:HTUIlower}
  \lim_{t \to\infty} \varlimsup_{n\to\infty} %
  \dP\PAR{\int\!s^{-c}\,d\nu_{n^{-1/\alpha}X-z}(s)>t} = 0. 
\end{equation}
Arguing as in the finite variance case, the latter will in turn follow from
two lemmas:

\begin{lemma}[Lower bound on least singular value]\label{le:smallestSVHT}
  For all $d \geq 0$, there exist $b,c \geq 0$ such that if $M \in \cM_n (
  \dC)$ is deterministic with $\NRM{M}_2\leq n^d$, then for $n\gg1$,
  \[
  \dP(s_n(X+M)\leq n^{-b}) \leq c \sqrt{\frac{\log n}{n}}.
  \]
\end{lemma}

The next lemma asserts that the $i$-th smallest singular of the random matrix
$n^{-1/\alpha}X+M$ is at least of order $\PAR{i/n}^{2\alpha/(\alpha+2)}$ in a
weak sense. This is not optimal but enough.

\begin{lemma}[Count of small singular values]\label{le:smallSVHT}
  There exist $0< \gamma < 1$ and $c_0 > 0$ such that for all $M \in \cM_n
  (\dC)$, there exists an event $F_n$ such that $\lim_{n\to\infty}\dP(F_n)=1$
  and for all $n^{1-\gamma} \leq i \leq n-1$ and $n\gg1$,
  \[
  \dE\SBRA{s_{n-i}^{-2}(n^{-1/\alpha}X+M) \Bigm| F_n} %
  \leq c_0\PAR{\frac{n}{i}}^{\frac{2}{\alpha} + 1}  .
  \]
\end{lemma}

Let us first check that these two lemmas imply \eqref{eq:HTUIlower} (and thus
statement $(ii)$ of lemma \ref{le:girkoQUAT}). Let us define the event
$E_n:=F_n\cap\{s_{n}(n^{-1/\alpha}X-z) \geq n^{-b}\}$. Let us define also
\[
\dE_n[\,\cdot\,] := \dE[\,\cdot\,|E_n].
\]
Since the event $E_n$ has probability tending $1$, the proof of
\eqref{eq:HTUIlower} would follow from
\[
\varlimsup_{n\to\infty} %
\dE_n\SBRA{\int\!x^{-p}\,d\nu_{n^{-1/\alpha}X-z}(s)} < \infty.
\]
For simplicity, we write $s_i$ instead $s_i ( n^{-1/\alpha}X - zI)$. Since
$s_{n} \geq n^{-b}$ has probability tending to $1$, by lemma
\ref{le:smallSVHT}, for all $n^{1-\gamma} \leq i \leq n-1$,
\[
\dE_n\SBRA{s_{n-i} ^{-2}} %
\leq \frac{\dE\SBRA{s_{n-i}^{-2} \Bigm| F_n}}{\dP(s_{n}\geq n^{-b})} %
\leq c_1 \PAR{\frac{n}{i}}^{\frac{2}{\alpha} + 1}.
\]
 Then, for $0 < p \leq 2$, using Jensen inequality, we find
\begin{align*}
  \dE_n \SBRA{\int\!s^{-p}\,d\nu_{n^{-1/\alpha}X-z}(s)} & =
  \frac{1}{n}\sum_{i=0}^{\FLOOR{n^{1-\gamma}}}\dE_n\SBRA{s_{n-i}^{-p}} %
  + \frac{1}{n}\sum_{i=\FLOOR{n^{1-\gamma}}+1}^{n-1}\dE_n\SBRA{s_{n-i}^{-p}} \\
  & \leq n^{-\gamma} n^{pb} %
  + \frac{1}{n}\sum_{i=\FLOOR{n^{1-\gamma}}+1}^{n-1 } %
  \dE_n\SBRA{s_{n-i}^{-2}}^{\frac p 2} \\
  & \leq n^{-\gamma+pb} +
  \frac{1}{n} %
  \sum_{i=1}^{n}c_1^{-p}\PAR{\frac{n}{i}}^{(\frac{2}{\alpha}+1)(\frac{p}{2})}.
\end{align*}
In this last expression we discover a Riemann sum. It is uniformly bounded if
$p < \gamma/b$ and $p < 2\alpha / (\alpha + 2)$. The uniform bound
\eqref{eq:HTUIlower} follows.

\begin{proof}[Proof of lemma \ref{le:smallestSVHT}] 
  The probability that $s_1 (X) \geq n^{1+p}$ is upper bounded by the
  probability that one of the entries of $X$ is larger that $n^p$. From
  Markov's inequality and the union bound, for $p$ large enough, this event
  has probability at most $1/n$. In particular, $s_1 (X+M)\leq s_1(X) + s_1
  (M)$ is at most $2 n^q$ for $q = \max ( p, d)$ with probability at least $1
  - 1/n$. The statement is then a corollary of lemma \ref{le:snRV}. Note: a
  simplified proof in the bounded density case may be obtained by adapting the
  proof of lemma \ref{le:sn} (see \cite{bordenave-caputo-chafai-heavygirko}).
\end{proof}

\begin{proof}[Sketch of proof of lemma \ref{le:smallSVHT}] 

  We now comment the proof of lemma \ref{le:smallSVHT}, the detailed argument
  is quite technical and is omitted here. It can be found in extenso in
  \cite{bordenave-caputo-chafai-heavygirko}. First, as in the finite variance
  case, the proof reduces to derive a good lower bound on
  \[
  \DIST^2(X_1,W) = \ANG{X_1,PX_1}, 
  \]
  where $X_1$ is the first row of $X$, $W$ is a vector space of Co-dimension $n
  - d \geq n^{1-\gamma}$ (in $\dR^n$ or $\dC^n$) and $P$ is the orthogonal
  projection on the orthogonal of $W$. However, in the finite variance case, $
  \DIST^2 ( X_1, W)$ concentrates sharply around its average: $n-d$. Here, the
  situation is quite different, for instance if $W = \mathrm{vect} ( e_{n-d
    +1}, \ldots , e_{n})$, we have
  \[
  (n-d)^{-\frac 2 \alpha}\DIST^2 ( X_1, W) = (n-d)^{-\frac 2 \alpha} \sum_{i = 1} ^{n-d} |X_{1i}|^2. 
  \]
  and thus $(n-d)^{-\frac 2 \alpha}\DIST^2 ( X_1, W)$ is close in distribution
  to a non-negative $\alpha/2$-stable random variable, say $S$.

  On the other hand, if $U$ is a $n\times n$ unitary matrix uniformly
  distributed on the unitary group (normalized Haar measure), and if $W$ is
  the span of the last $d$ row vectors, then it can be argued than
  $\DIST^2(X_1,W) $ is close in distribution to $c (n-d)
  n^{\frac{2}{\alpha}-1}S$. Hence, contrary to the finite variance case, the
  order of magnitude of the distance of $X_1$ to the vector space $W$ depends
  on the geometry of $W$ with respect to the coordinate basis. We have proved
  some lower bound on this distance which are universal on $W$. More
  precisely, for any $0 < \gamma < \alpha/4$, there exists $c_1 >0$, such that
  for some event $G_n$ with $\dP ( G_n ^c ) \leq c_1 n^{-(1-2\gamma)/\alpha}$,
  \[
  \dE\SBRA{\DIST^{-2}(X_1,W) \bigm| G_n} \leq c_1 (n-d)^{-\frac 2 \alpha}. 
  \]
  The above holds for $n - d \geq n^{1-\gamma}$. We have crucially used the
  fact that for all $p > 0$, $\dE S^{-p}$ is finite, i.e.\ the non-negative
  $\alpha/2$-stable law is flat in the neighborhood of $0$.

  Note: the result implies that the vector space $W = \mathrm{vect}(e_{n-d
    +1},\ldots,e_{n})$ reaches the worst possible order of magnitude.
  Unfortunately, the upper bound on the probability of the event $G^c_n$ is
  not good enough, and we also have to 
  define the proper event $F_n$ given in lemma \ref{le:smallSVHT}. This event
  $F_n$ satisfies $\dP(F_n^c) \leq c \exp(-n^\delta)$ for some $\delta >0$ and
  $c >0$.
\end{proof}

\subsection{The objective method and the Poisson Weighted Infinite Tree (PWIT)}
\label{subsec:PWIT}

\subsubsection*{Local convergence} We now describe our strategy to obtain the
convergence of $\dE \Gamma_{n^{-1/\alpha} X}$. It is an instance of the
objective method : we prove that our sequence of random matrices converges
locally to a limit random operator. To do this, we first notice that a
$n\times n$ complex matrix $M$ can be identified with a bounded operator in
\[
\ell^2(\dN) = \{(x_k)_{k \in \dN} \in \dC^{\dN}: \sum_{k} |x_k|^2 < \infty\}
\]
by setting
\[
M e_i = 
\begin{cases}
  \sum_{j = 1} ^ n M_{ji} e_j  & \text{ if $1 \leq i   \leq n$} \\
  0 & \text{otherwise}.
\end{cases}
\]
With an abuse of notation, without further notice, we will identify our
matrices with their associated bounded operator in $\ell^2(\dN)$. The precise
notion of convergence that we will use is the following.

\begin{definition}[Local convergence]\label{def:convloc}
  Let $\cD(\dN)$ be the set of compactly supported elements of $\ell^2(\dN)$.
  Suppose $(A_n)$ is a sequence of bounded operators on $\ell^2(\dN)$ and $A$
  is a linear operator on $\ell^2(\dN)$ with domain $D(A) = \cD(\dN)$.
  For any $u,v\in \dN$ we say that $(A_n,u)$ converges locally to $(A,v)$, and
  write
  \[
  (A_n,u) \to (A,v)
  \] 
  if there exists a sequence of bijections $\sigma_n:\dN\to \dN$ such that
  \begin{itemize}
  \item $\sigma_n (v) = u$ 
  \item for all $\phi\in\cD(\dN)$, $\lim_{n\to\infty}\sigma_n ^{-1} A_n
    \sigma_n \phi = A \phi$ in $\ell^2(\dN)$.
  \end{itemize}
\end{definition}

With a slight abuse of notation we have used the same symbol $\sigma_n$ for
the linear isometry $\si_n: \ell^2(\dN)\to \ell^2(\dN)$ induced in the obvious
way. Note that the local convergence is the standard strong convergence of the
operator $\sigma_n ^{-1} A_n \sigma_n$ to $A$. This re-indexing of $\dN$
preserves a distinguished element. It is a local convergence in the following
way, if $P(x,y)$ is a non-commutative polynomial in $\dC$, then the definition
implies
\[
\ANG{e_{u},P( A_n,A_n^*)e_u} \to \ANG{e_{v},P( A,A^*)e_v}.
\]

We shall apply this definition to random operators $A_n$ and $A$ on
$\ell^2(\dN)$: to be precise, in this case we say that $(A_n ,u) \to (A ,v)$
\emph{in distribution} if there exists a random bijection $\sigma_n$ as in
definition \ref{def:convloc} such that $\sigma_n ^{-1} A_n \sigma_n \phi $
converges in distribution to $A \phi $, for all $\phi \in \cD(\dN)$, where a
random vector $\psi_n \in \ell^2 (\dN)$ converges in distribution to $\psi$ if
\[
\lim_{n\to\infty} \dE f (\psi_n) = \dE f(\psi)
\]
for all bounded continuous functions $f:\ell^2 (\dN)\to\dR$. Finally, we may
without harm replace $\dN$ by an infinite countable set $V$. All definitions
carry over by considering any bijection from $\dN$ to $V$: namely $\ell^2(V)$,
for $v \in V$, the unit vector $e_v$, $\cD(V)$ and so on.

\subsubsection*{The Poisson Weighted Infinite Tree (PWIT)} 
We now define our limit operator on an infinite rooted tree with random
edge-weights, the Poisson weighted infinite tree (PWIT) introduced by Aldous
\cite{aldous92}, see also \cite{aldoussteele}.

The $\pwit$ is the random weighted rooted tree defined as follows. The vertex
set of the tree is identified with $\dN^f:= \cup_{k \geq 1} \dN^k$ by indexing
the root as $\dN^0 = \so$, the offsprings of the root as $\dN$ and, more
generally, the offsprings of some $v \in \dN^k$ as $(v1),(v2), \ldots \in
\dN^{k+1}$ (for short notation, we write $(v1)$ in place of $(v,1)$). In this
way the set of $v\in\dN^n$ identifies the $n^\text{th}$ generation. We then
define $T$ as the tree on $\dN^f$ with (non-oriented) edges between the
offsprings and their parents (see figure \ref{fig:PWIT}).

\begin{figure}[htbp]
  \begin{center} 
  \includegraphics[angle=0,width = 8.5cm]{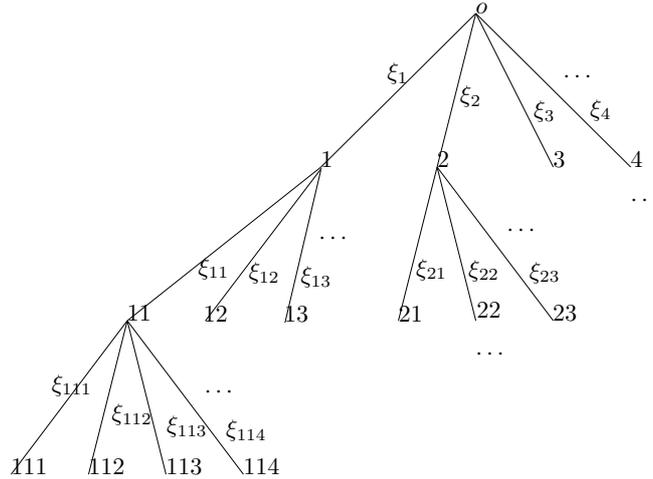}
  \end{center}
  \caption{Representation of the $\pwit$.}  
  \label{fig:PWIT}
\end{figure}

We denote by $\mathrm{Be}(1/2)$ the Bernoulli probability distribution $\frac
1 2 \delta_0 + \frac 1 2\delta_1$. Also, recall that by assumption $\lim_{ t
  \to \infty } \dP ( X_{11} / | X_{11} | \in \cdot \, | \, | X_{11} | \geq t )
= \theta(\cdot)$, a probability measure on the unit circle $\dS^1$. Now,
assign marks to the edges of the tree $T$ according to a collection $\{ \Xi_v
\}_{v \in \dN^f}$ of independent realizations of the Poisson point process
with intensity measure $(2\ell) \otimes \theta\otimes \mathrm{Be}(1/2)$ on
$\dR_+ \times \dS^1 \times \{0,1\}$, where $\ell$ denotes the Lebesgue measure
on $\dR_+$. Namely, starting from the root ${\so}$, let $\Xi_{\so} =
\{(y_1,\omega_1,\veps_1),(y_2,\omega_2,\veps_2),\dots\}$ be ordered in such a
way that $0 \leq y_1 \leq y_2 \leq \cdots$, and assign the mark
$(y_i,\omega_i,\veps_i)$ to the offspring of the root labeled $i$. Now,
recursively, at each vertex $v$ of generation $k$, assign the mark $(y_{v
  i},\omega_{vi},\veps_{vi})$ to the offspring labeled $v i$, where
$\Xi_{v}=\{(y_{v 1},\omega_{v1},\veps_{v1}),(y_{v
  2},\omega_{v2},\veps_{v2}),\dots \}$ satisfy $0 \leq y_{v 1} \leq y_{v 2}
\leq \cdots$. The Bernoulli mark $\veps_{vi}$ should be understood as an
orientation of the edge $\{v, vi\}$: if $\veps_{vi} = 1$, the edge is oriented
from $vi$ to $v$ and from $v$ to $vi$ otherwise.

We may define a random operator $A$ on $\cD (\dN^f)$ by the formula, for all
$v \in \dN^f \backslash \{ \so \}$
\begin{equation}\label{eq:defA}
  A e_{v  } =  \sum_{k \geq 1} (1- \veps_{v k} ) \omega_{vk} y_{v k}^{-1/\alpha} e_{vk}  +  \veps_{v} \omega_{a(v)}  y_{v}^{-1/\alpha} e_{a (v)} 
\end{equation}
where $a(v)$ denotes the ancestor of $v$, while
\[
A e_{\so} =  \sum_{k \geq 1} (1- \veps_{k} )  \omega_{vk}  y_{k}^{-1/\alpha} e_{k} .
\]
This defines a proper operator on $\cD(\dN^f)$. Indeed, since
$\{y_{v1},y_{v2},\dots\}$ is an homogeneous Poisson point process of intensity
$2$ on $\dR_+$: we have a.s.\ $\lim_{k \to \infty} y_{vk} / k = 2$. We thus
find for $v \in \dN^f \backslash \{ \so \}$
\[ 
\NRM{A e_{v}}_2 ^2 = \sum_{k \geq 1}
(1- \veps_{v k} ) y_{vk}^{-2/\alpha} + \veps_{a (v)} y_{v}^{-2/\alpha} <
\infty,
\]
and similarly with $\NRM{A e_{\so}}_2$.

\begin{theorem}[Local convergence to PWIT]\label{th:locconvPWIT} 
  In distribution $(n^{-1/\alpha} X, 1) \to (A,\so)$.
\end{theorem}

\begin{proof}[Sketch of proof] 
  We start with some intuition behind theorem \ref{th:locconvPWIT}. The
  presence of Poisson point processes is an instance of the Poisson behavior
  of extreme ordered statistics. If $V_{11}\geq V_{12} \geq\cdots\geq V_{1n}$
  is the ordered statistics of vector $(|X_{11} |, \dots ,|X_{1n}|)$ then,
  it is well-known that the random variable in the space of non-increasing
  infinite sequences
  \[
  n^{-1/\alpha} \PAR{V_{11} ,  V_{12} , \ldots, V_{1n} , 0 , \ldots } 
  \]
  converges weakly, for the finite dimensional convergence, to
  \begin{equation}\label{eq:Poiextr}
    \PAR{x^{-1/\alpha}_1, x^{-1/\alpha}_2, \ldots }
  \end{equation}
  where $x_{1} \leq x_{2} \leq \ldots$ are the points of an homogeneous
  Poisson point process of intensity $1$ on $\dR_+$. As observed by LePage,
  Woodroofe and Zinn \cite{zinn81}, this fact follows easily from a beautiful
  representation for the order statistics of i.i.d.\ random variables. Namely,
  if $G (u) = \dP(|X_{11}|> u)$ is (one minus) the distribution function of
  $|X_{11}|$, then
  \[
  (V_{11},\dots, V_{1n}) %
  \overset{d}{=} %
  \PAR{G^{-1}\PAR{x_1/x_{n+1}},\dots,G^{-1}\PAR{x_n/x_{n+1}}},
  \]
  where $G^{-1}(u) = \inf\{y>0:\, G(y)\leq u\} $, $u\in(0,1)$. To obtain the
  convergence to \eqref{eq:Poiextr}, it remains to notice that $G^{-1} (u )
  \sim u^{-1/\alpha}$ as $u\to0$, and $x_{n} \sim n$ a.s.\ as $n\to\infty$.

  More generally, we may reorder non-increasingly the vector
  \[
  \PAR{ ( X_{11}  , X_{11}  ),  ( X_{12} ,  X_{21}  ) , \ldots , ( X_{1n} ,  X_{n1} ) },
  \]
  and find a permutation $\pi \in \cS_n$ such that
  \[ 
  \NRM{(X_{1 \pi(1) } , X_{\pi(1)1} )}_2 %
  \geq \NRM{( X_{1 \pi(2)},X_{\pi(2)1} )}_2 %
  \geq \cdots \geq \NRM{( X_{1 \pi(n)} , X_{ \pi(n)1})}_2.
 \]
 Then, the random variable (in the space of infinite sequences in $\dC^2$ of
 non-increasing norm)
 \[
 n^{-1/\alpha} \PAR{ %
   (X_{1\pi(1)},X_{\pi(1)1}), %
   (X_{1\pi(2)},X_{\pi(2)1}), \ldots,(X_{1\pi(n)},X_{\pi(n)1}),(0,0),\ldots}
 \] 
 converges weakly, for the finite dimensional convergence, to
 \begin{equation}\label{eq:Poiextr2}
   \PAR{(\veps_1 w_1 y^{-1/\alpha}_1,(1-\veps_1) w_1 y^{-1/\alpha}_1),(\veps_2 w_2 y^{-1/\alpha}_2,(1-\veps_2)w_2 y^{-1/\alpha}_2),\ldots}. 
 \end{equation}
 In particular, we may define a bijection $\sigma_n$ in $\dN^f$, such that
 $\sigma_n (\so) = 1$, $\sigma_n (k) = \pi (k)$ if $k \ne \pi^{-1} (1)$, and
 $\sigma_n$ arbitrary otherwise. Then, for this sequence $\sigma_n$, we may
 check that $n^{-1/\alpha} \sigma_n ^{-1} X\sigma_ne_{\so}$ converges weakly to
 $A e_{\so}$ in $\ell^2 ( \dN^f)$.

 This is not good enough since we aim at the convergence for all $\phi \in
 D(\dN^f)$, not only $e_{\so}$. In particular, the above argument does not
 explain the presence of a tree in the limit operator. Note however that from
 what precedes, only the entries such that $|X_{ij}| \geq \delta n^{1/\alpha}$
 will matter for the operator convergence (for some small $\delta >0$). By
 assumption,
 \[
 \dP ( |X_{ij}| \geq \delta n^{1/\alpha} ) = \frac c n,
 \] 
 where $c = c(n) \sim \delta^{-1/\alpha}$. In other words, if we define $G$ as
 the oriented graph on $\{1, \ldots, n\}$ such that the oriented edge $(i,j)$
 is present if $|X_{ij}| \geq \delta n^{1/\alpha}$ then $G$ is an oriented
 Erd\H{o}s-R\'enyi graph (each oriented edge is present independently of the
 other and with equal probability). An elementary computation shows that the
 expected number of oriented cycles in $G$ containing $1$ and of length $k$ is
 equivalent to $c^k/n$. This implies that there is no short cycles in $G$
 around a typical vertex. At a heuristic level, this locally tree-like
 structure of random graphs explains the presence of the infinite tree $T$ in
 the limit.

 We are not going to give the full proof of theorem \ref{th:locconvPWIT}. For
 details, we refer to
 \cite{bordenave-caputo-chafai-ii,bordenave-caputo-chafai-heavygirko}. The
 strategy is as follows. For integer $m$, define $J_m = \cup_{k=0}^m \{1,
 \cdots , m \}^k \subset \dN^f$ and consider the matrix $A_{|m}$ obtained as
 the projection of the random operator $A$ on $J_m$. We prove that for all
 integer $m$, there exists an injection $\pi_{m}$ from $J_m$ to $\{1, \ldots,
 n\}$ such that $\pi_{m} (\so) = 1$ and the projection of $n^{-1/\alpha} X$ on
 $\pi_m(J_m)$ converges weakly to $A_{|m}$. The conclusion of theorem
 \ref{th:locconvPWIT} follows by extracting a sequence $m_n\to\infty$ such
 that the latter holds.

 To construct such injection $\pi_m$, we explore the entries of $X$: we first
 consider the $m$ largest entries of the vector in $(\dC^2)^m$, $\PAR{ (
   X_{12} , X_{21} ), \ldots , ( X_{1n} , X_{n1} ) }$, whose indices are
 denoted by $i_1 , \ldots ,i_m$. We then look at the $m$-largest entries of
 $\PAR{ ( X_{i_1 j } , X_{j i_1} ) }_{j \ne (1 , i_1, \ldots, i_k)}$,
 whose indices are $i_{1,1}, \ldots , i_{1,m}$. We repeat this procedure
 iteratively until we have discovered $| J_m|$ indices, and we define the
 injection $\pi_m$ as $\pi_m ( v) = i_v$. The fact that the restriction of
 $n^{-1/\alpha} X$ to $(i_v)_{v \in J_m}$ converges weakly to $A_{|m}$ can be
 proved by developing the ideas presented above.
\end{proof}

\subsubsection*{Continuity of quaternionic resolvent for local convergence}

Note that theorem \ref{th:locconvPWIT} will have a potential interest for us,
only if we know how to link the local convergence of definition
\ref{def:convloc} to the convergence of the quaternionic resolvent introduced
in sub-section \ref{subsec:convres}.

Recall that an operator $B$ on a dense domain $D(B)$ is Hermitian if for all
$x, y \in D(B)$, $\ANG{x,By} = \ANG{Bx,y}$. This operator will be {\em
  essentially self-adjoint} if there is a unique self-adjoint operator $B_1$
on $D(B_1) \supset D(B)$ such that for all $x \in D(B)$, $B_1 x = B x$ (i.e.\
$B_1$ is an extension of $B$).

\begin{lemma}[From local convergence to resolvents]\label{le:strongres} 
  Assume that $(A_n)$ and $A$ satisfy the conditions of definition
  \ref{def:convloc} and $(A_n,u) \to (A,v)$ for some $u,v\in \dN$. If the
  bipartized operator $B$ of $A$ is essentially self-adjoint, then, for all $q
  \in \dH_+$,
  \[
  R_{A_n} ( q) _{uu}  \to R_{A} ( q) _{vv}.
  \]
\end{lemma}

\begin{proof}
  Fix $z \in \dC$ and let $B_n (z) = B_n - q(z,0) \otimes I$, where $B_n$ is
  bipartized operator of $A_n$. By construction, for all $\phi \in D(B) = \cD
  ( \dN \times \dZ/2 \dZ)$, $\sigma_n^{-1} B_n (z) \sigma_n \phi$ converges to
  $B (z) \phi$ (this is the strong operator convergence). The proof is then a
  direct consequence of \cite[Theorem VIII.25(a)]{reedsimon}: in this
  framework, the strong operator convergence implies the strong resolvent
  convergence. Namely, for all $\phi, \psi \in D(B)$ and $\eta \in \dC_+$,
  \[ 
  \ANG{\phi,(\sigma_n^{-1}B_n(z)\sigma_n-\eta I)^{-1}\psi}
  \rightarrow \ANG{\phi,(B(z)-\eta I)^{-1}\psi}.
  \]
  We conclude by applying this to $\phi, \psi \in \{ e_{v}, e_{\hat v}\}$.
\end{proof}


\begin{remark}[A non-self-adjoint Hermitian operator]
  A key assumption in the above lemma is the essential self-adjointness of the
  bipartized limit operator. A local limit of Hermitian matrices will
  necessary be Hermitian. It may not however be always the case that the limit
  is essentially self-adjoint. Since any bounded Hermitian operator is
  essentially self-adjoint, for an example, we should look for an unbounded
  operator. Let $(a_k)_{k \in \dN}$ be a sequence on $\dR_+$ and define on
  $D(\dN)$, for $k \geq 2$,
  \[
  B e_k = a_k e_{k+1} + a_{k-1} e_{k-1}. 
  \]
  while $B e_1 = a_1 e_{2}$. In matrix form, $B$ is a tridiagonal symmetric
  infinite matrix. The work of Stieltjes \cite{stieltjes94} implies that $B$
  will be essentially self-adjoint if and only if
  \[
  \varlimsup_{n\to\infty}  \sum_{k \geq n} a_k  ^{-1} = \infty.  
  \]
\end{remark}

\subsection{Skeleton of the main proofs}
\label{ss:skeleton}

All ingredients have finally been gathered. The skeleton of proof of theorems
 \ref{th:HTcircle}, \ref{th:HTMPz} and the characterization of $\mu_\alpha$ and
$\nu_{\alpha,z}$ is as follows:

\begin{enumerate}
\item By lemma \ref{le:concQUAT}, for all $q \in \dH_+$, a.s.\ , in norm,
  \[
  \Gamma_{n^{-1/\alpha} X}(q) -\dE \Gamma_{n^{-1/\alpha} X }(q)\to0
  \]
\item Since $X$ has exchangeable rows, for all $q \in \dH_+$,
  \[
  \dE \Gamma_{n^{-1/\alpha} X} (q) = \dE R_{n^{-1/\alpha} X} (q)_{11}
  \]
\item\label{it:skel:0} We prove in sub-section \ref{ss:selfadj} that the
  bipartized operator $B$ of the random operator $A$ of sub-section
  \ref{subsec:PWIT} is a.s.\ essentially self-adjoint
\item It follows by theorem \ref{th:locconvPWIT} and lemma \ref{le:strongres},
  \[
  \lim_{n \to \infty} \dE \Gamma_{n^{-1/\alpha} X}(q)  %
  = \dE R_{A} (q)_{\so \so} %
  =  %
  \begin{pmatrix}
   a(q) & b(q)   \\  
    \bar b(q)   & a (q) 
  \end{pmatrix}
  \]
\item By lemma \ref{le:propRes}, a.s.\
  $\nu_{n^{-1/\alpha}X-z}\weak\nu_{\alpha,z}$ as $n\to\infty$, where
  $\nu_{\alpha,z}$ is characterized by
  \[
  m_{\check \nu_{\alpha,z}} (\eta) = a (q)
  \]
\item We know from sub-section \ref{subsec:UIHT} that statement $(ii)$ of
  lemma \ref{le:girkoQUAT} holds for $n^{-1/\alpha} X$ in probability. Then, in probability,
  $\mu_{n^{-1/\alpha} X}\weak\mu_\alpha$ as $n\to\infty$, where $\mu_\alpha$
  is characterized by, in $\cD'(\dC)$,
  \[
  \mu_\alpha =  - \frac 1 \pi  \lim_{q ( z, it ) : t \downarrow 0} \pd  b (q)
  \]
\item\label{it:skel:1} We analyze in sub-section \ref{ss:selfadj} $R_{A}(q)_{\so \so}$ to obtain
  the properties of $\nu_{\alpha,z}$ and $\mu_\alpha$.
  
  \item Finally, when $X_{12}$ has a bounded density we improve the convergence to almost sure (in sub-section \ref{subsec:asconv}). 
\end{enumerate}

\subsection{Analysis of the limit operator}
\label{ss:selfadj}

This sub-section is devoted to items \ref{it:skel:0} and \ref{it:skel:1} which
appear above in the skeleton of proof of theorems \ref{th:HTcircle} and
\ref{th:HTMPz}.

\subsubsection*{Self-adjointness}

Here we check the self-adjointness of the bipartized operator $B$ of $A$.

\begin{proposition}[Self-adjointness of bipartized operator on
  PWIT] \label{prop:sa} Let $A$ be the random operator defined by
  \eqref{eq:defA}. With probability one, $B$ is essentially self-adjoint.
\end{proposition}

This proposition relies on the following criterion of self-adjointness (see
\cite{bordenave-caputo-chafai-heavygirko} for a proof).

\begin{lemma}[Criterion of self-adjointness of the bipartized
  operator]\label{le:criteresa}
  Let $\kappa > 0$ and $T = (V,E)$ be an infinite tree on $\dN^f$ and
  $(w_{uv}, w_{vu})_{\{v,u \} \in E}$ be a collection of pairs of weight in
  $\dC$ such that for all $u \in V$,
  \[
  \sum_{v : \{u , v\} \in E } |w_{uv}|^2 + |w_{vu}|^2 < \infty.  
  \] Define the operator on $\cD(V)$ as 
  \[
  A e_u = \sum_{ v : \{u , v\} \in E } w_{vu} e_{v}. 
  \]  
  Assume also that there exists a sequence of connected finite subsets
  $(S_n)_{n \geq 1}$ in $V$, such that $S_n\subset S_{n+1}$, $\cup_n S_n = V$,
  and for every $n$ and $v \in S_n$,
  \[
  \sum_{u \notin S_n : \{u , v\} \in E} \PAR{|w_{uv} |^2 + |w_{vu} |^2}  \leq \kappa.
  \]  
  Then the bipartized operator $B$ of $A$ is essentially self-adjoint.
\end{lemma}

We will use a simple lemma on Poisson processes (for a proof \cite[Lemma
A.4]{bordenave-caputo-chafai-ii}).

\begin{lemma}[Poisson process tail]\label{le:taufinite}
  Let $\kappa >0$, $0 < \alpha < 2$ and let $0 < x_1 < x_2 < \cdots$ be a
  Poisson process of intensity $1$ on $\dR_+$. If we define 
  \[
  \tau := %
  \inf\BRA{t\in\dN : \sum_{k = t+1}^\infty x_k^{-2/\alpha} \leq\kappa}
  \]
  then $\dE\tau $ is finite and goes to $0$ as $\kappa$ goes to infinity.
\end{lemma}

\begin{proof}[Proof of proposition \ref{prop:sa}]
  For $\kappa >0$ and $v \in \dN^f$, we define 
  \[
  \tau_v= \inf \{ t \geq 0 : \sum_{k = t+1}^\infty |y_{v k}|^{-2/\alpha} \leq
  \kappa\}.
  \]
  The variables $(\tau_{v})$ are i.i.d.\ and by lemma \ref{le:taufinite},
  there exists $\kappa >0$ such that $\dE \tau_{v} < 1$. We fix such $\kappa$.
  Now, we put a green color to all vertices $v$ such that $\tau_v\geq 1$ and a
  red color otherwise. We consider an exploration procedure starting from the
  root which stops at red vertices and goes on at green vertices. More
  formally, define the sub-forest $T^g$ of $T$ where we put an edge between
  $v$ and $vk$ if $v$ is a green vertex and $1 \leq k \leq \tau_{v}$. Then, if
  the root $\so$ is red, we set $S_1 = C^g (T) = \{\so\}$. Otherwise, the root
  is green, and we consider $T^g_{\so}=(V^g_{\so},E^g_{\so})$ the subtree of
  $T^g$ that contains the root. It is a Galton-Watson tree with offspring
  distribution $\tau_{\so}$. Thanks to our choice of $\kappa$, $T^g_{\so}$ is
  almost surely finite. Consider $L^g_{\so}$ the leaves of this tree (i.e.\
  the set of vertices $v$ in $V^g_{\so}$ such that for all $1 \leq k \leq
  \tau_v$, $vk$ is red). The following set satisfies the condition of Lemma
  \ref{le:criteresa}:
  \[
  S_1 = V^g_{\so} \bigcup_{v \in L^g_{\so}} \{1 \leq k \leq \tau_v : v k \}.
  \]
  We define the outer boundary of $\{\so\}$ as 
  \[
  \pd_\tau \{\so \}= \{1, \ldots, \tau_{\so}\}
  \]
  and for $v= (i_1, \ldots, i_k) \in \dN^f \backslash \{\so\}$ we set
  \[
  \pd_\tau \{v \} = \{(i_1,\ldots, i_{k-1}, i_{k} +1) \} \cup \{(i_1,\ldots,
  i_{k},1), \ldots, (i_1,\ldots, i_{k},\tau_v)\} .
  \]
  For a connected set $S$, its outer boundary is
  \[
  \pd_\tau S %
  = \PAR{ \bigcup_{ v \in S} \pd_\tau \{v\} } \backslash S.
  \]
  Now, for each vertex $u_1, \ldots, u_k \in \pd_\tau S_1 $, we repeat
  the above procedure to the rooted subtrees $T_{u_1}, \ldots, T_{u_k}$. We
  set 
  \[
  S_2 = S_1 \bigcup \cup_{1 \leq i \leq k} C^b ( T_{u_i}).
  \]
  Iteratively, we may thus almost surely define an increasing connected
  sequence $(S_n)$ of vertices with the properties required for lemma
  \ref{le:criteresa}.
\end{proof} 

\subsubsection*{Computation of resolvent}

As explained in section \ref{ss:skeleton}, the properties of the measures
$\mu_\alpha$ and $\nu_{\alpha,z}$ can be deduced from the analysis of the
limit resolvent operator. Resolvent are notoriously easy to compute on trees.
More precisely, let $T = (V,E)$ be a tree and $A,B$ be as in lemma
\ref{le:criteresa} and let $\so \in V$ be a distinguished vertex of $V$ (in
graph language, we root the tree $T$ at $\so$). For each $v \in V \backslash
\{\so\}$, we define $V_v \subset V$ as the set of vertices whose unique path
to the root $\so$ contains $v$. We define $T_v = (V_v, E_v)$ as the subtree of
$T$ spanned by $V_v$. We may consider $A_v$, the projection of $A$ on $V_v$,
and $B_v$ the bipartized operator of $A_v$. Finally, we note that if $B$ is
self-adjoint then so is $B_v (z)$ for every $z \in \dC$. The next lemma is an
operator analog of the Schur inversion by block formula \eqref{eq:schur}.

\begin{lemma}[Resolvent on a tree] 
  \label{le:schurB} 
  Let $A,B$ be as in lemma \ref{le:criteresa}. If $B$ is self-adjoint then for
  any $q= q(z,\eta) \in \dH_+$,
  \[
    R_A(q) _{\so\so}   
    =  - \PAR{ q  + \sum_{v \sim \so}   
      \begin{pmatrix}    
        0      &  w_{ \so v}    \\ 
        \Ol w_{v \so  }  & 0   
      \end{pmatrix}  \Wt R_A (q)_{vv}   
      \begin{pmatrix}    0      &  w_{v \so}    \\ 
        \Ol w_{\so v }  & 0   
      \end{pmatrix} }^{-1} ,
  \]
  where $\Wt R_A (q)_{vv} := \Pi_v R_{B_v} (q) \Pi^*_v$, and
  $R_{B_v}(q)=(B_v(z)-\eta)^{-1}$ is the resolvent 
  of $B_v$.
\end{lemma}

We come back to our random operator $A$ defined on the PWIT and its quaternionic
resolvent $R_A(q)$. We analyze the random variable
\[
R_A(q)_{\so \so} := 
\begin{pmatrix} 
 a (z,\eta) & b(z,\eta)  \\  
 b' (z,\eta) &  c (z,\eta) 
\end{pmatrix}.
\]
The random variables $a(z,\eta)$ solves a nice \emph{recursive distribution
  equation} (RDE). This type of recursion equation is typical of combinatorial
observable defined on random rooted trees. More precisely, we define the
measure on $\dR_+$,
\[
\Lambda_\alpha = \frac{\alpha}{2} x ^{-\frac \alpha 2 -1}dx.  
\]
\begin{lemma}[Recursive distribution equation]
  For all $q = q (z, \eta) \in \dH_+$, if $L_q$ is the distribution on $\dC_+$
  of $a(z,\eta)$ then $L_q$ solves the equation in distribution:
  \begin{equation} \label{eq:RDEa} %
    a \overset{d}{=} \frac{ \eta + \sum_{k \geq 1} \xi_k a_k }{ | z |^2 -
      \PAR{ \eta + \sum_{k\geq 1} \xi_k a_k } \PAR{ \eta + \sum_{k
          \geq 1} \xi'_k a'_k } },
  \end{equation} 
  where $a$, $(a_k)_{k \in \dN}$ and $(a'_k)_{k \in \dN}$ are i.i.d.\ with law
  $L_q$ independent of $\{\xi_k\}_{k \in \dN}$, $\{\xi'_k\}_{k \in \dN}$ two
  independent Poisson point processes on $\dR_+$ with intensity
  $\Lambda_\alpha$.

  Moreover, with the same notation,
  \begin{equation} \label{eq:RDEb} %
    b \overset{d}{=} \frac{  - z}{ | z |^2 -
      \PAR{ \eta + \sum_{k \geq 1} \xi_k a_k } \PAR{ \eta + \sum_{k
          \geq 1} \xi'_k a'_k } }. 
  \end{equation} 
\end{lemma}

\begin{proof} 
  This is sa simple consequence of lemma \ref{le:schurB}. Indeed, for $k \in
  \dN$, we define $T_k$ as the subtree of $T$ spanned by $k \dN^f$. With the
  notation of lemma \ref{le:schurB}, for $k \in \dN$, $R_{B_k} (q) = (B_k(z) -
  \eta ) ^{-1}$ is the resolvent operator of $B_k$ and set
  \[
  \Wt  R_A (q)_{kk} = \Pi_k R_{B_k} (q) \Pi^*_k = \begin{pmatrix} a_k & b_k \\
    b'_k & c_k\end{pmatrix}.
  \]
  Then, by lemma \ref{le:schurB}, we get
  {\footnotesize
   \begin{align*} 
    R(q) _{\so\so}   
    &=  %
    - \PAR{q  + \sum_{k \geq 1}   
      \begin{pmatrix}    
        0  & \veps _k w_k y_{k}^{-1/\alpha } \\ 
        (1- \veps _k ) w_k y_{k}^{-1/\alpha }  & 0  
      \end{pmatrix}  
      \begin{pmatrix} 
        a_k & b_k  \\  
        b'_k &  c_k
      \end{pmatrix}   
      \begin{pmatrix}    
        0 &(1- \veps_k )  w_k y_{k}^{-1/\alpha }     \\ 
        \veps_k  w_k y_{k}^{-1/\alpha }  & 0   
      \end{pmatrix}
      }^{-1} \\
    & = %
    - \PAR{ U  +     
      \begin{pmatrix} 
        \sum_{k} (1 - \veps_k ) y_{k}^{-2/\alpha } c_k & 0 \\
        0 &   \sum_{k} \veps_k y_{k}^{-2/\alpha } a_k 
      \end{pmatrix}     
    }^{-1} \\
    & = %
    D^{-1} %
    \begin{pmatrix}  
      \eta +  \sum_{k }  \veps_k y_{k}^{-2/\alpha }  a_k     &  - z  \\ 
      -  \bar z   & \eta +\sum_{k}(1-\veps_k )y_{k}^{-2/\alpha }c_k  
    \end{pmatrix},
  \end{align*}
  } 
  with 
  {\small
  \[
  D := |z|^2 - %
  \PAR{\eta + \sum_{k \geq 1} \veps_k y_{k}^{-2/\alpha } a_k} %
  \PAR{\eta + \sum_{k \geq 1} (1-\veps _k) y_{k}^{-2/\alpha } c_k}.
  \]
  } 
  Now the structure of the PWIT implies that 
  \begin{itemize}
  \item[(j)] $a_k$ and $c_k$ have common distribution $L_q$
  \item[(jj)] the variables $(a_k,c_k)_{k \in \dN}$ are i.i.d.
  \end{itemize}
  Also the thinning property of Poisson point processes implies that 
  \begin{itemize}
  \item[(jjj)] $\{ \veps _k y_{k}^{-2/\alpha } \}_{k \in \dN}$ and $\{
    (1-\veps _k)y_{k}^{-2/\alpha } \}_{k \in \dN}$ are independent Poisson
    point process with common intensity $\Lambda_\alpha$.
  \end{itemize}
\end{proof}

Even if \eqref{eq:RDEa} looks complicated at first sight, for $\eta = it$, it
is possible to solve it explicitly. First, for $t \in \dR_+$, $a(z, it)$ is
pure imaginary and we set
\[
h(z,t) = \Im  ( a (z, it) )  = - i a (z, it) \in (0, t^{-1}]. 
\]
The crucial ingredient, is a well-known and beautiful lemma. It can be derived
form a representation of stable laws, see e.g.\ LePage, Woodroofe, and Zinn
\cite{zinn81} and also Panchenko and Talagrand \cite[Lemma 2.1]{MR2371485}.

\begin{lemma}[Poisson-stable magic formula]\label{le:magic}
  Let $\{ \xi_k\}_{k \in \dN}$ be a Poisson process with intensity
  $\Lambda_\alpha$. If $(Y_k)$ is an i.i.d.\ sequence of non-negative random
  variables, independent of $\{ \xi_k\}_{k \in \dN}$, such that $\dE [ Y_1
  ^{\frac \alpha 2}] < \infty$ then
  \[
  \sum_{k \in \dN} \xi_k Y_k \stackrel{d}{=} \dE [ Y_1 ^{\frac \alpha 2} ]^{\frac 2
    \alpha} \sum_{k \in \dN} \xi_k \stackrel{d}{=} \dE [ Y_1 ^{\frac \alpha 2} ]^{\frac 2
    \alpha} S,
  \]
  where $S$ is the positive ${\frac \alpha 2}$-stable random variable with Laplace
  transform for all $x \geq 0$,
  \begin{equation}\label{eq:laplaceZ}
    \dE \exp ( - x S) =   \exp\PAR{ - \Gamma\PAR{1-{\frac \alpha 2}} x^{\frac \alpha 2} }.
  \end{equation}
\end{lemma}

\begin{proof}[Proof of lemma \ref{le:magic}] Recall the formulas, for $y \geq
  0$, $\eta >0$ and $0 < \eta < 1$ respectively,
  \begin{equation} \label{eq:gammaLaplace} %
    y^{-\eta} %
    = \Gamma(\eta)^{-1} \int_0 ^\infty x ^{\eta -1} e^{- x y} dx %
    \quad\text{and}\quad %
    y^{\eta} %
    = \Gamma(1-\eta)^{-1} \eta \int_0 ^\infty x ^{-\eta -1} (1 - e^{- x y}) dx.
  \end{equation} 
  From the L\'evy-Khinchin formula we deduce that, with $s \geq 0$,
  \begin{align*}
    \dE \exp\PAR{ - s \sum_{k} \xi_k Y_k }
    & = \exp \PAR{ \dE \int_0 ^\infty ( e^{-x  s Y_1} - 1 ) \beta x ^{-{\frac \alpha 2} - 1} dx } \nonumber \\
    &= \exp\PAR{ - \Gamma\PAR{1-{\frac \alpha 2}} s^{\frac \alpha 2} \dE
      [Y_1^{\frac \alpha 2}]} \label{eq:laplace}.
  \end{align*}
\end{proof}

Hence, by lemma \ref{le:magic}, we may rewrite \eqref{eq:RDEa} as
\begin{equation}
  \label{eq:RDEhy}
  h  \overset{d}{=}  \frac{  t  + y S    }{ |  z |^2 + \PAR{ t  + y  S    } \PAR{ t  + y  S'    }  },
\end{equation}
where $S$ and $S'$ are i.i.d.\ variables with common Laplace transform
\eqref{eq:laplaceZ} and the function $y = y (|z|^2, t) = \dE[
h^{\alpha/2}]^{2/\alpha}$ is the unique solution of the equation in $y$:
\[
1 = \dE\PAR{\frac{ty^{-1}+S}{|z|^2 + \PAR{t+yS}\PAR{t+yS'}}}^{\frac{\alpha}{2}}. 
\]
(since the left hand side is decreasing in $y$, the solution is unique). In
the above equations, it is also possible to consider the limits as $t
\downarrow 0$.

As explained in section \ref{ss:skeleton}, this implies that, in $\cD'(\dC)$,
$\mu_\alpha$ is equal to
\[
- \frac 1 \pi \lim_{t \downarrow 0} \dE b ( \cdot , it). 
\]
Using \eqref{eq:RDEb}, after a simple computation, we find that the density $g_\alpha$ of
$\mu_\alpha$ at $z$ is
\[
\frac 1 \pi \PAR{y^2 _*( |z|^2) - 2 |z|^{2} y _* ( |z|^2) y' _* ( |z|^2)
} \dE \frac {S S'}{ \PAR{| z |^2 + y^2 _* ( |z|^2) S S' }^2}
\]
where $y_* ( r) = y (r)$ is the unique solution 
\[
1 = \dE \PAR{  \frac{  S    }{r + y^2  S     S'   }   }^{\frac \alpha 2} . 
\]
After more computations, it is even possible to study the regularity of $y_*$,
find the explicit solution at $0$, and an asymptotic equivalent as $r \to
\infty$. All these results can then be translated into properties of
$\mu_\alpha$. We will not pursue here these computation which are done in
\cite{bordenave-caputo-chafai-heavygirko}. We may simply point out that
$\mu_\alpha$ converges weakly to the circular law as $\alpha \to 2$, is a
consequence of the fact that the non-negative $\alpha/2$-stable random variable
$S / \Gamma( 1 - \alpha/2) ^{2 / \alpha}$ converges to a Dirac mass as $\alpha \to
2$ (see \eqref{eq:laplaceZ}).


\subsection{Improvement to almost sure convergence}

\label{subsec:asconv}

Let $\nu_{\alpha,z}$ be as in theorem \ref{th:HTMPz}. In order to improve the
convergence to a.s., it is sufficient to prove that for all $z \in \dC$, a.s.
\[
\lim_{n \to \infty} U_{\mu_{n^{-1/\alpha} X}} (z)=  L %
\quad\text{where}\quad %
L:= - \int_0^\infty\!\log(s)\,d\nu_{\alpha,z} (s).
\]
We have already proved that this convergence holds in probability. It is thus
sufficient to prove that there exists a \emph{deterministic} sequence $L_n$
such that a.s.
\begin{equation}\label{eq:ULn}
\lim_{n \to \infty} \left( U_{\mu_{n^{-1/\alpha} X}} (z) -  L_n \right) = 0.
\end{equation}
Now, thanks to the bounded density assumption and remark \ref{re:lesn}, one
may use lemma \ref{le:sn} for the matrix $X - n^{1/\alpha} zI$ in order to
show that that there exists a number $b > 0$ such that a.s.\ for $n \gg 1$,
\[
s_n ( n^{-1/\alpha} X - z I ) \geq n^{-b}. 
\] 
Similarly, up to an increase of $b$ if needed, we also get from
\eqref{eq:tightness} that a.s.\ for $n \gg 1$,
\[
s_1 ( n^{-1/\alpha} X - z I ) \leq n^{b}. 
\] 
Now, we consider the function 
\[
f_n (x) = \IND_{\{n^{-b} \leq | x| \leq n^b\}}\log(x).
\]
From what precedes, a.s.\ for $n\gg1$,
\begin{equation}\label{eq:Ufn}
U_{\mu_{n^{-1/\alpha} X}} (z) %
= -\int_0^\infty\!\log(s)\,d\nu_{n^{-1/\alpha}X-zI}(s) %
= -\int_0^\infty\!f_n(s)\,d\nu_{n^{-1/\alpha} X-zI}(s).
\end{equation}
The total variation of $f_n$ is bounded by $c \log n$ for some $c >0$. Hence
by lemma \ref{le:concspec}, if
\[
L_n := \dE \int\!f_n(s)\,d\nu_{n^{-1/\alpha} X - z I} (s),
\]
then we have, 
\[
\dP\PAR{\ABS{ \int\!f_n(s)\,d\nu_{n^{-1/\alpha} X - z I} (s) - L_n} \geq t} %
\leq 2 \exp\PAR{- 2 \frac{ n t^2 } { (c \log n)^2 } }.
\]
In particular, from the first Borel-Cantelli lemma, a.s.,
\[
\lim_{n \to \infty} \PAR{\int\!f_n(s)\,d\nu_{n^{-1/\alpha}X-zI}(s)-L_n}= 0.
\]
Finally, using \eqref{eq:Ufn}, we deduce that \eqref{eq:ULn} holds almost
surely.

\section{Open problems}
\label{se:open}

We list in this section some open problems related to the circular law theorem.

\subsubsection*{Universality of Gaussian Ensembles}

The universality dogma states that if a real or complex functional of $X$ is
enough symmetric and depend on enough entries then it is likely that this
functional behaves asymptotically ($n\to\infty$) like in the Gaussian case
(Ginibre Ensemble here) as soon as the first moments of $X_{11}$ match certain
Gaussian moments (depends on the functional). This can be understood as a sort
of non-linear central limit theorem. Among interesting functionals, we find
for instance the following:
\begin{itemize}
\item Spectral radius (has Gumbel fluctuations for the Complex Ginibre Ensemble);
\item argument of $\la_1(X)$ (is uniform on $[0,2\pi]$ for the Complex
  Ginibre Ensemble);
\item Law of $\la_n(n^{-1/2}X)$ (see \cite[Chapter 15]{MR2641363} for the
  Complex Ginibre Ensemble). The square of the smallest singular value
  $s_n(n^{-1/2}G)^2$ of the Complex Ginibre Ensemble follows an exponential
  law \cite{MR964668} and this result is asymptotically universal
  \cite{MR2647142};
\item Gap probabilities and Vorono{\"{\i}} cells (see \cite{MR2536111} and
  \cite{MR2582643} for the Ginibre Ensemble);
\item Linear statistics of $\mu_X$ (some results are available such as
  \cite{MR2294978,MR2361453,MR2530159,MR2095933});
\item Empirical distribution of the real eigenvalues of $n^{-1/2}X$ when
  $X_{11}$ is real (tends to uniform law on $[-1,1]$ for the Real Ginibre
  Ensemble);
\item Unitary matrix in the polar decomposition (Haar unitary for the Complex
  Ginibre). This is connected to the R-diagonal concept in free probability
  theory \cite{MR1784419};
\item If $X_{11}$ has infinite fourth moment then the eigenvalues of largest
  modulus blow up and are asymptotically independent (Poisson statistics at
  some scale);
\item A large deviations principle for $\mu_X$ at speed $n^2$ which includes as
  a special case the one obtained for the Complex Ginibre Ensemble by Hiai and
  Petz \cite{MR1606719} (see also Ben Arous and Zeitouni \cite{MR1660943}) and
  references therein. The analogous question for Hermitian models (Wigner and
  GUE) is also open. The answer depends on the chosen scale, the class of
  deviations, and the topology.
\end{itemize}
One may group most of these functionals by considering the spectrum as a point
process.

It is also possible to consider universality beyond i.i.d.\ entries models.
For instance, if $X$ has exchangeable entries as a random vector of
$\dC^{n^2}$ and if $X$ satisfies to suitable mean-variance normalizations,
then we expect that $\dE\mu_X$ tends to the circular law due to a Lindeberg
type phenomenon, see \cite{MR2294976} for the Hermitian case (Wigner).
Similarly, if $X$, as a random vector of $\dC^{n^2}$, is log-concave (see
footnote \ref{fn:logconc}) and isotropic (i.e.\ its covariance matrix is
identity) then we expect that $\dE\mu_X$ tends to the circular law, see
\cite{EJP2011-37} for i.i.d.\ log-concave rows. Since the indicator of a
convex set is a log-concave measure, one may think about the Birkhoff polytope
(convex envelope of permutation matrices) and ask if the circular law holds
for random uniform doubly stochastic matrices, see \cite{djalil-dme} and
\cite{chatterjee-diaconis-sly}.

\subsubsection*{Variance profile}

We may consider the matrix $Y$ defined as $Y_{ij} = X_{ij} \sigma( i/ N ,
j/N)$ where $\sigma : [0,1]^2 \to [0,1]$ is a measurable function. The measure
$\mu_{n^{-1/2} Y}$ should converge a.s.\ to a limit probability measure
$\mu_\sigma$ on $\dC$. For finite variance Hermitian matrices, this question
has been settled by Khorunzhy, Khoruzhenko, Pastur and Shcherbina \cite{KKPS},
for heavy tailed Hermitian matrices, by Belinschi, Dembo, Guionnet
\cite{BDG09}. Girko has also results on the singular values of random matrices
with variance profile. 

\subsubsection*{Elliptic laws}

We add some dependence in the array $(X_{ij})_{i,j \geq 1}$ : we consider an
infinite array $(X_{ij}, X_{ji})_{ 1 \leq i < j \leq n}$ of i.i.d.\ pairs of
complex random variables, independent of $(X_{ii})_{i \geq 1}$ an i.i.d.\
sequence of random variables. Assume that $\VAR(X_{12}) = \VAR(X_{21}) = 1$
and $\mathrm{Cor}(X_{12}, X_{21}) = t \in \{z\in\dC:|z|\leq1\}$. There is a
conjectured universal limit for $\mu_{n^{-1/2} X}$ computed by Girko
\cite{MR1080966}, called the elliptic law. This model interpolates between
Hermitian and non-Hermitian random matrices. When $X = \sqrt {( 1 + \tau)/2}
H_1 + i \sqrt {( 1 - \tau)/2}H_2$, with $0\leq \tau \leq 1$ and $H_1, H_2$ two
independent GUE, this model has been carefully analyzed by Bender in
\cite{MR2594353}, see also Ledoux \cite{MR2446909}, Johansson
\cite{MR2288065}, and Khoruzhenko and Sommers \cite{KS}.

\subsubsection*{Oriented $r$-regular graphs and Kesten-McKay measure} 

Random oriented graphs are host of many open problems. For example, for
integers $n \geq r \geq 3$, an oriented $r$-regular graph is a graph on $n$
vertices such that all vertices have $r$ incoming and $r$ outgoing oriented
edges. Consider the adjacency matrix $A$ of a random oriented $r$-regular
graph sampled from the uniform measure (there exists suitable simulation
algorithms using matchings of half edges). It is conjectured that as $n \to
\infty$, a.s.\ $\mu_A$ converges to the probability measure
\[
\frac{1}{\pi} \frac{r ^2(r-1)}{(r^2 - |z|^2)^2 }\IND_{\{|z|<\sqrt r\}}\,dxdy.
\]
It turns out that this probability measure is also the Brown measure of the
free sum of $r$ unitary, see Haagerup and Larsen \cite{MR1784419}. The
Hermitian (actually symmetric) version of this measure is known as the
Kesten-McKay distribution for random non-oriented $r$-regular graphs, see
\cite{kesten59,mckay}. We recover the circular law when $r\to\infty$ up to
renormalization.

\subsubsection*{Invertibility of random matrices}

The invertibility of random matrices is one of the keys behind the circular
law theorem \ref{th:circular}. Let us consider the case were $X_{11}$ is
Bernoulli $\frac{1}{2}(\de_{-1}+\de_{1})$. A famous conjecture by Spielman and
Teng (related to their work on smoothed analysis of algorithms
\cite{MR2078601,MR1989210}) states that there exists a constant $0<c<1$ such
that 
\[
\dP(\sqrt{n}\,s_n(X)\leq t)\leq t+c^n
\]
for $n\gg1$ and any small enough $t\geq0$. This was almost solved by Rudelson
and Vershynin \cite{MR2407948} and Tao and Vu \cite{MR2647142}. In particular,
taking $t=0$ gives $\dP(s_n(X)=0)=c^n$. This positive probability of being
singular does not contradict the asymptotic invertibility since by the first
Borel-Cantelli lemma, a.s.\ $s_n(X)>0$ for $n\gg1$. Regarding the constant $c$
above, it has been conjectured years ago that
\[
\dP(s_n(X)=0)=\PAR{\frac{1}{2}+o(1)}^n.
\]
This intuition comes from the probability of equality of two rows, which
implies that $\dP(s_n(X)=0)\geq(1/2)^n$. Many authors contributed to the
analysis of this difficult nonlinear discrete problem, starting from Koml\'os,
Kahn, and Szemer\'edi. The best result to date is due to Bourgain, Vu, and
Wood \cite{MR2557947} who proved that
$\dP(s_n(X)=0)\leq \PAR{1/\sqrt{2}+o(1)}^n$.

\subsubsection*{Roots of random polynomials}

The random matrix $X$ has i.i.d.\ entries and its eigenvalues are the roots of
its characteristic polynomial. The coefficients of this random polynomial are
neither independent nor identically distributed. Beyond random matrices, let
us consider a random polynomial $P(z)=a_0+a_1z+\cdots+a_nz^n$ where
$a_0,\ldots,a_n$ are independent random variables. By analogy with random
matrices, one may ask about the behavior as $n\to\infty$ of the roots
$\la_1(P),\ldots,\la_n(P)$ of $P$ in $\dC$ and in particular the behavior of
their empirical measure $\frac{1}{n}\sum_{i=1}^n\de_{\la_i(P)}$. The
literature on this subject is quite rich and takes its roots in the works of
Littlewood and Offord, Rice, and Kac. We refer to Shub and Smale
\cite{MR1230872}, Aza{\"{\i}}s and Wschebor \cite{MR2149413}, and Edelman and
Kostlan \cite{MR1290398,MR1376652} for (partial) reviews. As for random
matrices, the case where the coefficients are real is more subtle due to the
presence of real roots. Regarding the complex case, the zeros of Gaussian
analytic functions is the subject of a recent monograph \cite{MR2552864} in
connection with determinantal processes. Various cases are considered in the
literature, including the following three families:
\begin{itemize}
\item Kac polynomials, for which $(a_i)_{0\leq i\leq n}$ are i.i.d.\ 
\item Binomial polynomials, for which $a_i=\sqrt{\binom{n}{i}}b_i$ for all $i$
  and ${(b_i)}_{0\leq i\leq n}$ are i.i.d.\
\item Weyl polynomials, for which $a_i=\frac{1}{\sqrt{i!}}b_i$ for all $i$ and
  ${(b_i)}_{0\leq i\leq n}$ are i.i.d. \
\end{itemize}
Geometrically, the complex number $z$ is a root of $P$ if and only if the
vectors $(1,z,\ldots,z^n)$ and $(\Ol{a_0},\Ol{a_1},\ldots,\Ol{a_n})$ are
orthogonal in $\dC^{n+1}$, and this connects the problem to Littlewood-Offord
type problems \cite{MR0009656} and small balls probabilities. Regarding Kac
polynomials, Kac \cite{MR0007812,MR0009655} has shown in the real Gaussian
case that the asymptotic number of real roots is about $\frac{2}{\pi}\log(n)$
as $n\to\infty$. Kac obtained the same result when the coefficients are
uniformly distributed \cite{MR0030713}. Hammersley \cite{MR0084888} derived an
explicit formula for the $k$-point correlation of the roots of Kac
polynomials. The real roots of Kac polynomials were extensively studied by
Maslova \cite{MR0334327,MR0368136}, Ibragimov and Maslova
\cite{MR0238376,MR0286157,MR0288824,MR0292134}, Logan and Shepp
\cite{MR0234512,MR0234513}, and by Shepp and Farahmand \cite{MR2768528}.
Shparo and Shur \cite{MR0139199} have shown that the empirical measure of the
roots of Kac polynomials with light tailed coefficients tends as $n\to\infty$
to the uniform law one the unit circle $\{z\in\dC:|z|=1\}$ (the arc law).
Further results were obtained by Shepp and Vanderbei \cite{MR1308023},
Zeitouni and Zelditch \cite{MR2738347}, Shiffman and Zelditch
\cite{MR1935565}, and by Bloom and Shiffman \cite{MR2318650}. If the
coefficients are heavy tailed then the limiting law concentrates on the union
of two centered circles, see \cite{gotze-zaporozhets} and references therein.
Regarding Weyl polynomials, various simulations and conjectures have been made
\cite{galligo,emiris-galligo-tsigaridas}. For instance, if $(b_i)_{0\leq i\leq
  n}$ are i.i.d.\ standard Gaussian, it was conjectured that the asymptotic
behavior of the roots of the Weyl polynomials is analogous to the Ginibre
Ensemble. Namely, the empirical distribution of the roots tends as
$n\to\infty$ to the uniform law on the centered disc of the complex plane
(circular law), and moreover, in the real Gaussian case, there are about
$\frac{2}{\pi}\sqrt{n}$ real roots as $n\to\infty$ and their empirical
distribution tends as $n\to\infty$ to a uniform law on an interval, as for the
real Ginibre Ensemble, see Remark \ref{rk:real-ginibre}. The complex Gaussian
case was considered by Leboeuf \cite{MR1700934} and by Peres and Vir\'ag
\cite{MR2231337}, while the real roots of the real Gaussian case were studied
by Schehr and Majumdar \cite{MR2415102}. Beyond the Gaussian case, one may try
to use the companion matrix\footnote{The companion matrix $M$ of
  $Q(z):=c_0+c_1z+\cdots+c_{n-1}z^{n-1}+z^n$ is the $n\times n$ matrix with
  null entries except $M_{i,i+1}=1$ and $M_{n,i}=c_{i-1}$ for every $i$. The
  characteristic polynomial of $M$ is $Q$.} of $P$ and the logarithmic
potential approach. Numerical simulations reveal strange phenomena depending
on the law of the coefficients but we ignore it they are purely numerical.
Note that if the coefficients are all real positive then the roots cannot be
real positive. The heavy tailed case is also of interest (rings?).

\appendix

\section{Invertibility of random matrices with independent entries}
\label{se:ap:inv}

This appendix is devoted to the proof of a general statement (lemma
\ref{le:snRV} below) on the smallest singular value of random matrix models
with independent entries. It follows form lemma \ref{le:snRV} below that if
$X={(X_{ij})}_{1\leq i,j\leq n}$ is a random matrix with i.i.d.\ entries such
as $X_{11}$ is not constant and $\dE(|X_{11}|^\kappa)<\infty$ for some
arbitrarily small real number $\kappa>0$, then for any $\ga>0$ there exists
are real number $\be>0$ such that for any $n\gg1$ and any deterministic matrix
$M\in\cM_n(\dC)$ with $s_1(M)\leq n^\ga$,
\[
\lim_{n\to\infty}\dP(s_n(X+M)\leq n^{-\be})=0.
\]
Both the assumptions and the conclusion are strictly weaker than the result of
Tao and Vu. It is enough for the proof of the circular law in probability and
its heavy tailed analogue.

\begin{lemma}[Smallest singular value of random matrices with independent entries]
  \label{le:snRV}
  If $(X_{ij})_{1 \leq i , j \leq n}$ is a random matrix with independent and
  non-constant entries in $\dC$ and if $a>0$ is a positive real number such
  that
  \[
  b:=\min_{1\leq i,j\leq n}\dP ( |X_{ij} | \leq a ) >0
  \quad\text{and}\quad %
  \sigma^2 := \min_{1\leq i,j\leq n} \VAR\PAR{X_{ij} \IND_{\{|X_{ij}|\leq a \}}} >0,
  \]
  then there exists $c = c(a,b, \si) > 1 $ such that for any $M \in \cM_n ( \dC)$, $n \geq c$,  
  $s \geq 1$, $0<t \leq 1$,
  \[ 
  \dP\PAR{s_n(X+M) \leq \frac{t}{\sqrt{n}}\ ;\ s_1(X+M) \leq s } %
  \leq   c \sqrt{\log  (   c s   ) }   %
  \PAR{ ts^2 +\frac{1}{\sqrt n}}.
  \]
\end{lemma}

The proof of lemma \ref{le:snRV} follows mainly from
\cite{MR2146352,MR2407948}. These works have already been used in the proof of
the circular law, notably in \cite{gotze-tikhomirov-new}. As we shall see, the
term $1/\sqrt{n}$ comes from the rate in the Berry-Esseen Theorem. Following
\cite{MR2146352}, it could probably be improved by using finer results on the
Littlewood-Offord problem \cite{MR2480613}. Note however, that lemma
\ref{le:snRV} is sufficient for proving convergence in probability of spectral
measures.

We emphasize that there is not any moments assumption on the entries in lemma
\ref{le:snRV}. However, (weak) moments assumptions may be used in order to
obtain an upper bound on the quantity $\dP(s_1 ( X + M ) \geq s )$.
Also, the variance (of the truncated variables) $\sigma$ may depend on $n$ :
this allows to deal with sparse matrix models (not considered here).

For the proof of the circular law and its heavy tailed analogue, lemma
\ref{le:snRV} can be used typically with $t =1/ (s^2 \sqrt n)$ and $s =
n^r$ large enough such that with high probability $s_1 ( X + M ) \leq s $. In contrast with the Tao and Vu result, lemma \ref{le:snRV} cannot provide
a summable bound usable with the first Borel-Cantelli lemma due to the
presence of $1/\sqrt{n}$.

Let us give the idea behind the proof of lemma \ref{le:snRV}. A geometric
intuition says that the smallest singular value of a random matrix can be
controlled by the minimum of the distances of each row to the span of the
remaining rows. The distance of a vector to a subspace can be controlled with
the scalar product of the vector with a unit norm vector belonging to the
orthocomplement of the subspace. Also, when the entries of the matrix are
independent, this boils down by conditioning to the control of a small ball
probability involving a linear combination of independent random variables.
The coefficients in this combination are the components of the orthogonal
vector. The asymptotic behavior of this small ball probability depends in turn
on the structure of these coefficients. When the coefficients are well spread,
we expect an asymptotic Gaussian behavior thanks to the central limit theorem,
more precisely its quantitative weighted version called the Berry-Esseen
theorem. We will follow this scheme while keeping the geometric picture in
mind.

The proof of lemma \ref{le:snRV} is divided into two parts which correspond to
a subdivision of the unit sphere $\dS^{n-1}$ of $\dC^n$. Namely, for some real
positive parameters $\delta, \rho > 0$ that will be fixed later, we define the
set of \emph{sparse vectors}
\[
\SPARSE:=
\{x\in\dC^n:\CARD(\SUPP(x))\leq \de n\}
\]
and we split the unit sphere $\dS^{n-1}$ into a set of \emph{compressible
  vectors} and the complementary set of \emph{incompressible vectors} as
follows:
\[
\COMP:= \{x\in\dS^{n-1}:\DIST(x,\SPARSE)\leq\rho\}
\quad\text{and}\quad
\INCOMP:= \dS^{n-1}\setminus\COMP.
\]

We will use the variational formula, for $A \in \cM_n (\dC)$,
\begin{equation}\label{eq:decompsn}
  s_n ( A) = \min_{ x \in \dS^{n-1}} \NRM{A x}_2 %
  = \min\PAR{\min_{x\in\COMP}\NRM{A x}_2,\min_{x\in\INCOMP}\NRM{Ax}_2}.
\end{equation}

\subsubsection*{Compressible vectors}

Our treatment of compressible vectors differs significantly from
\cite{MR2146352,MR2407948} (it gives however a weaker statement). We start with a variation around lemma
\ref{le:concdist}.

\begin{lemma}[Distance of a random vector to a small subspace]
  \label{le:concdist2}
  There exist $\veps, c, \de_0 >0$ such that for all $n\gg1$, all $1\leq i\leq
  n$, any deterministic vector $v\in\dC^n$ and any subspace $H$ of $\dC^n$
  with $1 \leq \DIM(H) \leq \de_0 n$, we have, denoting
  $C:=(X_{1i},\ldots,X_{ni}) + v$,
  \[
  \dP\PAR{\DIST(C,H) \leq \veps \sigma\sqrt {n}} %
  \leq c \exp(-c \si^2 n).
  \]
\end{lemma}

\begin{proof}
  First, from Hoeffding's deviation inequality,
  \[
  \dP\PAR{\sum_{k=1}^n \IND_{\{|X_{ki}|\leq a\}} \leq \frac{nb}{2}} %
    \leq \exp\PAR{-\frac{nb^2}{2}}.
  \] 
  It is thus sufficient to prove the result by conditioning on 
  \[
  E_m := \{ |X_{1i} |\leq a, \ldots ,|X_{mi}|\leq a \} %
  \quad\text{with}\quad %
  m := \CEIL{n b /2}.
  \]
  Let $\dE_m[\,\cdot\,] := \dE [\,\cdot\,|E_m;\cF_m]$ denote the conditional
  expectation given $E_m$ and the filtration $\cF_m$ generated by
  $X_{m+1,i},\ldots,X_{n,i}$. Let $W$ be the subspace spanned by $H$, $v$, and
  the vectors $u := (0,\ldots,0,X_{m+1,i},\ldots,X_{n,i})$ and
  \[
  w:=\PAR{\dE\SBRA{X_{1i}\bigm||X_{1i}|\leq a},\ldots,\dE\SBRA{X_{mi}\bigm||X_{mi}|\leq a},0,\ldots,0}.
  \]
  By construction $\DIM(W) \leq \DIM(H) + 3$ and $W$ is
  $\cF_m$-measurable. We note also that
  \[
  \DIST(C,H) \geq \DIST(C,W) = \DIST(Y, W), 
  \]
  where 
  \[
  Y := \PAR{X_{1i}-\dE\SBRA{X_{1i}\bigm||X_{1i}|\leq a},\ldots,X_{mi}-\dE\SBRA{X_{mi}\bigm||X_{mi}|\leq a},0,\ldots,0} %
  = C - u - v - w.
  \]
  By assumption, for $1 \leq k \leq m$,
  \[
  \dE_m Y_ k = 0 \quad \text{and } \quad \dE_m |Y_ k|^2 \geq \sigma^2. 
  \]
  Let $D = \{z : |z |\leq a\}$. We define the function $f:D^m \to \dR_+$
  by
  \[
  f(x)= \DIST((x_1,\ldots,x_m,0,\ldots,0),W).
  \]
  This function is convex and $1$-Lipschitz. Hence, Talagrand's concentration
  inequality gives
  \[
  \dP_m\PAR{\ABS{\DIST(Y, W)-M_m}\geq t} %
  \leq 4 \exp\PAR{- \frac{t^2}{16a^2}},
  \]
  where $M_m$ is the median of $f$ under $\dP_m$. In particular, 
  \[
  M_m \geq
  \sqrt{ \dE_m \DIST^2 (Y, W) }- c a.
  \]
  Also, if $P$ denotes the orthogonal projection on the orthogonal of $W$, we
  find
  \begin{align*}
    \dE_m   \DIST^2(Y, W) %
    & = \sum_{k=1}^m \dE_m |Y_{k}|^2 P_{kk} \\
    & \geq \sigma^2\PAR{\sum_{k=1}^n P_{kk} - \sum_{k=m+1}^n P_{kk}} \\
    & \geq \sigma^2 \PAR{n-\DIM(H)-3-(n-m)} \\
    & \geq \sigma^2 \PAR{\frac{nb}{2}-\DIM(H)-3}.
  \end{align*}
  The latter, for $n$ large enough, is lower bounded by $c \sigma^2 n$ if
  $\delta_0 = b /4$. 
\end{proof}

Let $0 < \veps < 1$ and $s\geq1$ be as in lemma \ref{le:concdist2}. We set
from now on
\[
\rho = \frac{1}{4} \min(1,\frac{\veps \sigma}{ s  \sqrt { \delta}} ),
\]
(in particular, $\rho \leq 1/4$). The parameter $0 < \de < 1$ is still to be
specified: at this stage, we simply assume that $\de < \de_0$. We note that if
$A \in \cM_n (\dC)$ and $y \in \dC^n$ is such that $\SUPP(y)\subset\pi \subset
\{1, \ldots , n\}$, then we have
\[
\NRM{A y}_2\geq \NRM{y}_2s_n(A_{|\pi}),
\]
where $A_{| \pi}$ is the $n \times |\pi|$ matrix formed by the columns of $A$
selected by $\pi$. We deduce that
\begin{equation}\label{eq:comp}
  \min_{ x \in \COMP} \NRM{Ax}_2 %
  \geq \frac{3}{4} %
  \min_{\pi\subset\{1,\ldots,n\} : |\pi| = \FLOOR{\delta n}}s_n(A_{|\pi}) %
  - \rho s_1 (A). 
\end{equation}
However, by Pythagoras theorem, for any $x \in \dC^{|\pi|}$,
\[
\NRM{A_{|\pi} x}_2^2 = \NRM{\sum_{i \in \pi} x_i C_i}_2^2 %
\geq \max_{i \in \pi}  |x_i|^2 \DIST^2(C_i , H_{i})%
\geq \min_{i \in \pi} \DIST^2(C_i , H_{i}) \frac{1}{ |\pi| } \sum_{i\in \pi} |x_i|^2
\]
where $C_i$ is the $i$-th column of $A$ and $H_i:=\SPAN\{C_j:j \in \pi, j \ne
i\}$. In particular,
\[
s_n ( A_{|\pi}) \geq \min_{i \in \pi}  \DIST(C_i, H_i) / \sqrt {|\pi|}. 
\]
Now, we apply this bound to $A = X +M$. Since $H_{i}$ has dimension at most
$n\delta$ and is independent of $C_{i}$, by lemma \ref{le:concdist2}, the
event that,
\[
\min_{i \in \pi}  \DIST(C_i, H_{i}) \geq \veps  \sigma \sqrt n,
\]
has probability at least $1 - c n  \delta \exp ( - c \si^2  n )$ for $n \gg 1$. Hence
\[
\dP \PAR{ s_n ( (X+M)_{|\pi}) \leq \frac{ \veps \sigma}{\sqrt \delta} } \leq c n  \delta \exp (
- c \si^2  n).
\]
Therefore, using the union bound and our choice of $\rho$, we deduce from \eqref{eq:comp}
\begin{align*}
\dP\PAR{%
  \min_{x\in\COMP}\NRM{(X+M)x}_2\leq \frac{ \veps \sigma}{2 \sqrt \delta} %
  \ ;\ s_1(X+M) \leq s  } %
&\leq c \binom{n}{\FLOOR{\delta n}} n \delta e^{- c \si^2  n} \\
&= cn \delta  e^{n(H(\delta )( 1 + o (1) ) - c \si ^2 )},
\end{align*}
with $H(\delta) := - \de \log \de - (1- \de) \log (1 -\de)$. Therefore, if
$\de$ is chosen small enough so that $H(\delta) < c  \si^2   / 2 $, we have proved that for
some $c_1 >0$,
\begin{equation}\label{eq:invcomp}
  \dP\PAR{%
    \min_{x\in\COMP}\NRM{(X+M)x}_2\leq \frac{ \veps \sigma}{2 \sqrt \delta}  %
    \ ;\ s_1(X+M) \leq s } %
  \leq   \exp(-c_1  n ), 
\end{equation}
 (note that the constant $c_1$ depends on $\si$). From now on, we fix $$\delta = \frac{ c_2 \si^2 }{ | \log \si |},$$ with $c_2$ small enough so that $\delta < \delta_0$ and $H(\delta) < c  \si^2   / 2 $. 
\subsubsection*{Incompressible vectors: invertibility via distance}

We start our treatment of incompressible vectors with two key observations
from \cite{MR2407948}. 

\begin{lemma}[Incompressible vectors are spread]\label{le:incompspread}
  Let $x \in \INCOMP$. There exists a subset $\pi \subset \{1, \ldots, n\}$
  such that $|\pi | \geq \de n / 2$ and for all $i \in \pi$,
  \[
  \frac{\rho}{\sqrt { n}} \leq |x_i|\leq \sqrt{\frac{2}{\delta n}}.
  \]
\end{lemma}

\begin{proof}
  For $\pi \subset \{1, \ldots, n\}$, we denote by $P_\pi$ the orthogonal
  projection on $\SPAN\{e_{i}; i \in \pi\}$. Let $\pi_1 = \{ k : |x_k
  |\leq \sqrt { 2 / (\de n) }\}$ and $\pi_2 = \{ k : |x_k |\geq \rho / \sqrt{
    n} \}$. Since $\NRM{x}_2^2 = 1$, we have
  \[
  | \pi_1 ^c |\leq   \frac{\de n}{2}. 
  \]
  Note also that 
  \[
  \NRM{x - P_{\pi_2} x}_2 = \NRM{P_{\pi_2 ^c }x}_2 \leq \rho.  
  \]
  Hence, the definition of incompressible vectors implies that $|\pi_2| \geq
  \delta n$. We put $\pi = \pi_1 \cap \pi_2$. From what precedes,
  \[
  |\pi | \geq n - |\pi_1 ^c |- |\pi_2^c| %
  \geq n - \frac{\de n}{2} - ( n - \de n) %
  = \frac{\de n}{2}.
  \]
\end{proof}

\begin{lemma}[Invertibility via average distance]\label{le:invIncomp}
  Let $A$ be a random matrix taking its values in $\cM_n(\dC)$, with columns
  $C_1, \ldots, C_n$, and for some arbitrary $1\leq k\leq n$, let $H_k$ be the
  span of all these columns except $C_k$. Then, for any $t \geq 0$,
  \[
  \dP\PAR{%
    \min_{ x \in \INCOMP} \NRM{Ax}_2 %
    \leq \frac{t \rho}{\sqrt n}} %
  \leq \frac{2}{\de n} \sum_{k=1}^n \dP \PAR{\DIST(C_k , H_k) \leq t}.
  \]
\end{lemma}

\begin{proof}
  Let $x \in \dS^{n-1}$, from $A x= \sum_k C_k x_k$, we get
  \[
  \NRM{A x}_2 %
  \geq \max_{ 1 \leq k \leq n} \DIST(Ax , H_k) %
  = \max_{ 1 \leq k \leq n} |x_k |\DIST(C_k , H_k).
  \] 
  Now if $x \in \INCOMP$ and $\pi$ is as in lemma
  \ref{le:incompspread}, we get
  \[
  \NRM{Ax}_2 %
  \geq \frac{\rho}{\sqrt { n}} \max_{ k \in \pi} \DIST(C_k , H_k).
  \]
  Then, the conclusion follows from the fact that for any reals
  $y_1,\ldots,y_n$ and $1 \leq m \leq n$,
  \[
  \IND_{\{\max_{ 1 \leq k \leq m} y_k \leq t\}}  %
  \leq \frac 1 m \sum_{k=1}^m \IND_{\{y_k \leq t\}} %
  \leq \frac 1 m \sum_{k=1}^n \IND_{\{y_k \leq t\}}.
  \]
\end{proof}

The strength of lemma \ref{le:invIncomp} lies in the fact that the control of
$\NRM{Ax}_2$ over all incompressible vectors is done by an average of the
distance between the columns of $A$.

\subsubsection*{Incompressible vectors: small ball probability}

Now, we come back to our matrix $X + M$: let $C$ be the $k$-th column of $X +
M$ and $H$ be the span of all columns but $C$. Our goal in this sub-section is
to establish the bound, for all $t \geq 0$,
\begin{equation}\label{eq:invincomp}
  \dP\PAR{\DIST(C,H) \leq \rho t \,;\, s_1 ( X + M ) \leq s  } %
  \leq \frac{c} { \si } \sqrt{ \frac { |\log\rho | } {  \delta }} \PAR{t+\frac{1}{\sqrt{n}}}. 
\end{equation}
To this end, we also consider a random vector $\zeta$ taking its values in
$\dS^{n-1}\cap H^\perp$, which is independent of $C$. Such a random vector
$\zeta$ is not unique, we just pick one and we call it the \emph{orthogonal
  vector} (to the subspace $H$). We have
\begin{equation}\label{eq:distnormal}
  \DIST(C,H) \geq |\ANG{\zeta,C}| . 
\end{equation}

\begin{lemma}[The random orthogonal vector is
  Incompressible]\label{le:normalincomp}
  For our choice of $\rho,\de$ and $c_1$ as in \eqref{eq:invcomp}, we have
  \[
  \dP\PAR{\zeta \in \COMP\, ; s_1 ( X + M ) \leq s  } %
  \leq \exp ( - c_1 \si ^2 n).
  \]
\end{lemma}

\begin{proof}
  Let $A \in \cM_{n-1,n} (\dC)$ be the matrix obtained from $(X+M)^*$ by
  removing the $k$-th row. Then, by construction : $A \zeta = 0$, $s_1 (
  (X+M)^*) = s_1 (X+M)$, and
  \[
  \min_{ x \in \COMP} \NRM{A x}_2 %
  \geq \min_{ x \in \COMP} \NRM{(X+M)^*x}_2.
  \]
  The left hand side (and thus the right hand side) is zero if
  $\zeta\in\COMP$. In particular,
  {\small
  \[
  \dP\PAR{\zeta \in \COMP\, ; s_1 ( X + M ) \leq s } 
  \leq
  \dP\PAR{\min_{x\in \COMP}\NRM{(X+M)^*x}_2=0\ ;\ s_1((X+M)^*)\leq s }.
  \]
  } 
  It remains to notice that obviously \eqref{eq:invcomp} holds with $(X+M)$
  replaced by $(X+M)^*$. Indeed the statistical assumptions are the same on $X
  +M$ and $(X+M)^*$.
\end{proof}

We have reached now the final preparation step before the use of the
Berry-Esseen theorem. This step consists in the reduction to a case where for
a fixed set of coordinates, both the components of $\zeta$ and the random
variables $X_{ik}+M_{ik}$ are well controlled. Namely, if $\zeta \in \INCOMP$,
let $\pi \subset \{1, \ldots, n\}$ be as in lemma \ref{le:incompspread}
associated to vector $\zeta$. Then conditioned on $\{\zeta \in \INCOMP\}$,
from Hoeffding's deviation inequality, the event that
\[
\sum_{i \in \pi} \IND_{\{|X_{ik} | \leq a\}} %
\geq \frac{|\pi| b}{2} \geq \frac{\de b n} { 4},
\]
has conditional probability at least $1 - \exp (- |\pi | b^2 /  2   ) \geq 1 -
\exp (- c \de n)$ (recall that $\zeta$ hence $\pi$ are independent of $C$). In
summary, using our choice of $\delta, \rho$, by lemma \ref{le:normalincomp} and \eqref{eq:distnormal}, in order to
prove \eqref{eq:invincomp}, it is sufficient to prove that for all $t \geq 0$,
\[
\dP_m\PAR{\ABS{\ANG{\zeta,C}}\leq \rho t} %
\leq  \frac{c} { \si } \sqrt{ \frac { |\log\rho  |} {  \delta }} \PAR{t+\frac{1}{\sqrt{n}}}.  
\]
where $\dP_m(\cdot)= \dP(\cdot|E_m,\cF_m)$ is the conditional probability
given $\cF_m$ the $\sigma$-algebra generated by all variables but
$(X_{1k},\ldots,X_{mk})$, $m = \FLOOR{\de bn/4}$, and the event
\[
E_m = %
\BRA{\frac{\rho}{\sqrt{n}}\leq |\zeta_i|\leq \sqrt{\frac{2}{\de n}}; %
  1 \leq i \leq m} %
\bigcup %
\BRA{|X_{ik}|\leq a; 1 \leq i \leq m}.
\]
We may write
\[
\ANG{\zeta,C} %
= \sum_{i=1}^n \bar \zeta_i \DOT{C,e_i} %
= \sum_{i =1}^m \bar \zeta_i X_{ik} + u,
\]
where $u \in \cF_m$ is independent of $(X_{1k},\ldots,X_{mk}) $. It follows
that
\begin{equation}\label{eq:preprocBE}
  \dP_m\PAR{|\ANG{\zeta,C} | \leq \rho t} %
  \leq  \sup_{z \in \dC , \pi \subset \{1 , \ldots, m\} } %
  \dP_m\PAR{\ABS{\sum_{i \in\pi}\bar\zeta_i(X_{ik}-\dE_m X_{ik})-z}\leq \rho t}.
\end{equation}
The idea, originated from \cite{MR2146352}, is now to use the rate of
convergence given by the Berry-Esseen theorem to upper bound this last
expression.

\begin{lemma}[Small ball probability via Berry-Esseen
  theorem]\label{le:smallballBE}
  There exists $c >0$ such that if $Z_1,\ldots,Z_n$ are independent centered
  complex random variables, then for all $t \geq 0$,
  \[
  \sup_{z \in \dC} \dP\PAR{\ABS{\sum_{i =1 }^n Z_i - z}\leq t} %
  \leq \frac{c t}{ \sqrt{\sum_{i=1}^n \dE ( |Z_i|^2)}} %
  + \frac{c\sum_{i=1}^n\dE(|Z_i|^3)}{\PAR{\sum_{i=1}^n\dE(|Z_i|^2)}^{3/2}}.
  \] 
\end{lemma}

\begin{proof}
  Let $\tau^2 = \sum_{i=1}^n \dE |Z_i|^2$, then either $\sum_{i=1}^n\dE(\Re
  Z_i)^2$ or $ \sum_{i=1}^n\dE(\Im Z_i) ^2$ is larger or equal to $\tau^2/2$.
  Also
  \[
  \dP\PAR{\ABS{\sum_{i =1 }^n Z_i-z} \leq t} %
  \leq \dP\PAR{\ABS{\sum_{i =1 }^n \Re(Z_i) -\Re( z)}\leq t}
  \]
  and similarly with $\Im$. Hence, up to loosing a factor $2$, we can assume
  with loss of generality that the $Z_i$'s are real random variables. Then, if
  $G$ is a real centered Gaussian random variable with variance $\tau^2$,
  Berry-Esseen theorem asserts that
  \[
  \sup_{t \in \dR} \ABS{\dP\PAR{\sum_{i =1 }^n Z_i \leq t}- \dP\PAR{G \leq t}} %
  \leq c_0\tau^{-3/2} \sum_{i=1}^n \dE(|Z_i|^3).
  \]
  In particular, for all $t\geq0$ and $x\in\dR$,
  \[
    \dP\PAR{\ABS{\sum_{i =1 }^n Z_i-x}\leq t} %
    \leq \dP\PAR{|G-x| \leq t}  + 2 c_0  \tau^{-3/2}\sum_{i=1}^n\dE(|Z_i|^3).
  \]
  We conclude by using the fact that $G$ has a density upper bounded by $1 /
  \sqrt { 2 \pi \tau^2} $.
\end{proof}

Define $L = \frac 1 2  \log_2 \frac{2}{\delta \rho^2}$. Note that for our choice of $\rho,\delta$, for some constant $c = c(a,b)$, 
\[
L \leq  c \ABS{\log \rho}.  
\] 
For $1\leq j\leq L$, we define
\[
\pi_j = \BRA{1 \leq i \leq m : %
  \frac{2^{j-1}\rho}{\sqrt{n}} \leq |\zeta_i|\leq \frac{2^{j}\rho}{\sqrt{n}}}.
\]
From the pigeonhole principle, there exists $j$ such that $| \pi_j |\geq m /
L$. We have
\[ 
\sigma^2_j %
= \sum_{i \in \pi_j } |\zeta_i|^2 \dE_m(|X_{ik}-\dE_m(X_{ik})|^2) %
\geq \frac{2^{2j-2} \rho^2 \sigma ^2|\pi_j| }{n},
\]
and,
\[
\sum_{i \in \pi_j } |\zeta_i | ^3 \dE_m(|X_{ik}-\dE_m(X_{ik})|^3) %
\leq \frac{ 2^{ j} a \rho }{\sqrt{ n}} \sigma_j^2.
\]
We deduce by \eqref{eq:preprocBE} and lemma \ref{le:smallballBE} that (by
changing the value of $c$), for all $t \geq 0$,
\begin{align*}
  \dP_m\PAR{\ABS{\ANG{\zeta,C}}\leq \rho t}  
  & \leq\frac{c\rho t}{\sigma_j} +\frac{c2^{ j}a\rho}{\sigma_j\sqrt{n}} \\
  & \leq\frac{ct\sqrt n}{\sigma\sqrt{|\pi_j|}} +\frac{c}{\sigma\sqrt{|\pi_j|}}\\
  & \leq  \frac{c} { \si } \sqrt{ \frac { \ABS{\log\rho } } {  \delta }}  \PAR{t +\frac{1}{\sqrt n}}. 
\end{align*}
The proof of \eqref{eq:invincomp} is complete.

\begin{proof}[Proof of lemma \ref{le:snRV}]
  All ingredient have now been gathered. By lemma \ref{le:invIncomp} and
  \eqref{eq:invincomp} we find, for all $t \geq 0$,
  \[
  \dP\PAR{\min_{x\in\INCOMP}\NRM{(X+m) x}_2 %
    \leq \frac{\rho^2 t }{\sqrt n }\ ;\ s_1 ( X + M ) \leq s \sqrt n} %
  \leq \frac{c} { \si } \sqrt{ \frac { \ABS{\log\rho } } {  \delta^3 }}  \PAR{t+\frac{1}{\sqrt{n}}}.
  \]
  Using our choice of $\rho, \delta$, we obtain for some new constant $c = c ( a, b, \si) >0$,
  \[
  \dP\PAR{%
    \min_{x\in\INCOMP}\NRM{(X+m) x}_2 %
    \leq\frac{t}{\sqrt{n}} %
    \ ;\ s_1(X+M) \leq s } %
  \leq c  \sqrt{ \log  (c s   )   } %
  \PAR{ ts^2 +\frac{1}{\sqrt {n}}}.
  \]
  The desired result follows then by using \eqref{eq:decompsn} and
  \eqref{eq:invcomp}.
\end{proof}

\section*{Acknowledgments}

These notes are the expanded version of lecture notes written for the
France-China summer school held in Changchun, China, July 2011. The authors
are grateful to the organizers Z.-D. Bai, A. Guionnet, and J.-F. Yao for their
invitation, to the local team for their wonderful hospitality, and to the
participants for their feedback. It is also a pleasure to thank our
collaborator Pietro Caputo for his useful remarks on the draft version of
these notes, Terence Tao for his answers regarding the Hermitization lemma,
and Alexander Tikhomirov and Van Vu for our discussions on the circular law
during the Oberwolfach workshop organized by M. Ledoux, M. Rudelson, and G.
Schechtman in May 2011. All numerics and graphics were done using the free
software GNU-Octave and wxMaxima provided by the Debian GNU/Linux operating
system. We are also grateful to the anonymous referee who pointed out a bug in
the first version of the notes, and who helped also to improve the
bibliography and the readability of the final version.


\providecommand{\bysame}{\leavevmode ---\ }
\providecommand{\og}{``}
\providecommand{\fg}{''}
\providecommand{\smfandname}{and}
\providecommand{\smfedsname}{eds.}
\providecommand{\smfedname}{ed.}
\providecommand{\smfmastersthesisname}{Memoir}
\providecommand{\smfphdthesisname}{Thesis}

\vfill


\end{document}